\documentclass[11pt]{article}
\usepackage[headheight=110pt,margin=2cm]{geometry}
\usepackage{amsfonts,amsmath,amssymb,amsthm}
\usepackage{cite}
\usepackage[british]{babel}
\usepackage{dsfont}
\usepackage{enumitem}
\usepackage{fancyhdr}
\usepackage{graphicx}
\usepackage{hyperref} 
\usepackage{mathtools}
\usepackage{titlesec}
\usepackage{multicol}
\usepackage{comment}
\hypersetup{
    colorlinks=false,
    pdfborder={0 0 0},
}

\renewcommand{\Pr}{\mathbb{P}}
\DeclareMathOperator{\Exp}{\mathbb{E}} 
\DeclareMathOperator{\Var}{\mathbb{V}ar}
\DeclareMathOperator{\Cov}{\mathbb{C}ov}

\newcommand{\N}{\mathbb{N}}
\newcommand{\R}{\mathbb{R}}
\newcommand{\RP}{\mathbb{R}_+}

\newcommand{\ZP}{\mathbb{Z}_+}
\newcommand{\Sp}{\mathbb{S}}
\newcommand{\B}{\mathbb{B}}

\newcommand{\calB}{\mathcal{B}}
\newcommand{\calC}{\mathcal{C}}
\newcommand{\calD}{\mathcal{D}}
\newcommand{\calF}{\mathcal{F}}
\newcommand{\calH}{\mathcal{H}}

\newcommand{\calL}{\mathcal{L}} 
\newcommand{\calM}{\mathcal{M}}
\newcommand{\calN}{\mathcal{N}}
\newcommand{\calP}{\mathcal{P}}

\newcommand{\calS}{\mathcal{S}}

\newcommand{\calV}{\mathcal{V}}
\newcommand{\calW}{\mathcal{W}}

\newcommand{\fC}{\mathfrak{C}}
\newcommand{\fK}{\mathfrak{K}}
\newcommand{\fS}{\mathfrak{S}}

\newcommand{\ind}{\mathds{1}}
\newcommand{\1}[1]{{\mathds{1}}{\{#1\}}} 

\newcommand{\tra}{{\scalebox{0.6}{$\top$}}}

\newcommand{\eps}{\varepsilon}

\DeclareMathOperator*{\diam}{diam}
\DeclareMathOperator*{\hull}{hull}
\DeclareMathOperator*{\cl}{cl}
\DeclareMathOperator*{\inte}{int}
\DeclareMathOperator*{\trace}{tr}

\newcommand{\toP}{\overset{\textup{p}}{\longrightarrow}}
\newcommand{\toas}{\overset{\textup{a.s.}}{\longrightarrow}}
\newcommand{\tod}{\overset{\textup{d}}{\longrightarrow}}

\newcommand{\as}{~\text{a.s.}}
\newcommand{\re}{\mathrm{e}}
\newcommand{\rc}{\mathrm{c}}
\newcommand{\ud}{\textup{d}}

\newtheorem{theorem}{Theorem}[section]
\newtheorem{corollary}[theorem]{Corollary}
\newtheorem{proposition}[theorem]{Proposition}
\newtheorem{lemma}[theorem]{Lemma}

\theoremstyle{definition}
\newtheorem{definition}[theorem]{Definition}
\newtheorem{examplex}[theorem]{Example}

\newenvironment{example}
  {\pushQED{\qed}\examplex}
  {\popQED\endexamplex}

\theoremstyle{remark}
\newtheorem{remark}[theorem]{Remark}

\numberwithin{equation}{section}
\numberwithin{figure}{section}
\makeatletter
\def\namedlabel#1#2{\begingroup  
    (#2)%
    \def\@currentlabel{#2}%
    \phantomsection\label{#1}\endgroup
}
\makeatother
\begin{document}
\title{Functional Limit Theorems For Random Walks}
\author{Chak Hei Lo \and James McRedmond \and Clare Wallace}
\date{\today}
\maketitle
\begin{abstract}
We survey some geometrical properties of trajectories of $d$-dimensional random walks via the application of functional limit theorems. We focus on the functional law of large numbers and functional central limit theorem (Donsker's theorem). For the latter, we survey the underlying weak convergence theory, drawing heavily on the exposition of Billingsley \cite{cpm}, but explicitly treat the multidimensional case.

Our two main applications are to the convex hull of a random walk and the centre of mass process associated to a random walk. In particular, we establish the limit sets of the convex hull in the two distinct cases of zero and non-zero drift which provides insight into the diameter, mean width, volume and surface area functionals. For the centre of mass process, we find the limiting processes in both the law of large numbers and central limit theorem domains.
\end{abstract}

\medskip


\medskip


\medskip 

\renewcommand{\abstractname}{Acknowledgements}
\begin{abstract}
\noindent The authors are grateful to Andrew R. Wade for his supervision and guidance on the topics of this report. The second and third authors are grateful for their support from EPSRC studentships (EP/M507854/1).	
\end{abstract}

\pagebreak
\tableofcontents
\pagebreak
\section{Introduction}

The early days of limit theorems in the form of a \lq law of averages\rq~saw slow progress. The first direct study was the theorem of Bernoulli \cite{bernoulli} on the sums of binary random variables, but this was only stated in 1713 over a century after comments of Cardano in his work on dice games \cite{cardano} and 50 years after Halley's treatise of mortality rates\cite{halley} which clearly expressed a knowledge of decreasing errors in large samples. The term \lq law of large numbers\rq~itself, was not coined until one of Poisson's late works on probability theory in 1837 \cite{poisson}, in which the sum of Bernoulli random variables with varying probabilities of success were shown to converge to the sum of the probabilities; the theorem was only rigorously proved by Chebyshev in 1867 \cite{chebyshev}. 

The first description of a law for more general random variables was produced in 1929 by Khinchin\cite{khinchin} and this became the weak law of large numbers. In the succeeding couple of years, Kolmogorov \cite{kolmogorov} improved the result to establish the well known strong law, in particular, he showed that if $\xi_1,\xi_2,\ldots,\xi_n$ are independent, identically distributed (i.i.d.) random variables with finite expectation $\mu$ we have

\begin{align*}
\frac{1}{n} \left(\sum_{i=1}^n \xi_i-n\mu\right) \toas 0, \text{ as } n \to \infty.
\end{align*} 

All of these laws of large numbers capture the same idea that was understood by Cardano $500$ years ago. The next natural question which Cardano could only associate to \lq luck\rq~is, how large are the errors likely to be for a given number of trials? The Lindeberg--L\'{e}vy central limit theorem, see \cite[pp.~247--249]{fischer}, answers this question when the random variables are i.i.d. with finite variance $\sigma^2$:
\begin{align}
\label{eqn:clt}
\frac{1}{\sqrt{n}} (S_n - n \mu) \tod \calN( 0, \sigma^2 ), \text{ as } n \to \infty,
\end{align}
where $\calN( 0, \sigma^2 )$ is the normal distribution with mean $0$ and variance $\sigma^2$.

These early results refer only to the sums of random variables because this was the quantity of interest in the contexts of long run profit in gambling games, and errors when sampling large amounts of data. However, if the results are considered as within the topic of random walks there are some further natural questions that arise. If the law of large numbers says that after $n$ steps, the position of the walk is approximately $n \mu$ from where you started, can a stronger claim be made, namely that you move towards that point at a constant rate? If so, what does the deviation from this direct path look like? Do the results extend to higher dimensions, or trajectories with discontinuities? The answer to all of these questions is yes, with the fact that the trajectory convergence is equivalent to the strong law of large numbers found in \cite{glynn}, and the higher dimensional and discontinuous path results can be found in Whitt's book \cite{whitt2002stochastic}. Many of these results build on the theory presented in the book of Billingsley \cite{cpm}, and all together they are the subject of the first half of this survey.

To formally describe these problems, we introduce a couple of definitions and assumptions, starting with our description of a random walk:
\begin{description}
\item
[\namedlabel{ass:walk}{\textbf{W}$_{\mu}$}]
Let $d \in \N$, and suppose that $\xi, \xi_1, \xi_2, \ldots$ are i.i.d.~random variables in $\R^d$ with $\Exp\|\xi\| < \infty$ and $\Exp\xi=\mu$. 
The random walk $(S_n,n\in \ZP)$ is the sequence of partial sums $S_n := \sum_{i=1}^n \xi_i$ with $S_0 := 0$.
\end{description}
Often this description is too broad and just as the Lindeberg-L\'evy central limit theorem requires finite variance, we will use some restriction on the higher moments of the underlying process, in the form of one of the following two conditions:
\begin{description}
\item
[\namedlabel{ass:Sigma}{\textbf{V}}]
Suppose that $\Exp[\|\xi\|^2] < \infty$ and write $\Sigma := \Exp[(\xi-\mu)(\xi-\mu)^\tra]$.
Here $\Sigma$ is a nonnegative-definite, symmetric $d$ by $d$ matrix; 
we write $\sigma^2 := \trace \Sigma = \Exp [ \| \xi - \mu \|^2 ]$,
\end{description}

\begin{description}
\item
[\namedlabel{ass:momentp}{\textbf{M}$_p$}]
Suppose that $\Exp[\|\xi\|^p]<\infty$.
\end{description}

In order to answer the question of progress towards $n\mu$, it is necessary to define the trajectory of the walk. First, consider the discrete jump process of the partial sums with each time step rescaled by $1/n$, so the partial sums are indexed by times in the the interval $[0,1]$, in fact they are at the times $\frac{k}{n}$ with $k = 0, 1, \ldots, n$. However, we wish to consider a continuous-time trajectory, so we have two choices on how to fill in the gaps. Either we can say the walk moves linearly from each partial sum to the next, in which case the trajectory is
\[X_n(t) := n^{-1} \left(S_{\lfloor nt \rfloor} + (nt - \lfloor nt \rfloor) \xi_{\lfloor nt \rfloor + 1}\right),\]
or we can consider the trajectory where the walk \lq stays still\rq~and makes small jumps when it reaches the next time indexing a new partial sum, in which case we have
\[X'_n(t) := \frac{1}{n} S_{\lfloor nt \rfloor}.\] 
Using this notation, the trajectory equivalent of Kolmogorov's strong law is as expected, see for example \cite[p.~20]{whitt2002stochastic},
\[X_n(t) \toas \mu t, ~{\rm as~} n \to \infty,\] 
in the sense of convergence of continuous functions of $t \in [0,1]$ (a formal definition is given later). It is not unreasonable to expect a similar result for $X'_n(t)$, but it is clear that this will require some care in dealing with the discontinuities. We discuss these results in Section \ref{FLLN} and show their use by applying the theory to the maximum functional.

As for developing the central limit theorem, the invariance principle proved by Donsker in 1951 \cite{donsker} which can be generalised to $d$-dimensions, see for example \cite[Theorem~4.3.5]{whitt2002stochastic}, states that the appropriately-scaled trajectories converge weakly to a continuous-time process with normally-distributed increments --- namely the $d$-dimensional Brownian motion $b_d$. In particular, if we have the covariance matrix $\Sigma$ as defined at \eqref{ass:Sigma}, then 
\begin{align*}
n^{-1/2} \left( S_{\lfloor nt \rfloor} - n \mu t \right) \Rightarrow \Sigma^{1/2 } b_d(t),
\mathrm{\ as\ } n \to \infty,
\end{align*}	
in the sense of weak convergence of functions (again, formal definitions come later). This is the subject of Section \ref{FCLT}, where we provide a proof of the result which differs from that in \cite{whitt2002stochastic} following more closely the work of Billingsley\cite{cpm} using Etemadi's inequality \cite[Theorem~22.5]{billpm} which we extend to $d$-dimensions instead of using Tychonoff's theorem. This section also includes the mapping theorem for weak convergence. Here we also exhibit the usefulness of these results by applying them to the maximum functional again, and also by generalising the classical arcsine law, see for example \cite[p.~93]{feller1}.

With all the theory now presented, we introduce some set theory notation in Section \ref{sec:set} and consider the partial sums as a random point set and consider the set's convergence to the path from $0$ to $n \mu$. This theory can stand alone and produce some results on the diameter of these points, but will also combine with the limit theorems in Section \ref{sec:hulls} where we consider the convex hull of the random walk.

Many results on the convex hull already exist in the literature but with the theory described above, we can extend many of these results to higher dimensions and consider some new functionals.

Finally, in Section~\ref{sec:COM}, we apply the mapping theorem and limit theorems to the centre of mass process of a random walk in $\R^d$, $G_n:= \frac{1}{n}\sum_{i=1}^n S_i$. For comparison, we establish the equivalent of Kolmogorov's, and Lindeberg's and L\'evy's results, before using the theory to find the trajectory equivalent of the strong law of large numbers,
\[\frac{1}{n}G_{\lfloor nt \rfloor} \toas \frac{\mu t}{2} {~\rm as~} n \to \infty.\] 

In the special case of $\mu=0$, we go further and consider the central limit theorem with the scaling $1/\sqrt{n}$ with the usual extra assumption \eqref{ass:Sigma}. We get 
\[\frac{1}{\sqrt{n}}\left(G_{\lfloor n t \rfloor}\right)_{t\in [0,1]} \Rightarrow \mathcal{GP}(0,K) {\rm ~as~} n\rightarrow \infty,\]
where $\mathcal{GP}(0,K)$ is a Gaussian process with mean $0$ and some symmetric covariance matrix $K$ which depends on $\Sigma$. 

The appendix contains the proof of Etemadi's inequality in $d$-dimensions.

\subsection{Further development and other applications}
For many of the results in this paper the second moment assumption on the increments of the random walk is required. It is possible to relax this assumption and retrieve more general theorems in the case where the increments are in the domain of attraction of a stable law, in which case we would see L\'{e}vy processes as the limit process for the functional central limit theorem, instead of Brownian motion. Some references relating to this area of study are \cite{press, samorodnitsky1994stable,whitt2002stochastic,KampfLastMolchanov}.

Of course there are also numerous further functionals of the random walk processes which may be relevant to different applications, for the convex hull one could consider the number of faces, for example. It would also be of interest to convolve our two main functionals and study the convex hull of the centre of mass process.

\section{Notation and preliminaries}

\subsection{Basic notation}

We use the notation $(\Omega, \calF, P)$ to denote a probability triple, where $\Omega$ is the sample space, the $\sigma$-algebra $\calF$ is the collection of all events,
and $P$ is a probability measure. On occasion it will be necessary to use the subscript $P_n$ to represent a sequence of probability measures and where the measure is specified, either explicitly or via an associated random variable, we will use $\Pr$ or $\Pr_n$. Thus, we say the distribution of a random variable $X$ taking values in a measurable space $(S, \calS)$
is determined by $\Pr ( X \in B)$ for all $B \in \calS$.
This survey is concerned with various types of convergence for random variables. We use the notation $\toas$ for almost-sure convergence,
$\toP$ for convergence in probability, $\tod$ for convergence in distribution, and $\Rightarrow$ for  weak convergence; we give definitions
later.

Given $x \in \R^d$ we write $x = (x_1, \ldots, x_d)^\tra$ for the vector in Cartesian coordinates; for definiteness, vectors are viewed as column vectors throughout.
We write $\| \, \cdot \, \|$ for the Euclidean norm on $\R^d$, so that $\|x\| := ( x_1^2 + \cdots + x_d^2 )^{1/2}$.
For $x, y \in \R^d$ we write the Euclidean distance between points $x$ and $y$ as $\rho_E ( x, y) := \| x -y \|$.
For $x \in \R^d \setminus \{ 0 \}$ we write $\hat x := x / \| x \|$ for the unit vector in the $x$ direction; it is convenient to set $\hat{0}:=0$.
The unit sphere in $\R^d$ is denoted by $\Sp^{d-1} := \{ x \in \R^d : \| x \| =1 \}$ and for the unit ball in $\R^d$ we write $\B^d:= \{ x \in \R^d : \| x \| \leq 1 \}$.

Let $(S,\rho)$ be a metric space. 
For $A \subseteq S$ and $x \in S$ we set $\rho ( x, A ) := \inf_{y \in A} \rho (x, y)$,
and for $A, B \subseteq S$ we set $\rho (A, B) := \inf_{x \in A} \inf_{y \in B} \rho (x,y)$. In the case where $S = \R^d$, we 
use $\rho_E$ for $\rho$ in the same way. 

We denote the $d$-dimensional normal distribution by $\calN(\mu,\Sigma)$, where $\mu$ is the $d$-dimensional mean vector, and $\Sigma$ is the $d \times d$ covariance matrix;
$\Sigma$ is nonnegative-definite and symmetric. We permit the notation $\calN (\mu , 0)$, where $0$ is the matrix of all $0$s, to stand for the degenerate
normal distribution concentrated at $\mu$. 
We write $I_d$ for the $d$-dimensional identity matrix.
The nonnegative-definite symmetric
matrix $\Sigma$ has a (unique) nonnegative-definite symmetric square-root $\Sigma^{1/2}$ satisfying $( \Sigma^{1/2} )^2 = \Sigma$.
The matrix $\Sigma^{1/2}$ induces the linear map $x \mapsto \Sigma^{1/2} x$ on $\R^d$.
We write $b_d = (b_d(t), t \in \RP)$ for standard $d$-dimensional Brownian motion started at $b_d(0) = 0$;
we work in a probability space on which $b_d$ is continuous. The process $\Sigma^{1/2} b_d$ is \emph{correlated} Brownian
motion with covariance matrix $\Sigma$. If $d=1$ we write simply $b$ for $b_1$.

For $A \subseteq \R^d$, we denote the closure of $A$ by $\cl A$, the complement of $A$ by $A^\rc := \R^d \setminus A$, and 
the  boundary of $A$ by $\partial A := \cl A \cap \cl A^\rc$.
The interior of $A \subseteq \R^d$ is  $\inte A := A \setminus \partial A$ and we use $\ind_A(x)$ to denote the indicator function of a set, that is 
\[\ind_A(x)=\begin{cases}1 ~{\rm if~} x\in A;\\ 0 {\rm~otherwise.~}\end{cases}\]
 We write $\calB_d$ for the Borel $\sigma$-algebra on $\R^d$.
For a measurable $A \subseteq \R^d$ we write $\mu_d (A)$ for the $d$-dimensional Lebesgue measure of $A$.

For the set of points at distance of at most $\eps$ from $A$ we use $A^{\eps}=\{ x \in \R^d : \rho_E (x,A) \leq \eps \}$. 
The projection of $A$ on to the space perpendicular to $u \in \Sp^{d-1}$ will be denoted by $A|u^{\perp}$.

For $x \in \R$ we set $x^+ := \max \{ x, 0\}$, $x^- := \max \{ -x, 0\}$, so that $x = x^+ - x^-$ and 
\[\mathrm{sgn}(x)=\begin{cases}1 ~{\rm if~} x>0;\\ 0 {\rm~if~} x=0;\\ -1 {\rm~if~} x<0.\end{cases}\]
For $x, y \in \R$ we write $x \wedge y := \min \{ x, y\}$ and $x \vee y := \max \{ x, y\}$. 
Given functions $f$ and $g$, the function $f \circ g$ is defined by
$(f \circ g) (x) = f (g (x))$.

\subsection{Spaces of trajectories}
\label{subsec:CandD}

Since our aim is to consider the limiting behaviour of trajectories of random walks we first discuss the spaces that these trajectories and their scaling limits inhabit. 
Throughout this paper our trajectories will be indexed over the interval $[0,1]$.
Let $\calM^d := \calM^d [0,1]$ denote the set of all bounded measurable $f : [0,1] \to \R^d$; we call elements
of $\calM^d$ \emph{trajectories}. Let $\calM_0^d:= \{ f \in \calM^d : f(0) =0 \}$.
For $t \in [0,1]$ and $f \in \calM^d$ we define the \emph{interval image} $f[0,t] := \{ f(x) : x \in [0,t]\}$.
For $f \in \calM^d$ we write $\| f \|_\infty := \sup_{0 \leq t \leq 1} \| f(t) \|$ for the supremum norm of $f$.
We endow $\calM^d$ with the supremum metric
\begin{equation}
\label{eqn:supnorm} 
\rho_{\infty}(f,g) := \|f-g\|_{\infty} := \sup_{0\leq t\leq 1} \|f(t)-g(t)\|, \text{ for } f, g \in \calM^d. 
\end{equation}
For $d=1$ we write simply $\calM := \calM^1$.
It will also be occasionally useful to consider
 the canonical projection at time $t \in [0,1]$, $\pi_t : \calM^d \to \R^d$, defined 
as $\pi_t f = f(t)$ for $f \in \calM^d$.

Our limit theorems are stated in terms of convergence of elements of a metric space.
We will primarily be interested in trajectories that are either continuous, or have at most countably many jump discontinuities. We will restrict our attention to the corresponding
 subspaces of $\calM^d$. For $f \in \calM^d$, we write $D_f \subseteq [0,1]$ for the set of discontinuities of $f$.

The set $\calC^d:=\calC^d[0,1]$ is the set of continuous $d$-dimensional functions on the unit interval (the set of $f:[0,1]\rightarrow \R^d$ such that $\lim_{x\rightarrow c}f(x)=f(c)$ for all $c\in [0,1]$). The space $(\calC^d,\rho_\infty)$ is the corresponding metric space with the distance between two elements defined by~\eqref{eqn:supnorm}. Note that functions in $\calC^d$ are bounded.
 Let $\calC_0^d:= \{ f \in \calC^d : f(0) =0 \}$.
In the case $d=1$, we write simply $\calC := \calC^1$.

We consider also the set $\calD^d := \calD^d[0,1]$, the set of right-continuous $d$-dimensional functions with left-hand limits on the unit interval, that is, 
\begin{enumerate}
	\item For $0\leq t<1, f(t+)=\lim_{s\downarrow t}f(s)$ exists and $f(t+)=f(t)$.
	\item For $0<t\leq 1, f(t-)=\lim_{s\uparrow t}f(s)$ exists.
\end{enumerate}
Functions in $\calD^d$ are bounded, and have (at most) countably many discontinuities of the first type (jump discontinuities): see \cite[pp.~121--122]{cpm}.
 Let $\calD_0^d:= \{ f \in \calD^d : f(0) =0 \}$.
In the case $d=1$, we write simply $\calD := \calD^1$.
The supremum metric
\eqref{eqn:supnorm} is well-defined on $\calD^d$, and in some cases we will consider the metric space $(\calD^d, \rho_\infty)$. However,
the following example motivates the consideration of an alternative metric.
\begin{example}\label{ex:threefunctions}
Consider the following three functions,

\begin{minipage}{0.4\textwidth}
\begin{align*}
f(t)&=\begin{cases}
	1 & \mathrm{for\ }t\in[0,1/2);\\
	0 & \mathrm{for\ }t\in[1/2,1];
	\end{cases}\\
g(t)&=\begin{cases}
	0.8 & \mathrm{for\ }t\in[0,1/2);\\
	0.2 & \mathrm{for\ }t \in[1/2,1];
	\end{cases}\\
h(t)&=
	\begin{cases}
	0.95 & \mathrm{for\ } t\in [0,0.49);\\
	0.05 & \mathrm{for\ }t\in [0.49,1].
	\end{cases}
\end{align*}
\end{minipage}
\begin{minipage}{0.48\textwidth}
\centering
\includegraphics[scale=0.28]{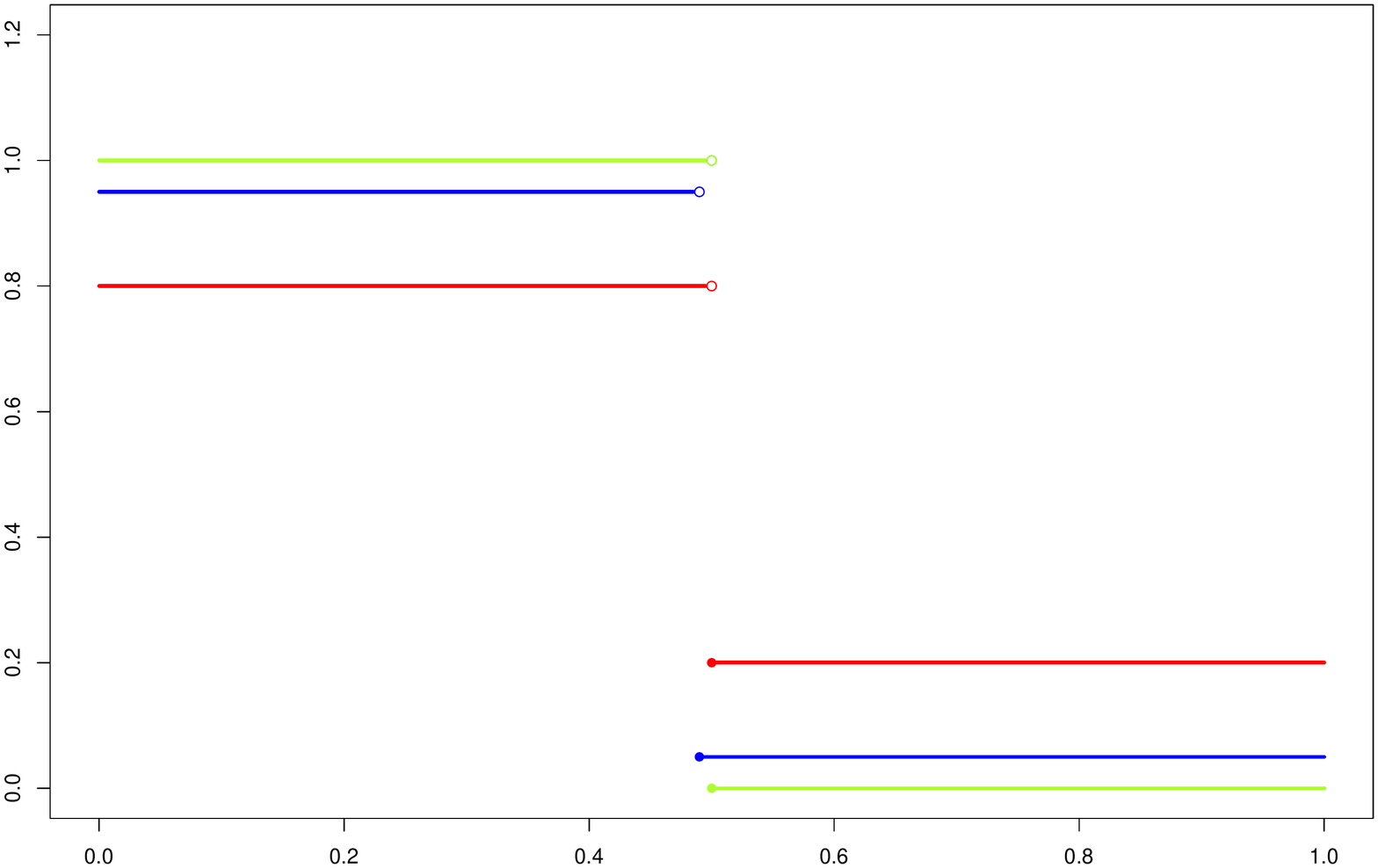}
\end{minipage}

Taking an overview of the plot, it seems reasonable to suggest that the blue function, $h(t)$, is \lq closer\rq~to the light-green function, $f(t)$, than the red function, $g(t)$, is to the light-green function. However, if we consider the supremum metric, 
we find that $\rho_{\infty}(f,h)=0.95$ whilst $\rho_{\infty}(f,g)=0.2$. 
This is due to the slightly earlier jump at $t=0.49$ for $h$, which, for a small interval of $t$, takes the function to the larger Euclidean distance of $0.95$ from $f$. So for processes with jumps, we may wish to consider a different measure of distance.
\end{example}

Our candidate for a more suitable metric on $\calD^d$ is the \emph{Skorokhod metric}, which allows some small perturbations in the $t$ direction to be considered when we measure the distance between two functions. Roughly speaking, we define the distance as the smallest perturbation needed in either the space or time direction to map the functions onto each other.

Formally (see \cite[p.~123]{pollard}), let $\Lambda$ denote the class of strictly increasing, continuous mappings of $[0,1]$ onto itself. If $\lambda \in \Lambda$, then $\lambda(0)=0$ and $\lambda(1)=1$,
and the inverse $\lambda^{-1}$ is also in $\Lambda$. 
For $f$ and $g$ in $\calM^d$, define $\rho_S(f,g)$ to be the infimum of those positive $\varepsilon$ for which there exists in $\Lambda$ a $\lambda$ satisfying
\begin{equation}
\label{eqn:Skorokhod1}
\sup_{0\leq t\leq 1}|\lambda (t)-t|=\sup_{0\leq t\leq 1}|t-\lambda^{-1}(t)|<\eps
\end{equation}
and
\begin{equation}
\label{eqn:Skorokhod2}
\sup_{0\leq t\leq 1}\|f(t)-g(\lambda(t))\|=\sup_{0\leq t\leq 1}\|f(\lambda^{-1}(t))-g(t)\|<\eps.
\end{equation}
To express this in more compact form, let $I$ be the identity map on $[0,1]$; then
\begin{equation}
\label{eq:skor-def}
\rho_S(f,g) \coloneqq \inf_{\lambda\in \Lambda}\left\{ \left\|\lambda-I\right\|_{\infty}\vee\left\|f-g\circ\lambda\right\|_{\infty}\right\} .
\end{equation}
\begin{example}
Consider the functions $f(t)$, $g(t)$ and $h(t)$ from Example \ref{ex:threefunctions}. The distance $\rho_S(f,g)=0.2$ because there is no perturbation of the time which would decrease the Euclidean distance between $f$ and $g$. However, when we consider $f$ and $h$, we could define $$\lambda(t) \coloneqq \begin{cases}
\frac{49}{50}t & \mathrm{for\ }t \in [0,1/2)\\
\frac{51}{50}t-\frac{1}{50} & \mathrm{for\ }t \in [1/2,1].
\end{cases}$$
It turns out that this $\lambda$ is optimal giving $\rho_S(f,h)=0.05$ because equation \eqref{eqn:Skorokhod1} gives us a lower bound of $\varepsilon=0.01$ attained when $t=0.5$ and $\lambda(t)=0.49$, and equation \eqref{eqn:Skorokhod2} gives the lower bound of $\varepsilon=0.05$, which is attained at any $t\in [0,1]$. 
\end{example}

We state the following simple fact so we can refer to it later.
\begin{lemma}
\label{lem:skor}
For any $f, g \in \calM^d$ we have $\rho_S (f,g) \leq \rho_\infty (f,g)$.
\end{lemma}
\begin{proof}
The infimum in~\eqref{eq:skor-def} is bounded above by the value at $\lambda = I$.
\end{proof}

For certain technical reasons, it is sometimes more convenient to work with an alternative metric on $\calD^d$, first described by Kolmogorov \cite{kolmogorov1956}, denoted $\rho_S^\circ$. 
For $\lambda \in \Lambda$, we set 
\begin{equation*}
\label{lambdacirc}
\|\lambda \|^\circ := \sup_{s<t} \left\lvert \log \frac{\lambda (t) - \lambda (s)}{t - s}\right\rvert.
\end{equation*}
If the slope of the chord between $(s, \lambda(s))$ and $(t, \lambda(t))$ is always close to 1, then $\|\lambda\|^\circ$ is close to 0. 
We then define
\[
\rho_S^\circ(f,g) \coloneqq \inf_{\lambda \in \Lambda} \left\lbrace \left\| \lambda \right\|^\circ \vee \left\| f - g\circ \lambda\right\|_{\infty} \right\rbrace.
\]

We also make the following technical observation about $\|\lambda\|^\circ$ which will be required in a couple of proofs.
\begin{lemma}
	\label{lem:estlam}
	Let $\lambda \in \Lambda$. Define $c ( \lambda ) := \max \{ \re^{\|\lambda\|^\circ}-1, 1- \re^{-\|\lambda\|^\circ}\}$. Then we have
	\begin{equation}
	\label{s3}
	|\lambda(t)-t| \le t c  (\lambda)  , \text{ for all } t \in [0,1]; 
	\end{equation}	
	and
	\begin{equation} 
	\label{s1}
	| \lambda' (t) - 1 | \leq c ( \lambda ) ,\text{ almost everywhere on } t \in (0,1) .
	\end{equation}
\end{lemma}

\begin{proof}
	From the definition of $\|\lambda\|^\circ$, we have 
	that for any $t \in [0,1)$ and $h>0$ sufficiently small,
	\[ 
	\log \left| \frac{\lambda(t+h)-\lambda(t)}{h} \right| \le \|\lambda\|^\circ
	\]
	so that
	\begin{equation}
	\label{lambdabound}
	\re^{-\|\lambda\|^\circ} \le  \frac{\lambda(t+h)-\lambda(t)}{h}  \le \re^{\|\lambda\|^\circ}.
	\end{equation}	
	By Lebesgue's theorem on the differentiability of monotone functions, see \cite[p.~321]{kolmogorov2012}, $\lambda' (t)$ exists almost everywhere on $t \in (0,1)$, and when it does exist, we have from \eqref{lambdabound} that
	\[
	\re^{-\|\lambda\|^\circ} \le  \lambda'(t)  \le \re^{\|\lambda\|^\circ} . 
	\]
	Hence we see that \eqref{s1} holds as required.
	For the first assertion, since $\lambda (0 ) = 0$, another application of the definition of $\| \lambda\|^\circ$ shows that
	$\log | \lambda (t) / t |   \leq \|\lambda\|^\circ$ for all $t \in (0,1)$, so that
	$|\lambda(t)| \le t \re^{\|\lambda\|^\circ}$, hence
	\[
	t \re^{-\|\lambda\|^\circ} \le  \lambda(t)  \le t\re^{\|\lambda\|^\circ}, \text{ for all } t \in [0,1].  
	\]
	It follows that\[ 
	-t\left(1-\re^{-\|\lambda\|^\circ}\right) \le  \lambda(t)-t  \le t \left(\re^{\|\lambda\|^\circ}-1\right), \]
	and so we get \eqref{s3} as required. Hence we completed the proof.
\end{proof}

One important result about the two metrics $\rho_S$ and $\rho_S^\circ$ concerns their equivalence.
\begin{proposition}\cite[Theorem~7]{kolmogorov1956}
	\label{prop:equivmetrics}
	 The metrics $\rho_S^\circ$ and $\rho_S$ are equivalent. That is, for a sequence of functions $f,f_1,f_2,\ldots$ on $\calD^d$, $\rho_S(f_n,f)\rightarrow 0$ as $n \rightarrow \infty$ if and only if $\rho_S^\circ(f_n,f) \rightarrow 0$ as $n\rightarrow \infty$.
\end{proposition} 

The fact that the metrics are equivalent means we can use either one to prove continuity of a functional on $\calD^d$ and the result will hold for the other, we will use the metric which is simplest for each application. Likewise, almost sure statements using one metric carry over to the other. Note also that as equivalent metrics, $\rho_S$ and $\rho_S^\circ$ generate the same topology (open sets) on $\calD^d$, and hence also the same Borel sets. One reason we may wish to consider $\rho_S^\circ$ is that it has the advantage that it provides a metric with which the space $\calD^d$ is a complete metric space. 

\begin{theorem}
\label{thm:Cdsep}
The space $\calC^d$ is separable and complete under $\rho_{\infty}$.
\end{theorem}

\begin{theorem}
\label{thm:Ddsep}
The space $\calD^d$ is separable under $\rho_S$ and $\rho^\circ_{S}$, and complete under $\rho^\circ_{S}$.
\end{theorem}
 
The one-dimensional case of Theorem~\ref{thm:Cdsep} is discussed at \cite[p.~11]{cpm}. The separability extends to higher dimensions by, for example \cite[\S 4A2Q]{fremlin}. This result also implies that there exists some measure for which the space is complete, and it is a simple exercise to see that every one-dimensional projection of a Cauchy sequence under $\rho_\infty$ in $d$-dimensions is also a Cauchy sequence and therefore has a limit in the product space. 

As mentioned above, the one-dimensional case of Theorem~\ref{thm:Ddsep} was proven by Kolmogorov in \cite{kolmogorov1956}, but is also discussed at \cite[Theorem~12.2]{cpm}. The separability for higher dimensions extends as in the continuous case, using \cite[\S 4A2Q]{fremlin} and the completeness of the space under the measure $\rho_S^\circ$ also follows with a similar simple calculation.

\subsection{Modulus of continuity}
\label{sec:modulus}

For $f \in \calC^d$, the associated modulus of continuity is defined by
\begin{equation*}
\label{eqn:Cmodulusofcont}
w_f(\delta):=\sup_{|s-t|<\delta} \|f(s)-f(t)\|, \text{ for } 0<\delta\leq 1. 
\end{equation*}
In $\calD^d$, the analogous concept is a little more involved (see \cite[p.~122]{cpm}). 
A set $\{t_i : 0 \leq i \leq v\}$ which has $0=t_0<t_1<\cdots< t_v=1$ is called 
$\delta$-sparse if it also satisfies $\min_{1\leq i\leq v}(t_i-t_{i-1})>\delta$. Then define, for $0<\delta \leq 1$,
\begin{align}
\label{eqn:Dmodulusofcont}
w'_f(\delta) & := 
 \inf_{\{t_i\}} \max_{1 \leq i \leq v} \sup_{t,s \in [t_{i-1}, t_i)} \|f(t) - f(s)\|,
\end{align}
where the infimum extends over all $\delta$-sparse sets $\{t_i\}$.

\subsection{Random elements of metric spaces}
\label{sec:MMeasSpaces}

Let $(S, \rho)$ denote a metric space and let $\calS$ denote the Borel $\sigma$-algebra,
that is, the $\sigma$-algebra on $S$ generated by the open sets. 
On probability space $(\Omega, \calF, \Pr)$, a random variable $X$ taking values in a measurable space $(S, \calS)$
is a mapping $X : \Omega \to S$ that is measurable, i.e., $X^{-1}(B) \in \calF$ for any $B \in \calS$.
Note that $\Pr$ induces a measure on $(S, \calS)$ via $\Pr ( X \in \, \cdot \,)$.
The triple $(S,\calS,\rho)$ we call a \emph{metric measure space},  we understand
that $\calS$ is always the Borel $\sigma$-algebra, and
we talk about $X$ as being a 
random element of $(S,\calS,\rho)$.
In what follows we discuss convergence of sequences of such random elements, 
starting with the definition of almost-sure convergence in the next section.

\section{Functional laws of large numbers}
\label{FLLN}
\subsection{Almost-sure convergence and the strong law}
 
\begin{definition}[Almost-sure convergence]
Let $X, X_1, X_2, \ldots$ be random variables on the probability space
$(\Omega, \calF, \Pr)$ taking values in the metric measure space $(S, \calS, \rho)$.
We write $X_n \toas X$ if $\rho ( X_n, X) \toas 0$, i.e., if
\[ \Pr \left( \left\{ \omega \in \Omega : \lim_{n \to \infty} \rho ( X_n (\omega) , X(\omega) ) = 0 \right\} \right) = 1 .\]
\end{definition}

Given two metric measure spaces $(S, \calS, \rho)$ and $(S', \calS', \rho')$
and a measurable function $h : S \to S'$, the set $D_h$ of discontinuities
of $h$ satisfies $D_h \in \calS$: see \cite[p.~243]{cpm}, and hence $\Pr(X\in D_h)$ is well defined.
Using this, the following result gives conditions under which almost-sure convergence
is preserved under mappings.

\begin{theorem}[Mapping theorem for almost-sure convergence]
\label{thm:mapping}
Let $X, X_1, X_2, \ldots$ be random variables on the probability space
$(\Omega, \calF, \Pr)$ taking values in the metric measure space $(S, \calS, \rho)$. Let $(S',\calS',\rho')$
be another metric measure space, and let $h: (S, \calS, \rho) \to (S', \calS', \rho')$ be measurable.
If $X_n \toas X$ and $\Pr ( X \in D_h) = 0$, then $h(X_n) \toas h(X)$.
\end{theorem}
\begin{proof}
For any $\omega$ such that $h$ is continuous at $X(\omega)$, 
$X_n(\omega) \to X(\omega)$ implies that $h(X_n(\omega)) \to h(X(\omega))$.
Then 
	\begin{align*}
	\Pr(\{\omega \in \Omega : \rho' ( h(X_n(\omega)) , h(X(\omega)) ) \to 0 \mathrm{\ as\ } n \to \infty\}) 
	&\geq \Pr(\{\omega \in \Omega : h \mathrm{\ is\  continuous\  at\ } X(\omega)\})\\
	&= \Pr(X \in D_h^\rc) = 1,
	\end{align*}
	so that $h(X_n)	\toas h(X)$. 
\end{proof}

Since the $\sigma$-algebra $\calS$ is always the Borel $\sigma$-algebra,
we often omit this and talk about almost-sure convergence on $(S, \rho)$ instead of $(S, \calS, \rho)$.

The formal result of the classical strong law of large numbers for $d$-dimensional random walks is as follows, see for example \cite[Theorem~2.4.1]{durrett}.
\begin{theorem}[Strong law of large numbers]
\label{thm:SLLN}
Consider the random walk defined at~\eqref{ass:walk}. Then, $n^{-1} S_n \toas \mu$ 
in the sense of almost-sure convergence on $(\R^d,\rho_E)$.
\end{theorem}
 
\subsection{The functional law of large numbers}

Consider the random walk on $\R^d$ as defined at \eqref{ass:walk}.
We construct a trajectory with time scaled to $[0,1]$ and space scaled 
as in the strong law of large numbers; the trajectory is 
an element of $\calM^d$ that takes values $n^{-1} S_k$ at times $t = k/n$ for $k \in \{0,1,\ldots, n\}$.
There are two standard ways to construct this trajectory: either use linear interpolation between the partial sums at the times $k/n$ and $(k+1)/n$ or make it piecewise constant
(constant over intervals $[k/n,(k+1)/n)$).
The first approach  ensures a continuous trajectory and so gives an element of $\calC^d$; the second
introduces $n$ jump discontinuities to each path, so it yields an element of $\calD^d$. 
The formal definitions are as follows. Define for $n \in \N$ and all $t \in [0,1]$,
\begin{align}
X_n(t) & := n^{-1} \left(S_{\lfloor nt \rfloor} + (nt - \lfloor nt \rfloor) \xi_{\lfloor nt \rfloor + 1}\right); \nonumber \\ X'_n(t) & := n^{-1} S_{\lfloor nt \rfloor}. \label{eqn:markovXn}
\end{align}
Then $X_n \in \calC^d_0$ and $X_n' \in \calD_0^d$.
See Figure~\ref{ex:traj} for an illustration of the two variants.
\begin{figure}
\centering
\includegraphics[scale=0.4]{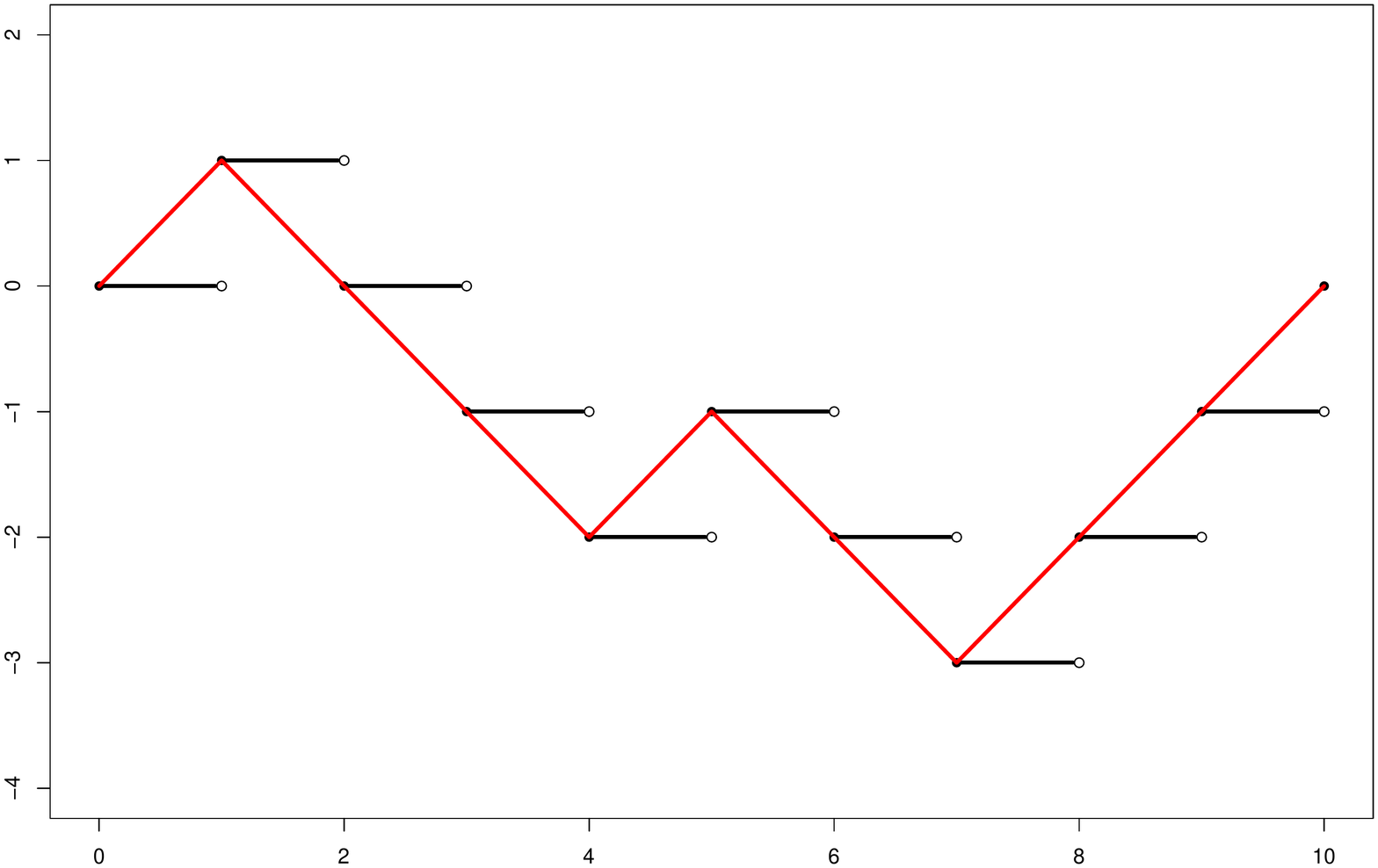}
\caption{An example of a possible random walk, and the two continuous-time trajectories we can create for it; the continuous interpolating $X_n(t)$ in red and the piecewise constant process $X'_n(t)$ in black.}
\label{ex:traj}
\end{figure}

Intuitively, we expect the trajectories to \lq look like\rq~they increase at a rate $\mu$. The strong law of large numbers allows us to develop a functional 
law of large numbers, which shows that this is indeed the case.
Theorem~\ref{thm:flln} is apparently stronger than Theorem~\ref{thm:SLLN} since convergence in the $\rho_\infty$ metric implies convergence of the endpoints 
$X_n (1) = n^{-1} S_n \toas \mu = I_\mu (1)$ and $X_n' (1) = n^{-1} S_n \toas \mu = I_\mu (1)$. However, we will see that Theorem~\ref{thm:flln} is in fact just a recasting of Theorem~\ref{thm:SLLN}, so the two results are equivalent. See Figure \ref{XnCS} for a simulation and for example \cite[p.~26]{whitt2002stochastic} for a reference.

\begin{theorem}[Functional law of large numbers]
\label{thm:flln} 
Consider the random walk defined at~\eqref{ass:walk}. Let $I_\mu \in \calC^d$ be the function defined by $I_\mu (t) := \mu t$ for $t \in [0,1]$.
\begin{itemize}
\item[(a)] We have $X_n \toas I_\mu$ on $(\calC_0^d, \rho_\infty)$.
\item[(b)] We have $X_n' \toas I_\mu$ on $(\calD_0^d, \rho_\infty)$.
\end{itemize}
\end{theorem}

\begin{remark}
By Lemma~\ref{lem:skor}, part~(b) also shows that $X_n' \toas I_\mu$ on $(\calD^d_0, \rho_S )$ and Proposition~\ref{prop:equivmetrics} in turn shows that $X_n' \toas I_\mu$ on $(\calD^d_0, \rho_S^\circ)$.
\end{remark}

\begin{figure}
\centering
\includegraphics[scale=0.4]{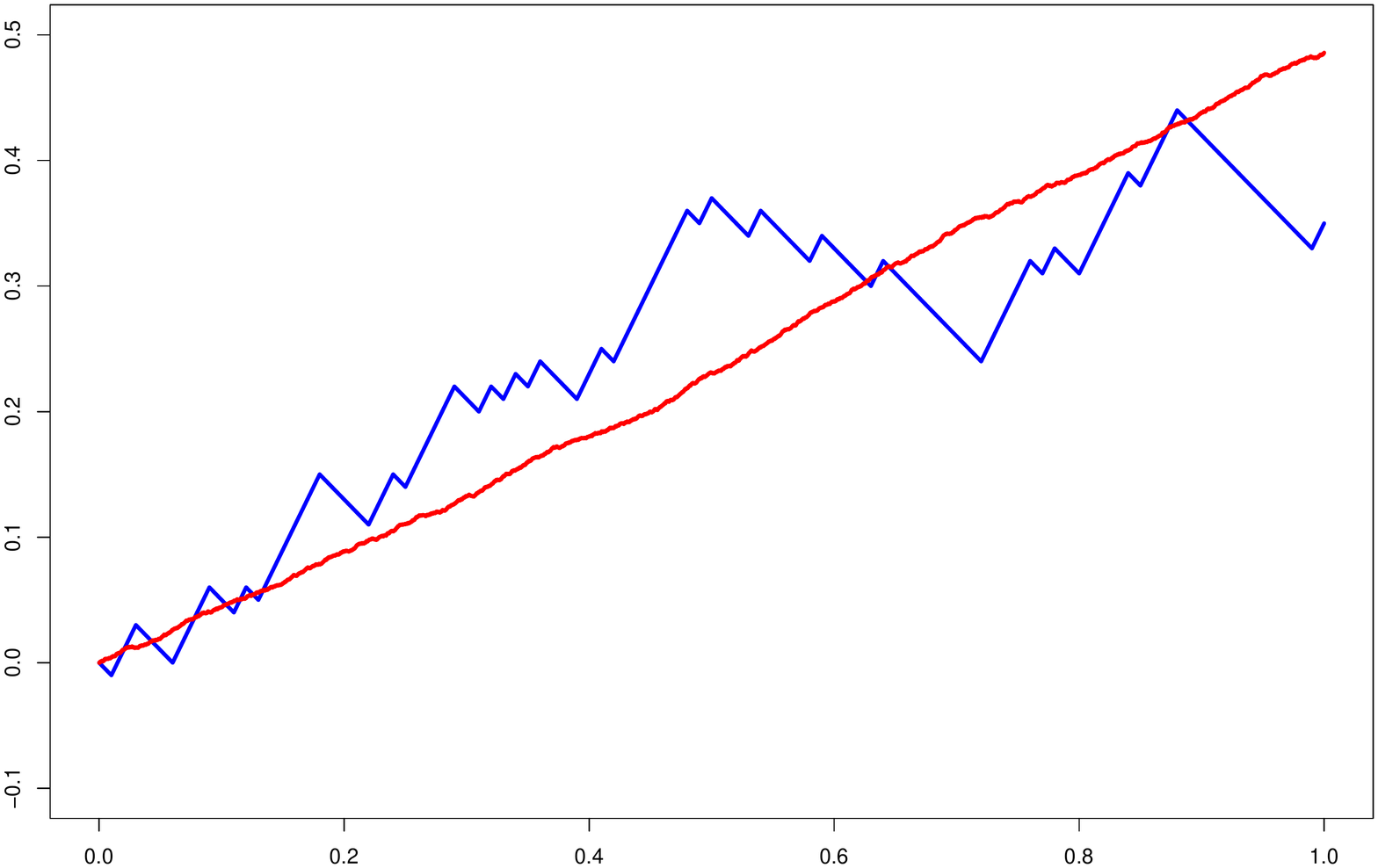}
\caption{Simulation of the trajectories $X_n(t)$ of a random walk on $\R$ with $\mu = 0.5$ for $n=100$ in blue and $n=10000$ in red.}
\label{XnCS}
\end{figure}

\begin{proof}[Proof of Theorem~\ref{thm:flln}]
Let $\eps > 0$. By Theorem~\ref{thm:SLLN}, there exists $N_{\eps}$ with $\Pr ( N_\eps < \infty) =1 $ such that, for all $n \geq N_{\eps}$, 
$\| n^{-1} S_n - \mu \| \leq \eps$.
Then
\begin{align}
\label{eq:fslln1}
\sup_{N_{\eps}/n \leq t \leq 1} \left \| X'_n (t) - \mu t \right \| 
& \leq \sup_{N_{\eps}/n \leq t \leq 1} \left \| X'_n(t) - \frac{\lfloor nt \rfloor}{n} \mu \right \|
+ \sup_{N_{\eps}/n  \leq t \leq 1} \left \| \frac{\lfloor nt \rfloor}{n} \mu - t \mu \right \| \nonumber\\
& \leq \sup_{N_{\eps}/n  \leq t \leq 1} \left( \frac{\lfloor nt \rfloor}{n} \right) \left \| \frac{S_{\lfloor nt \rfloor}}{\lfloor nt \rfloor} - \mu \right \|
+ \sup_{0 \leq t \leq 1} \left | \frac{\lfloor nt \rfloor}{n} - t \right | \| \mu \| \nonumber\\
& \leq \varepsilon + \frac{\| \mu \|}{n}.
\end{align}
On the other hand,
\begin{align}
\label{eq:fslln2}
\sup_{0 \leq t \leq  N_{\varepsilon}/n} \left \| X'_n(t) - \mu t \right \|
& \leq \frac{1}{n} \max_{0 \leq k \leq N_{\varepsilon}} \| S_k\| + \frac{N_{\varepsilon} \| \mu \|}{n} \toas 0 {~\rm as~} n\rightarrow \infty,
\end{align}
since $\Pr (N_\eps < \infty) =1$. Thus
combining~\eqref{eq:fslln1} and~\eqref{eq:fslln2} we obtain
\[ \limsup_{n \to \infty} \sup_{0 \leq t \leq 1} \left \| X'_n (t) - \mu t \right \| \leq \eps ,\]
and since $\eps >0$ was arbitrary, we get $\rho_\infty ( X'_n , I_\mu ) \toas 0$, proving part (b).

Let $X''_n(t) = S_{\lfloor nt \rfloor +1}$.
A similar argument to that above shows that, for $n \geq 1$,
\begin{align*} 
\sup_{N_\eps/n \leq t \leq 1} \| X''_n(t) - \mu t\| &
\leq \sup_{N_\eps /n \leq t \leq 1} \left( \frac{ \lfloor nt \rfloor + 1}{n} \right)
\left\| \frac{S_{\lfloor nt \rfloor +1}}{\lfloor nt \rfloor + 1} - \mu \right\| + \sup_{0 \leq t \leq 1 } \left| \frac{\lfloor nt \rfloor +1}{n} - t \right| \| \mu \| \\
& \leq 2 \eps + \frac{\| \mu \|}{n} ,
\end{align*}
and
\begin{align*}
\sup_{0 \leq t \leq N_\eps/n} \left \| X''_n(t) - \mu t \right \|
& \leq \frac{1}{n} \max_{0 \leq k \leq N_\eps +1 } \| S_k\| + \frac{N_{\varepsilon} \| \mu \|}{n} \toas 0 {~\rm as~} n\rightarrow \infty.
\end{align*}
It follows that $\rho_\infty (X''_n(t), I_\mu) \toas 0$ as well.
Let $\alpha_n (t) = nt - \lfloor nt \rfloor $; note that $\alpha_n(t) \in [0,1)$ for all $n \geq 1$ and all $t \in [0,1]$. Then
\[ X_n (t) = X_n'(t) + n^{-1} \alpha_n(t) \xi _{\lfloor nt \rfloor +1} = (1-\alpha_n (t) ) X_n ' (t) + \alpha_n (t) X_n'' (t) ,\]
so that
\begin{align*} \rho_\infty (X_n, I_\mu) & = \sup_{0\leq t \leq 1} \|(1-\alpha_n(t))(X_n'(t)-I_\mu(t)) + \alpha_n(t) (X_n''(t)-I_\mu(t))\| \\
& \leq \sup_{0 \leq t \leq 1} | 1 - \alpha_n (t) | \| X_n' (t) - I_\mu (t) \| + \sup_{0 \leq t \leq 1} |  \alpha_n (t) | \| X_n'' (t) - I_\mu (t) \| \\
& \leq \rho_\infty ( X_n' , I_\mu) + \rho_\infty (X_n'' , I_\mu ) ,\end{align*}
which tends to $0$ a.s., establishing part (a).
\end{proof}

\subsection{The maximum functional}
\label{sec:max}

As a first example of the theory developed above, we take $d=1$ and consider the \emph{maximum functional} $M : \calM \to \R$ defined by
$M (f) := \sup_{0 \leq t \leq 1} f (t)$.  Note that $|M(f)| \leq \| f \|_\infty$.
The next result shows that $M$ is a continuous map from $(\calM, \rho_\infty)$ to $(\R, \rho_E)$ and also
a continuous map from $(\calM, \rho_S)$ to $(\R, \rho_E)$. 

\begin{theorem}
\label{thm:maximum}
Let $d=1$. For any $f, g \in \calM$ we have $| M (f ) - M (g) | \leq \rho_S ( f,g ) \leq \rho_\infty (f,g)$.
\end{theorem}
\begin{proof}
Take $f, g \in \calM$, and suppose without loss of generality that $\sup_{s \in [0,1]} f(s) \geq \sup_{t \in [0,1]} g(t)$. 
Let $\Lambda'$ be the set of $\lambda : [0,1] \to [0,1]$ that are surjective, i.e., $\lambda [0,1] = [0,1]$.
Note that $\Lambda \subseteq \Lambda'$.
Then for any $\lambda \in \Lambda'$,
\begin{align*}
| M(f) - M(g) | & = \sup_{s \in [0,1]} f(s) - \sup_{t \in [0,1]} g(t) \\
& = \sup_{s \in [0,1]} f(s) - \sup_{t \in [0,1]} g \circ \lambda (t),\end{align*}
since $\lambda [0,1] = [0,1]$. Hence
\begin{align*}
| M (f) - M(g) | & =  \sup_{s \in [0,1]} \left( f(s) - \sup_{t \in [0,1]} g \circ \lambda (t) \right)  \\
& \leq \sup_{s \in [0,1]} \left( f(s) - g \circ \lambda (s) \right) \\
& \leq \sup_{s \in [0,1]} \left| f(s) - g \circ \lambda (s)  \right| \\
& = \| f - g \circ \lambda \|_{\infty} \\
& \leq \| \lambda - I \|_{\infty} \vee \| f - g \circ \lambda \|_{\infty}.
\end{align*}
We therefore have that
\begin{align*}
|M(f) - M(g)| & \leq \inf_{\lambda \in \Lambda'} \left\{ \| \lambda - I \|_{\infty} \vee \| f - g \circ\lambda \|_{\infty}  \right\} \leq \rho_S(f,g),
\end{align*}
since $\Lambda \subseteq \Lambda'$.
Lemma~\ref{lem:skor} completes the proof.
\end{proof}

Since we have shown the maximum functional is continuous, we can also apply the mapping theorem to the functional law of large numbers, to obtain the following result.

\begin{theorem}
Consider the random walk defined at~\eqref{ass:walk} with $d=1$. 
Then, as $n \to \infty$, \[ \frac{1}{n} \max_{0 \leq k \leq n} S_k \toas \mu^+ .\]
\end{theorem}
\begin{proof}
Let $X'_n(t)$ be as defined at \eqref{eqn:markovXn}. The functional strong law of large numbers, Theorem~\ref{thm:flln}, says that $X_n' \toas I_\mu$ on $(\calD,\rho_\infty)$,
while Theorem~\ref{thm:maximum} says that $M$ is continuous. Thus the mapping theorem, Theorem~\ref{thm:mapping}, implies that $M ( X_n' ) \toas M (I_\mu)$ on $(\R,\rho_E)$.
But $M ( X'_n ) = n^{-1} \max_{0 \leq k \leq n} S_k$ and $M ( I_\mu ) = \mu^+$, giving the result.
\end{proof}

\section{Functional central limit theorems}\label{FCLT}

\subsection{Weak convergence and the central limit theorem}
In this section we state and prove the functional central limit theorem, which extends the Lindeberg-L\'evy central limit theorem~\eqref{eqn:clt} in a similar
way that the functional law of large numbers extended the usual strong law as demonstrated in Section~\ref{FLLN}. 
First we recall the notion of convergence
in distribution for random variables on $\R^d$, in order to extend this appropriately which will allow us to state the functional theorem.

Given a random variable $X$ taking values in $\R^d$, we write
$X = (X_1, \ldots, X_d)^\tra$ in components. The distribution function $F$ of $X$ is defined for $t = (t_1,\ldots, t_d)^\tra \in \R^d$ by $F(t) := \Pr ( X_1 \leq t_1, \ldots, X_d \leq t_d )$.

\begin{definition}
Let $X, X_1, X_2, \ldots$ be a sequence of $\R^d$-valued random variables with corresponding distribution functions $F, F_1, F_2, \ldots$.
Then we say that $X_n$ converges in distribution to $X$, and write $X_n \tod X$, if $\lim_{n\to\infty}F_n(t)=F(t)$ for all continuity points $t$ of $F$.
\end{definition}

For random variables taking values in general metric spaces (such as our spaces of trajectories)
the concept that generalizes convergence in distribution is \emph{weak convergence}. First we define the concept for measures.

\begin{definition}\label{def:wc}
The probability measures $P_1,P_2,\ldots$ defined on a metric measure space $(S, \calS, \rho)$ converge weakly to $P$, that is, $P_n \Rightarrow P$, if 
\[ \int_{S}f \ud P_n \rightarrow \int_{S}f \ud P \]
 for all bounded, continuous $f: S \to \R$.
\end{definition}

It is often more convenient to speak of weak convergence of random variables. Consider a random variable $X$ on $(\Omega, \calF, \Pr)$, taking values in a metric measure space $(S, \calS, \rho)$. Consider also a sequence of random variables $X_n$, defined on possibly different probability spaces $(\Omega_n, \calF_n, \Pr_n)$, but all taking values in the same metric measure space $(S, \calS, \rho)$. We associate with $X, X_1, X_2, \ldots$ probability measures $P, P_1, P_2, \ldots$ on $(S, \calS, \rho)$ in the natural way: for any $B \in \calS$,
	\begin{align}
\label{eq:measure-rv}
	P ( B )  = \Pr ( X \in B ), \text{ and } P_n ( B )  = \Pr_n ( X_n \in B ) .
	\end{align}
\begin{definition}
In this context, we say that $X_n \Rightarrow X$ if $P_n \Rightarrow P$.
\end{definition}
	In other words, $X_n \Rightarrow X$ if $\lim_{n\to\infty} \Exp_n f(X_n) = \Exp f(X)$ for all bounded, uniformly continuous $f : S \to \R$, where $\Exp$ and $\Exp_n$ are
expectations under $\Pr$ and $\Pr_n$, respectively. 

\begin{remark}
In the case where $(S,\calS, \rho)$ is $(\R^d,\calB_d,\rho_E)$, where $\calB_d$ is the Borel $\sigma$-algebra of $\R^d$, weak convergence reduces to convergence in distribution: see \cite[p.~203]{gut}.
 \end{remark}

As with the almost-sure convergence in the previous section, it will be necessary, and fruitful, to consider mappings of random variables and thus necessary to show the convergence of the random variables is preserved under certain mappings. For this we state an analogue of Theorem~\ref{thm:mapping} which holds for weak convergence. This result can be found for example as \cite[Theorem~2.7]{cpm}. We defer the proof until Section~\ref{sec:weak-theory}.

Recall that given two metric measure spaces $(S, \calS, \rho)$ and $(S', \calS', \rho')$ and a function $h : S \to S'$, the set of discontinuities of $h$ is denoted by $D_h$.

\begin{theorem}[Mapping theorem for weak convergence]
	\label{thm:mapweak}
	Let $P, P_1, P_2, \ldots$ be a sequence of probability measures on a metric measure space $(S, \calS, \rho)$.
	Let $(S',\calS',\rho')$
	be another metric measure space, and let $h: (S, \calS, \rho) \to (S', \calS', \rho')$ be measurable.
	For each $n$, we define $P_n h^{-1}$ a probability measure on 
	$(S', \calS', \rho')$ 
	by $P_nh^{-1} ( A ) = P_n ( h^{-1} ( A ) )$ for $A \in \calS'$. 
	If $P_n \Rightarrow P$ and $P ( D_h ) = 0$, then $P_n h^{-1} \Rightarrow P h^{-1}$.
\end{theorem}

Again, we may recast this result about weak convergence of measures in the language of weak convergence of random variables.
Consider a random variable $X$ on $(\Omega, \calF, \Pr)$, taking values in a metric measure space $(S, \calS, \rho)$.
Consider also a sequence of random variables $X_n$, defined on possibly different probability spaces $(\Omega_n, \calF_n, \Pr_n)$, but all taking values in the same 
metric measure space $(S, \calS, \rho)$.
Let $(S', \calS', \rho')$ be another metric measure space. Let $h: (S, \calS, \rho) \to (S', \calS', \rho')$ be measurable.
By~\eqref{eq:measure-rv}, we may thus deduce from Theorem~\ref{thm:mapweak} the following.

\begin{corollary}
	\label{cor:weak-mapping}
	If $X_n \Rightarrow X$ and $\Pr ( X \in D_h ) =0$, then $h(X_n) \Rightarrow h(X)$.
\end{corollary}

The final classical ingredient for this section is the multidimensional version of the classical central limit theorem which can be found for example as \cite[Theorem~3.9.6]{durrett}. We also state a theorem of P\'{o}lya, see for example~\cite[Theorem~2.6.1]{lehmann}, which will allow us to take the convergence in the central limit theorem to be uniform convergence.

\begin{theorem}[Multidimensional central limit theorem]
	\label{thm:CLT}
	Suppose that we have a random walk as defined at \eqref{ass:walk} satisfying \eqref{ass:Sigma}. 
	Then as $n \to \infty$,
	\[ \frac{1}{\sqrt{n}} \left( S_n - n\mu \right) \tod \mathcal{N}(0,\Sigma). \]
\end{theorem}

\begin{theorem}\label{thm:polya}
	Let $F_1,F_2,\ldots$ be a sequence of cumulative distribution functions such that $F_n \tod F$. If $F$ is continuous, then $F_n(x)$ converges to $F(x)$ uniformly in $x$.
\end{theorem}
\subsection{Functional central limit theorem and some applications}
\label{Donsker}

The functional version of the central limit theorem is known as \emph{Donsker's theorem}. For simplicity of exposition,
and with the applications that come later in mind, we consider only the zero drift case where $\mu =0$.
Again we work with trajectories indexed by $[0,1]$; now the scaling is the central limit theorem scaling rather than
the law of large numbers scaling. Precisely, for $n \in \N$ and $t \in [0,1]$ we define
\begin{align}
\label{eqn:trajectories}
Y_n(t) & := \frac{1}{ \sqrt{n}}  \left( S_{\lfloor nt \rfloor} + (nt - \lfloor nt \rfloor) \xi_{\lfloor nt \rfloor + 1} \right);\\
Y'_n(t) & := \frac{1}{\sqrt{n}} S_{\lfloor nt \rfloor}.\nonumber 
\end{align}
Here $Y_n \in \calC_0^d$ and $Y'_n \in \calD_0^d$.
Despite the latter living in $\calD_0^d$, in the domain of the central limit theorem, we would not expect a single increment to be macroscopic on the scale of $\sqrt{n}$, so we might expect to see a limiting distribution with continuous trajectories, and increments which are independent and normally distributed. The obvious candidate for such a limit process is the $d$-dimensional Brownian motion.

\begin{theorem}[Donsker's theorem]
\label{thm:Donsker-dd}
Suppose that we have a random walk as defined at \eqref{ass:walk} with $\mu =0$ and satisfying \eqref{ass:Sigma}. 
\begin{itemize}
\item[(a)]
We have $Y_n \Rightarrow \Sigma^{1/2}b_d$ in the sense of weak convergence on $(\calC_0^d, \rho_{\infty})$.
\item[(b)] 
We have $Y'_n \Rightarrow \Sigma^{1/2} b_d$ in the sense of weak convergence on $(\calD_0^d, \rho_S)$.
\end{itemize}
\end{theorem}

This result in the one-dimensional case was proved by Donsker in 1951~\cite{donsker}. We point the reader to \cite[\S 5]{ethier} for a comprehensive discussion of both $d$-dimensional Brownian motion and the steps leading to this result.
  Figure~\ref{sim:fclt} shows some simulations of one-dimensional processes. 
  
  \begin{remark} \label{rem:wcequivalence}
   Part (b) of Theorem~\ref{thm:Donsker-dd} is stated for the space $(\calD_0^d,\rho_S)$, but weak convergence on $(\calD_0^d,\rho_S)$ is equivalent to weak convergence on $(\calD_0^d,\rho_S^\circ)$. To see this, recall the definition of weak convergence from Definition~\ref{def:wc}, and note Proposition~\ref{prop:equivmetrics} which tells us that a continuous function $f$ under one metric is continuous under the other. Thus, the set of bounded continuous functions is the same in both metric spaces and so weak convergence must be equivalent.
  \end{remark}
  
  We discuss the proof of Theorem~\ref{thm:Donsker-dd} later in this section, but first, demonstrate the power of Donsker's theorem with the mapping theorem by returning to our example of the maximum functional in $d=1$ as defined in Section~\ref{sec:max}, followed by a further example --- a generalisation of the arcsine law.

\begin{figure}
\centering
\includegraphics[scale=0.4]{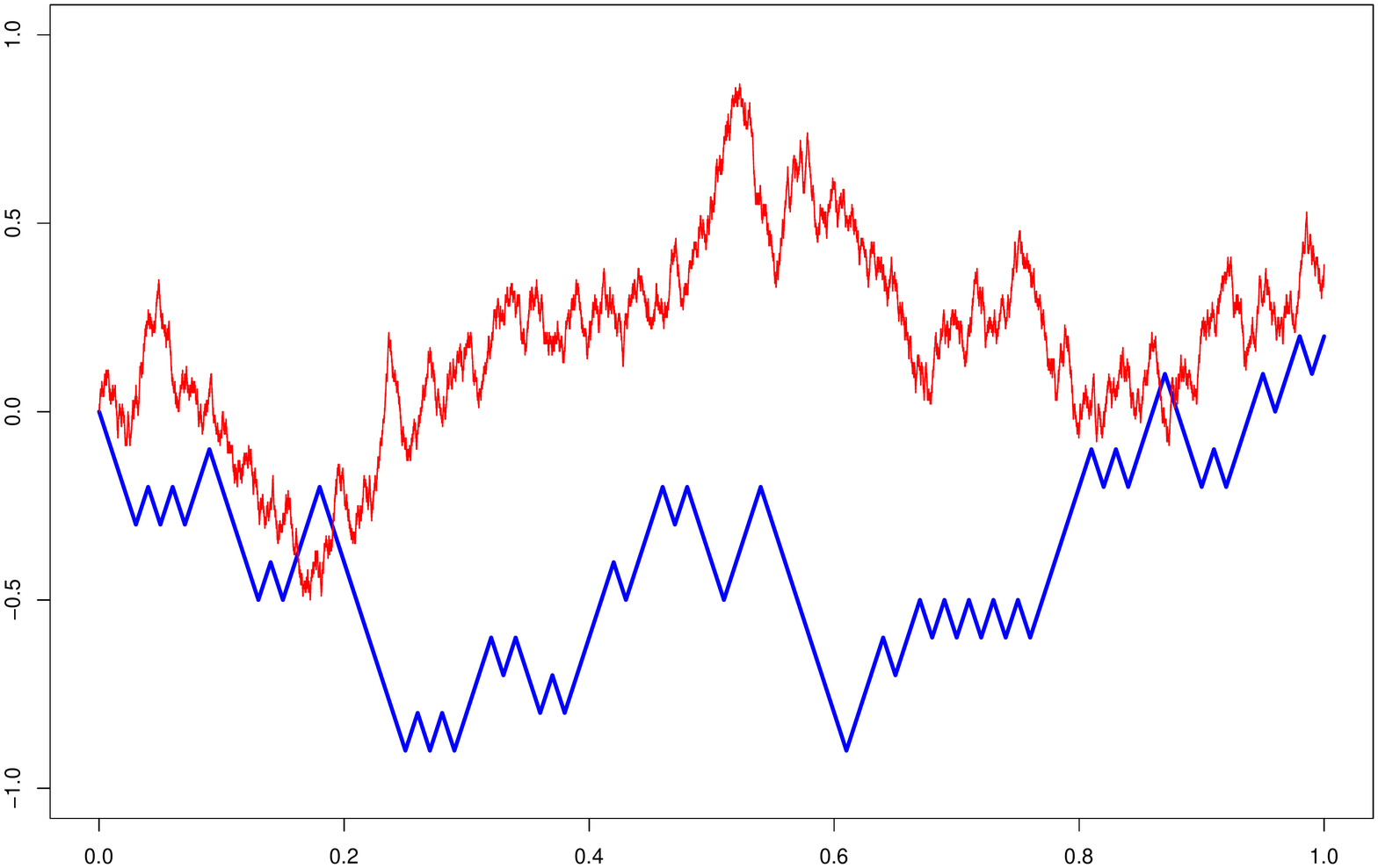}
\caption{Simulation of a sample path of $Y_n(t)$ in the case where $d=1$, for $n=100$ in blue and $n=10000$ in red.}\label{sim:fclt}
\end{figure}

\begin{theorem}
	\label{thm:max-dist}
	Suppose that we have a random walk as defined at \eqref{ass:walk} with $d=1$, $\mu =0$, and satisfying \eqref{ass:Sigma}. 
	Then as $n \to \infty$,
	\[ \frac{1}{\sqrt{n}} \max_{0 \leq k\leq n}  S_k \tod \sigma \sup_{t \in [0,1]} b(t). \]
\end{theorem}
\begin{proof}
	In the case $d=1$, $\Sigma$ is the scalar $\sigma^2$.
	Donsker's theorem, Theorem~\ref{thm:Donsker-dd},
	together with the mapping theorem, Corollary~\ref{cor:weak-mapping},
	and continuity of the function $M : (\calD, \rho_S ) \to (\RP, \rho_E)$, Theorem~\ref{thm:maximum},
	shows that 
	$$ M ( Y_n') = \sup_{t \in [0,1]} Y'_n(t) \tod M ( \sigma  b ) = \sigma  \sup_{t \in [0,1]} b(t).$$
	But we have that $\sup_{t\in[0,1]}Y'_n(t) = n^{-1/2} \max \{S_0,S_1,\ldots,S_n\}$, completing the proof.
\end{proof}

The distribution of $\sup_{t \in [0,1]} b(t)$
can be determined by the reflection principle for Brownian motion,
and so 
Theorem~\ref{thm:max-dist} gives us the limiting distribution for $\max_{1 \leq k \leq n} S_k$:
see~\cite[pp.~91--93]{cpm}.

We now turn to our second example. The classical arcsine law states the following \cite[p.~82]{feller1}, first established for the simple symmetric random walk.

\begin{theorem}
	If $0 < \gamma < 1$, the probability that an $n$-step simple symmetric random walk spends less than $\gamma n$ time on the positive side tends to $2\pi^{-1} \arcsin \sqrt{\gamma}$ as $n\rightarrow \infty$.
\end{theorem}

We wish to extend this to higher dimensions, which requires a generalisation of the functional itself. In \cite{bingham} the functional which generalises \lq time on the positive side\rq~to \lq time in the positive quadrant\rq~is considered and shown not to follow an arc-sine distribution by comparison of moments. The generalisation that we will consider is $\pi_n(A)$, defined to be the proportion of time the normalised walk spends in a given subset of the sphere. Formally, recall $\hat{x}:=x/\|x\|$ for $x \neq 0$ and $\hat{0}:=0$, then, for a measurable set $A \subseteq \Sp^{d-1}$, 
\[ \pi_n(A):=  \frac{1}{n} \sum_{i=1}^n \1{ \hat S_i \in A } .\]

\begin{theorem}\label{thm:arcsineconv}
	Suppose that we have a random walk as defined at \eqref{ass:walk} with $\mu=0$, and satisfying \eqref{ass:Sigma}. Let $\hat{b}^\Sigma_d(t):=\Sigma^{1/2} b_d(t)/\|\Sigma^{1/2} b_d(t)\|$, the $d$-dimensional Brownian motion projected onto the sphere and $A \subseteq \Sp^{d-1}$ with $\mu_{d-1}(\partial A)=0$, where $\mu_{d-1}$ here denotes Haar measure on $\Sp^{d-1}$. Then as $n\rightarrow \infty$,
	
	\[ \pi_n(A) \tod \int_0^1 \1{\hat{b}_d^\Sigma(t)\in A} \ud t.\]
\end{theorem}

As with all our examples, we must prove the continuity of the functional in order to complete the proof. First, for measurable $A \subseteq \Sp^{d-1}$ and $f \in \calD^d$, define
\[ \varpi_A (f) := \int_0^1 \mathds{1} \left\{ \widehat{ f(t) } \in A\right\} \ud t .\]
Note that $\pi_n(A) = \varpi_A (Y_n')$.

\begin{lemma}
	\label{lem:occupation_continuity}
	Fix a measurable $A \subseteq \Sp^{d-1}$.
	Then,
	as a function from $(\calD^d,\rho_S)$
	to $([0,1],\rho_E)$, 
	$f \mapsto \varpi_A (f)$ is   continuous   on the set 
	\[ F_A := \left\{ f \in \calD^d : \int_0^1 \mathds{1} \left\{ \widehat{ f(t) } \in \{ 0 \} \cup \partial A \right\} \ud t = 0 \right\}. \]
\end{lemma}

\begin{proof}
	Since $\rho_S$ and $\rho_S^\circ$ are equivalent (see Proposition~\ref{prop:equivmetrics}), it suffices to work with the latter.
	For $f \in \calD^d$ define for all measurable $B \subseteq \R^d$,
	\[ \nu_f (B) := \int_0^1 \1{ f(t)  \in B} \ud t .\]
	Note that $\nu_f$ is a finite measure
	on $\R^d$.
	Now 
	let $\tilde A = \{ r x : x \in A, \, r > 0\}$, then $\varpi_A (f) = \nu_f (\tilde A)$.
	Take $f, g \in \calD^d$ and suppose, without loss of generality,
	that $\nu_f (\tilde A) \geq \nu_g (\tilde A)$, let $\tilde A^\eps = \{ x \in \R^d : \rho (x, \tilde A ) \leq \eps \}$
	and let $\tilde A_\eps = \{ x \in \R^d : \rho (x, \R^d \setminus \tilde A ) \geq \eps\}$,
	then $\tilde A_\eps \subseteq \tilde A \subseteq \tilde A^\eps$ and
	\begin{align*}
	| \varpi_A (f) - \varpi_A (g) | = \nu_f(\tilde A)- \nu_g (\tilde A)  
	= \nu_f ( \tilde A_\eps ) - \nu_g ( \tilde A^\eps ) + \nu_f ( \tilde A \setminus \tilde A_\eps)
	+ \nu_g (\tilde A^\eps \setminus \tilde A) 
	.\end{align*}
	If $f, g \in F_A$ then since $\hat x \in \partial A$ implies $x \in \partial \tilde A$,
	we have that as $\eps \to 0$, by continuity of measures along monotone limits,
	$\nu_f ( \tilde A \setminus \tilde A_\eps) \to \nu_f ( \partial \tilde A ) =0 $, and
	$\nu_g (\tilde A^\eps \setminus \tilde A) \to \nu_g ( \partial \tilde A) = 0$.
	Moreover, with the change of variable $t = \lambda(s)$ in the $\nu_g$-integral,
	\begin{align*}
	\nu_f ( \tilde A_\eps ) - \nu_g ( \tilde A^\eps ) 
	& = \int_0^1 \1{ f(t) \in \tilde A_\eps } \ud t - \int_0^1 \1{ g ( t) \in \tilde A^\eps} \ud t \\
	& = \int_0^1 \1{ f(t) \in \tilde A_\eps} \ud t - \int_0^1 \lambda' (s) \1{ g ( \lambda (s) ) \in \tilde A^\eps} \ud s\\
	& \leq \| \lambda' - 1 \| + \int_0^1 \1{ f(t) \in \tilde A_\eps, \, g(\lambda(t)) \notin \tilde A^\eps} \ud t
	.\end{align*}
	Here we have that
	\[
	\int_0^1 \mathds{1} \left\{ f(t) \in \tilde A_\eps, \, g(\lambda(t)) \notin \tilde A^\eps \right\} \ud t
	\leq   \mathds{1} \left\{\| f  - g \circ \lambda \| \geq  2 \eps \right\}   .\]  
	In particular, given $f  \in F_A$ and $\eps >0$, we can choose $\delta$ sufficiently small so that any $g$ with $\rho^\circ_S (f,g) < \delta$
	has a $\lambda$ for which, by Lemma~\ref{lem:estlam},	$\| \lambda' - 1 \| \leq c(\lambda) < \eps$ and $\| f -g \circ \lambda \| < \eps$.
	Hence, since $\eps>0$ was arbitrary,
	$| \varpi_A (f) - \varpi_A (g) | \to 0$ as $\rho^\circ_S (f,g) \to 0$.
\end{proof}

\begin{proof}[Proof of Theorem~\ref{thm:arcsineconv}]
Fix $A \subseteq \Sp^{d-1}$ with $\mu_{d-1} (\partial A ) = 0$.
By Donsker's theorem, Theorem 4.9(b), $Y_n' \Rightarrow  \Sigma^{1/2} b_d$,
and since  $F_A$ has measure $1$ under the law of $\Sigma^{1/2} b_d$, the 
continuous mapping theorem, Theorem 4.5, with Lemma~\ref{lem:occupation_continuity} shows that 
$\pi_n ( A  ) = \varpi_A ( Y_n') \tod \varpi_A ( \Sigma^{1/2} b_d )$.	 
\end{proof}

In particular, we can use this result to determine that there is no limit for the proportion of time spent in any non-trivial set.
\begin{corollary}
	For any set $A \subseteq \Sp^{d-1}$ with $0<\mu_{d-1}(A)<\mu_{d-1}(\Sp^{d-1})$ and $\mu_{d-1}(\partial A)=0$,	
	\begin{equation*}
	 \liminf_{n \rightarrow \infty} \pi_n(A) = 0 \as \qquad {\textrm and} \qquad \limsup_{n \rightarrow \infty} \pi_n(A) = 1 \as
	\end{equation*}
\end{corollary}
\begin{proof}
We will use the Hewitt-Savage zero-one law\cite[p.~180]{durrett}. In order to do so, we need to show that $\limsup_{n\rightarrow \infty} \pi_n(A)$ and $\liminf_{n\rightarrow \infty} \pi_n(A)$ are exchangeable random variables. For this, note
\begin{align*}
\limsup_{n\rightarrow \infty} \pi_n(A) &= \limsup_{n\rightarrow \infty} \left( \frac{1}{n} \sum_{i=1}^k \1{\hat{S}_i \in A} + \frac{1}{n} \sum_{i=k+1}^n \1{\hat{S}_i \in A} \right)\\
& = \limsup_{n\rightarrow \infty} \frac{1}{n} \sum_{i=k+1}^n \1{\hat{S}_i \in A} \ \as
\end{align*} 	
which clearly does not depend on the order of the first $k$ increments, and since $k$ was arbitrary, it is clearly exchangeable. The exact same argument is true for the $\liminf$ as well.

Thus, it will be sufficient to show that $\Pr(\limsup_{n\rightarrow \infty} \pi_n(A) \geq 1-\eps) >0$ for any $\eps >0$, and $\Pr(\liminf_{n\rightarrow \infty} \pi_n(A) \leq \eps) >0$ for any $\eps >0$. For the former, note
\begin{align}
\Pr(\limsup_{n\rightarrow \infty} \pi_n(A) \geq 1-\eps) &\geq \Pr(\pi_n(A)>1-\eps {\rm\ i.o.}) \nonumber \\
&\geq \Pr(\cap_{n=1}^\infty \cup_{m\geq n} \{ \pi_m(A) > 1-\eps \})\nonumber\\
&= \lim_{n \rightarrow \infty} \Pr(\cup_{m\geq n} \{ \pi_m(A) > 1-\eps \})\nonumber\\
&\geq \lim_{n\rightarrow \infty} \Pr(\pi_n(A) > 1-\eps). \label{eqn:arcsine1}
\end{align}
Then Theorem~\ref{thm:arcsineconv} states that $\pi_n(A) \tod \int_0^1 \1{\hat{b}_d^\Sigma(t) \in A} \ud t$ so for all but countably many $\eps>0$, 
\[\lim_{n\rightarrow \infty} \Pr(\pi_n(A) > 1-\eps) = \Pr\left( \int_0^1 \1{\hat{b}_d^\Sigma(t) \in A} \ud t > 1-\eps \right).\]

Now, recall $\inte A:=A \setminus \partial A$ is the interior of $A$, which is an open set, see for example \cite[pp.~44--46]{kelley}. By the assumptions $\mu_{d-1}(A)>0$ and $\mu_{d-1}(\partial A)=0$ it follows that $\mu_{d-1}(\inte A)>0$ and so the interior is non-empty. Since the interior is an open, non-empty subset of $A$, it follows that there exists at least one ball, call it $A_\eps$, with radius $\eps>0$ such that $A_\eps \subseteq A$. 

Then it is easy to see that there is positive probability that $\hat{b}_d^\Sigma (t)$ stays in $A_\eps$ for all $t\in [\eps,1]$ for any $\eps>0$ (allowing the path to move away from $0$). Thus, combining this with \eqref{eqn:arcsine1}, we have 
\[\Pr\left(\limsup_{n\rightarrow \infty} \pi_n(A) \geq 1-\eps\right) \geq \lim_{n\rightarrow \infty}\Pr(\pi_n(A) \geq 1-\eps) = \Pr\left( \int_0^1 \1{\hat{b}_d^\Sigma(t) \in A} \ud t > 1-\eps \right)>0\] for any $\eps>0$, and the Hewitt-Savage zero-one law gives us $\limsup_{n\rightarrow \infty}\pi_n(A) = 1 \as$ as required.

Finally, note that $\pi_n(A^c)\leq 1-\pi_n(A)$ (the inequality is due to possible visits to $0$) and since $\mu_{d-1}(A)+\mu_{d-1}(A^c) = \mu_{d-1}(\Sp^{d-1})$, we get $0 < \mu_{d-1}(A^c)<\mu_{d-1}(\Sp^{d-1})$. Also, $\partial A = \partial A^c$, see for example \cite[p.~46]{kelley}, so $\mu_{d-1}(\partial A^c)=\mu_{d-1}(\partial A)=0$. Thus, the conditions of the previous calculation are in fact satisfied for $A^c$, so $\liminf_{n\rightarrow \infty}\pi_n(A) \leq 1- \limsup_{n\rightarrow \infty}\pi_n(A^c) = 1-1 = 0 \as$ which completes the proof.
\end{proof}
\subsection{Proof overviews and a motivating example}

In order to prove both the mappping theorem and Donsker's theorem, we will need to delve further into weak convergence theory. First, in Section \ref{sec:weak-theory} we will present different characterisations of weak convergence and note Slutsky's theorem in this context, all of which will be necessary for the proofs. Then we will present the proof of the mapping theorem.

For the proof of Donsker's theorem, it could be suggested that a sufficient method would be to take some finite number of points on the trajectory, and see if the distribution of their location converges to the equivalent distribution for such points on a Brownian path. We now demonstrate why this will not be sufficient.

First, for $t \in [0,1]$ and $f \in \calM^d$, recall the \emph{projection}
$\pi_t : \calM^d \to \R^d$ is denoted $\pi_t f := f (t)$.
More generally, for $k \in \N$ and $t_1, t_2, \ldots, t_k \in [0,1]$, we define $\pi_{t_1, t_2, \ldots, t_k} : \calM^d \to (\R^d )^k$ by 
\[ \pi_{t_1, \ldots, t_k} f := ( f(t_1 ), \ldots, f(t_k) ). \] 

We say the finite-dimensional distributions of a function converge if we have the following,
\begin{description}
	\item
	[\namedlabel{ass:finiteDimDist}{\textbf{FDD}}]
	\begin{itemize}
		\item[(i)] If $X, X_1, X_2, \ldots$ is a sequence in $\calC^d$ then, for all $t_i \in [0,1]$, \[\pi_{t_1, t_2, \ldots, t_k} X_n 
	= \left(X_n(t_1), X_n(t_2), \ldots , X_n(t_k)\right) 
	\Rightarrow \left( X(t_1), X(t_2), \ldots, X(t_k) \right) 
	= \pi_{t_1, t_2, \ldots, t_k} X,\] where the convergence is on $(\R^d)^k$.
	\item[(ii)] If $X, X_1, X_2, \ldots$ is a sequence in $\calD^d$ then, \[\pi_{t_1, t_2, \ldots, t_k} X_n 
	= \left(X_n(t_1), X_n(t_2), \ldots , X_n(t_k)\right) 
	\Rightarrow \left( X(t_1), X(t_2), \ldots, X(t_k) \right) 
	= \pi_{t_1, t_2, \ldots, t_k} X, \]
	where the convergence is on $(\R^d)^k$ and holds for all $(t_1, t_2, \ldots, t_k)$ such that each $\pi_{t_i}$ is continuous. 
\end{itemize}
Note that, in both cases, the weak convergence on $(\R^d)^k$ is convergence in distribution.
\end{description}

Noting that $\|\pi_t f - \pi_t g\| \leq \|f-g\|_{\infty}$ and so $\|\pi_{t_1,t_2,\ldots,t_k} f -\pi_{t_1,t_2,\ldots,t_k} g\| \leq \sqrt{k}\|f-g\|_{\infty}$, it follows that the projection is a continuous function from $(\calC^d,\rho_\infty)$ to $(\R^d,\rho_E)$, hence it is a direct consequence of the mapping theorem, Theorem \ref{thm:mapweak}, that, if $X_n \Rightarrow X$ on $\calC^d$, then the finite-dimensional distributions also converge. Unfortunately, the reverse is not necessarily true; there exist sequences of probability measures whose finite-dimensional distributions converge weakly, though the measures themselves do not. 
\begin{example}\label{ex:FDD}
	Consider the following functions, with examples $z_3$ plotted in blue, $z_4$ plotted in light-green and $z_{10}$ plotted in red;
	
	\begin{minipage}{0.4\textwidth}
		\begin{align*}
		z_n(t)&=\begin{cases}
		nt & \mathrm{for\ }t\in[0,1/n);\\
		2-nt & \mathrm{for\ }t\in[1/n,2/n);\\
		0 & \mathrm{for\ }t \geq 2/n .
		\end{cases}\\
		\end{align*}
	\end{minipage}
	\begin{minipage}{0.5\textwidth}
		\centering
		\includegraphics[scale=0.28]{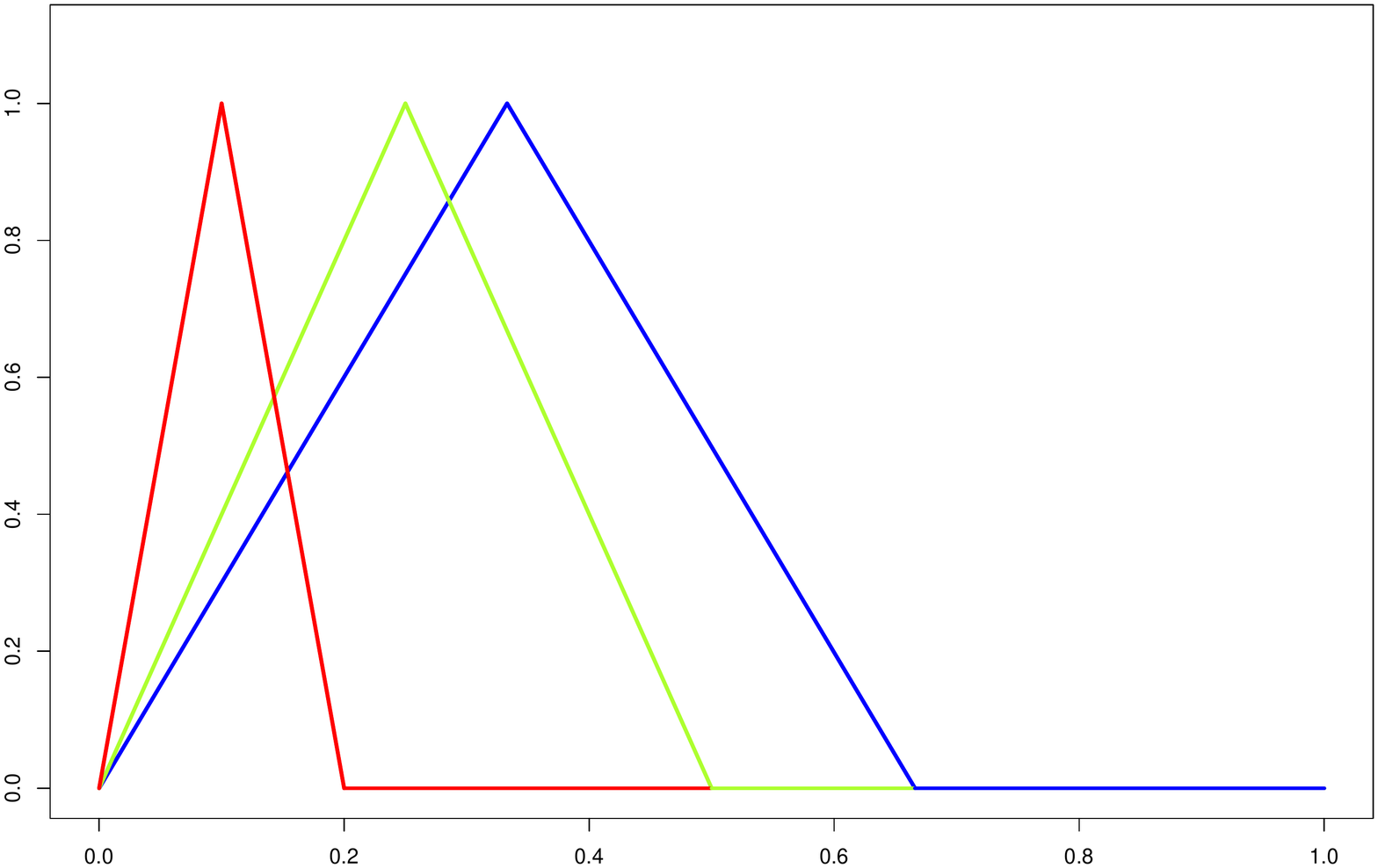}
	\end{minipage}
	
	If we set $P_n = \delta_{z_n}$, the point mass at the function $z_n$, and $P= \delta_0$, then as soon as $t_i \geq 2n^{-1}$ for all $i$, $\pi_{t_1, \dots, t_k}z_n = (0, \dots, 0) = \pi_{t_1, \dots, t_k} 0$, so weak convergence of finite-dimensional distributions holds; but, since $\rho_{\infty}(z_n,0)=1$ for all $n$, $z_n \nrightarrow 0$ so $P_n \nRightarrow P$; we do not have weak convergence.
\end{example}

Based on this example, it is clear that we need a further condition on the family $\{P_n\}$. For trajectories in $\calC_0^d$ it happens that such a sufficient condition is relative compactness, but it is hard to directly prove that a family of measures is relatively compact. However, Prokhorov's theorem tells us that tightness implies relative compactness, so we can work with tightness. Finally, we will use a couple of probability bounds on the running maximum of the trajectory to prove the tightness in $\calC_0^d$. We complete the proof by showing the finite-dimensional distributions do in fact converge in this case.

Of course, the results for continuous trajectories are only enough to prove part (a) of Theorem \ref{thm:Donsker-dd}. For part (b), we will show that tightness is still a sufficient condition to ensure the finite-dimensional limit is in fact the weak limit of the family of measures. Then we note some relevant changes to the conditions for tightness which we will prove are satisfied in $\calD_0^d$. Finally, we show the finite-dimensional distribution convergence enabling us to conclude the weak convergence statement of Theorem~\ref{thm:Donsker-dd} part (b).

\begin{remark}
\label{rem:donsker}
It suffices to prove Donsker's theorem for the case $\Sigma=I_d$. To see this, consider the walk defined at \eqref{ass:walk} with $\mu =0$, for which \eqref{ass:Sigma} holds with some arbitrary $\Sigma$.
Assuming $\sigma^2>0$, if any of the eigenvalues of $\Sigma$ are zero, then the walk is not truly $d$-dimensional and can be mapped to a walk with smaller dimension such that the covariance matrix for this walk is positive definite, see for example \cite[p.~4]{lawler}. Any results where this is the case, would of course then relate to weak convergence to the Brownian path in the lower dimension contained on the hypersurface, and this statement can be mapped back to the original space. Hence, it suffices to assume $\Sigma$ is positive definite.
Indeed, if $\Sigma$ is positive definite, then the (unique) symmetric square-root $\Sigma^{1/2}$ is also positive definite, and $\Sigma^{1/2}$
has inverse $\Sigma^{-1/2}$. 
Then set $\zeta := \Sigma^{-1/2} \xi$, 
and let $\zeta_i= \Sigma^{-1/2}\xi_i$ for $i \in \N$. 
By linearity of expectation, $\Exp \zeta = \Sigma^{-1/2} \Exp \xi = 0$ and 
\[ \Exp [ \zeta \zeta^\tra ] = \Exp [ \Sigma^{-1/2}\xi \xi^\tra \Sigma^{-1/2} ] = \Sigma^{-1/2} \Exp [ \xi \xi^\tra ] \Sigma^{-1/2} = I_d .\]
Let $\tilde S_n := \sum_{i=1}^n \zeta_i$ be the random walk associated with $\zeta$. Then $\tilde S_n = \Sigma^{-1/2} S_n$, and $\tilde S_n$ satisfies
 \eqref{ass:walk} and \eqref{ass:Sigma} with $\mu =0$ and $\Sigma=I_d$. 
The analogue of $Y'_n$ for $\tilde S_n$ is 
\[ \tilde Y'_n (t) = n^{-1/2} \tilde S_{\lfloor nt \rfloor} = \Sigma^{-1/2} Y'_n (t) ,\]
so $Y'_n = \Sigma^{1/2} \tilde Y'_n$. 
The case of Theorem~\ref{thm:Donsker-dd}(b) where $\Sigma = I_d$
yields $\tilde Y'_n \Rightarrow b_d$. Since
$x \mapsto \Sigma^{1/2} x$ is continuous, the mapping theorem, Theorem~\ref{thm:mapweak},
shows that $ Y'_n = \Sigma^{1/2} \tilde Y'_n \Rightarrow \Sigma^{1/2} b_d$,
which is the conclusion of Theorem~\ref{thm:Donsker-dd}(b) in the general case. A similar argument holds for Theorem~\ref{thm:Donsker-dd}(a).
Thus we can conclude that Donsker's theorem holds for general $\Sigma$ following from the special case where $\Sigma = I_d$.
\end{remark}

\subsection{Some weak convergence theory}
\label{sec:weak-theory}
 
The Portmanteau theorem (see e.g.~\cite[Theorem~2.1]{cpm}) gives several different characterisations of weak convergence. We only state them in terms of probability measures, but throughout consider random variables $X_n$ to be endowed with the respective measure $P_n$, and hence statements like (ii) could be written as convergence of expectations of the respective random variables, notation we will use later.

\begin{theorem}[Portmanteau theorem]
\label{thm:port}
Let $P, P_1, P_2, \ldots$ be probability measures on metric measure space $(S,\calS,\rho)$.
	The following statements are equivalent.
	\begin{itemize}
		\item[(i)] $P_n \Rightarrow P$.
		\item[(ii)] $\int_{S} f \ud P_n \rightarrow \int_{S} f \ud P$ for all bounded, uniformly continuous $f$.
		\item[(iii)] $\limsup_{n\to\infty} P_n ( F ) \leq P ( F )$ for all closed sets $F$.
		\item[(iv)] $\liminf_{n\to\infty} P_n ( G ) \geq P ( G )$ for all open sets $G$.
		\item[(v)] $\lim_{n\to\infty} P_n ( A ) = P ( A )$ for all $A$ such that $P( \partial A ) = 0$.
	\end{itemize}
\end{theorem}

\begin{proof} First, note that (i) implies (ii) by definition.

Next we show that (ii) implies (iii). 
Let $F$ be a closed set, let $\eps>0$ and recall $\ind_A$ is the indicator
function of the set $A$. Take $f$ defined by
 $f(x) = (1- \eps^{-1} \rho(x,F))^+$, so 
$f(x) =1$ for $x \in F$ and $f(x) = 0$ for $x \notin F^\eps$, which gives
$\ind_F(x) \leq f(x) \leq \ind_{F^{\eps}}(x)$.
 Thus $f$ is bounded. A simple calculation also shows $|f(x)-f(y)|\leq \eps^{-1} \rho(x,y)$, so $f$ is also uniformly continuous.
Then, \[ \limsup_{n\to \infty} P_n ( F ) = \limsup_{n\to \infty} \int \ind_F \ud P_n
 \leq \limsup_{n\to\infty}  \int f \ud P_n  .\]
So by (ii) we get
\[ \limsup_{n\to \infty} P_n ( F ) \leq  \int f \ud P \leq \int \ind_{F^{\eps}} \ud P  = P ( F^{\eps} ). \]
Take $\eps = 1/k$.
Since 
$F$ is closed, $F = \cap_{k \in \N} F^{1/k}$. 
Then continuity along monotone limits shows that $P( F^{1/k} ) \downarrow P(F)$ as $k \to \infty$, and we obtain (iii).
	
Next, observe that (iii) is equivalent to (iv) by complementation.
	
We next show that (iii) and (iv) together imply (v). Indeed,
\begin{align*}
 P (\cl  A) &\geq \limsup_{n\to\infty} P_n ( \cl A ) \geq \limsup_{n \to \infty} P_n (A)  \\
 & \geq \liminf_{n \to \infty}  P_n ( A ) \geq \liminf_{n \to \infty} P_n ( \inte A ) \geq P ( \inte A ).
 \end{align*}
 If $P ( \partial A ) = 0$, then the extreme terms have the same value, and we obtain (v).
 
 Finally, we show that (v) implies (i).
 Take $f$ bounded and continuous; assume without loss of generality that $0 < f(x) < 1$ for all $x$. 
Let $t \geq 0$.
Note that $\{ x \in S : f(x) > t \}^\rc = \{ x \in S : f(x) \leq t \}$,
and, since $f$ is continuous,
 $\cl \{ x \in S : f(x) > t \} \subseteq \{ x \in S : f(x) \geq t \}$.
Hence 
\[ \partial \{ x \in S : f(x) > t \} \subseteq \{ x \in S : f(x) = t \}. \]
Here we have that $P ( \{ x \in S : f(x) = t \} ) =0$ except for countably many $t$.
To see this, consider $\{t:P(\{ x \in S : f(x)=t\})\in (1/(n+1),1/n]\}$ for each $n \in \N$. The number of elements in each of these sets must be finite, or the law of total probability is contradicted, and thus we can label the set of $t$ starting with those in the set with $n=1$, then $n=2$, and so on, hence there are only countably many of such $t$. 

Using the short-hand $\{ f > t \} = \{ x \in S : f(x) > t \}$, we have by Fubini's theorem that
\[ \int_S f \ud P_n = \int_S \int_0^1 \ind_{\{ f > t \}} \ud t \ud P_n = \int_0^1 \int_S \ind_{\{ f > t \}} \ud P_n \ud t = \int_0^1 P_n ( \{ f > t \} ) \ud t,\]
and then by (v) and the bounded convergence theorem, we obtain
 \begin{align*}
 \int_S f \ud P_n    =   \int_0^{1} P_n ( \{ f > t \} ) \ud t  \to \int_0^1 P ( \{ f > t \} ) \ud t  = \int_S f \ud P  ,
 \end{align*}
which completes the proof.
\end{proof}

Another useful consequence of the Portmanteau theorem is the following characterisation of weak convergence~\cite[Theorem~2.6]{cpm}, which we do state in terms of random variables.
\begin{theorem} \label{thm:subsequence}
	$X_n \Rightarrow X$ if and only if every subsequence $\{X_{n_i}\}$ contains a further subsequence converging weakly to $X$.
\end{theorem}
\begin{proof}
The \lq only if\rq~part is easy: if $X_n \Rightarrow X$, then for any bounded, continuous $f$ we have $\Exp f(X_n) = \int f \ud P_n \to \int f \ud P = \Exp f(X)$,
and then by properties of convergence of real numbers we have that any subsequence of $\Exp f(X_{n_i})=\int f \ud P_{n_i}$ also converges
to $\int f \ud P = \Exp f(X)$, i.e., $X_{n_i} \Rightarrow X$.

For the \lq if\rq~part, we prove the contrapositive. Suppose that $X_n \nRightarrow X$, then $\Exp f(X_n) = \int f \ud P_n \not\to \int f \ud P = \Exp f(X)$ for some bounded, continuous $f$. 
We then have that for some subsequence $n_i$ of $\N$ and some $\varepsilon > 0$, $|\Exp f(X_{n_i})-\Exp f(X)| = | \int f \ud P_{n_i}  - \int f \ud P | > \varepsilon$ for all $i$, so that $X_{n_i}$ has no weakly convergent subsequence.
\end{proof}

Here we also note Slutsky's result (see e.g.~\cite[Theorem~3.1]{cpm}) stated in the context of weak convergence.

\begin{theorem}
\label{thm:slutsky}
Suppose that $X, X_1, Y_1, X_2, Y_2, \ldots$ are random variables on a probability space $(\Omega, \calF, \Pr)$
taking values in metric measure space $(S, \calS, \rho)$. If $X_n \Rightarrow X$ and $\rho (X_n, Y_n) \toP 0$,
then $Y_n \Rightarrow X$.
\end{theorem}

\subsection{Proof of the mapping theorem}

The Portmanteau theorem is enough for us to prove the mapping theorem.

\begin{proof}[Proof of Theorem~\ref{thm:mapweak}.]
	Given that $P_n \Rightarrow P$, it follows that for any $F \in S'$,
	\begin{equation}\label{eqn:map2}
	\limsup_{n\to\infty} P_n \left(  h^{-1}F\right)   \leq \limsup_{n\to\infty} P_n \left( \cl ( h^{-1} F ) \right)  \leq P \left( \cl ( h^{-1} F ) \right) ,
	\end{equation}
	by the equivalence of parts (i) and (iii) of the Portmanteau theorem, Theorem~\ref{thm:port}. Also, let $F \in \calS'$ be closed; then, since $h$ is measurable, $h^{-1} F \in \calS$.
	If $x \in \cl ( h^{-1} F )$, then there exist $x_n \in h^{-1}F$ such that $\rho(x_n,x)\to 0$. 
	Since $h ( x_n) \in F$, we have $h(x_n) \to h ( x ) \in \cl F = F$ if $h$ is continuous at $x$. 
	We therefore have
	\begin{equation}\label{eqn:map1}
	D_h^\rc \cap \cl (h^{-1}F ) \subseteq h^{-1} F.
	\end{equation} 
	Combining \eqref{eqn:map1} and \eqref{eqn:map2} gives
	\begin{align*}\limsup_{n\to\infty} P_n \left(  h^{-1}F\right)
	\leq P(\cl (h^{-1} F)) = P \left(D^\rc_h \cap \cl ( h^{-1} F ) \right)  \leq P( h^{-1} F ),
	\end{align*}
	since $P ( D_h^\rc ) = 1$. This holds true for all closed F, thus another application of parts (i) and (iii) of  the Portmanteau theorem yields weak convergence of $P_n h^{-1}$ to $P h^{-1}$.
\end{proof}

\subsection{Weak convergence conditions for continuous trajectories}

\label{sec:donsker-proof}

In order to show weak convergence in the case of $\calC^d$ we need to show a collection of probability measures on $\calC^d$ is \textit{relatively compact} for which we have the following definition stated for an arbitrary measure space.
\begin{description}
\item
[\namedlabel{ass:relCompact}{\textbf{RC}}]
A collection of probability measures $\Pi$ on $(S, \mathcal{S})$ is called relatively compact if for every sequence $P_n$ of elements of $\Pi$, there exists a weakly convergent subsequence $P_{n_m}$.
\end{description}

We say that \eqref{ass:relCompact} holds for random variables $X_1,X_2,\ldots$ if \eqref{ass:relCompact} holds for probability measures $P_1,P_2,\ldots$ and the random variables and probability measures are associated as described at \eqref{eq:measure-rv}.

Considering Theorem \ref{thm:subsequence}, it seems that the two concepts of relative compactness and convergence of finite-dimensional distributions would be sufficient to determine weak convergence. The following result confirms that this is in fact the case. We state the result for random variables, the result for probability measures can be found as Example 5.1 from \cite{cpm}.

\begin{theorem}\label{thm:rc}
For elements $X,X_1,X_2,\ldots$ of $\calC^d$, if \eqref{ass:finiteDimDist} and \eqref{ass:relCompact} hold, then $X_n \Rightarrow X$. \end{theorem}

\begin{proof}
	By \eqref{ass:relCompact} we have that any subsequence $X_{n_m}$ has a further subsequence $X_{n_{m_i}}$ such that $X_{n_{m_i}}\Rightarrow Y$ for some random variable $Y$, possibly depending on the subsequences chosen. Then the mapping theorem implies $\pi_{t_1, \dots, t_k} X_{n_{m_i}} \Rightarrow \pi_{t_1, \dots, t_k} Y$. But by \eqref{ass:finiteDimDist}, we have $\pi_{t_1, \dots, t_k}X_{n_{m_i}} \Rightarrow \pi_{t_1, \dots, t_k}X$, so $\pi_{t_1, \dots, t_k} X$ has the same distribution as $\pi_{t_1, \dots, t_k} Y$. Since the class of finite-dimensional sets is a separating class for $\calC^d$, see \cite[p.~12]{cpm}, this implies that $X$ and $Y$ have the same distribution, and since the subsequences were arbitrary, we have that all such subsequences contain a further subsequence which weakly converges to $X$. By the \lq only if\rq~statement in Theorem \ref{thm:subsequence}, we complete the proof. 
\end{proof}

It is difficult to prove relative compactness directly; however, a more convenient condition that we can work with and which implies relative compactness in certain spaces is \emph{tightness}.
For a family of probability measures tightness is defined as follows.
\begin{description}
\item 
[\namedlabel{ass:tight}{\textbf{T}}]
A family $\Pi$ of probability measures on metric measure space $(S, \calS,\rho)$
is called \emph{tight} if for every $\varepsilon > 0$ there exists a compact $K \in \calS$ such that for all $P \in \Pi$, $P ( K ) > 1 - \varepsilon$.
\end{description}

Again, we use the terminology in the natural way for random variables: a collection $(X_\alpha, \alpha \in I)$
of random variables on a probability space $(\Omega,\calF,\Pr)$ and taking values in a metric measure space $(S,\calS,\rho)$ is tight if the collection of probability measures $(P_\alpha, \alpha \in I)$, defined by $P_\alpha ( B) = \Pr ( X_\alpha \in B )$ for $B \in \calS$, is tight. 

To formalise the statement tightness implies relative compactness we state the following theorem of Prokhorov \cite[Theorems~5.1~\&~5.2]{cpm}.
\begin{theorem}[Prokhorov's theorem]\label{thm:Prok}
\eqref{ass:tight} implies \eqref{ass:relCompact}. If $S$ is separable and complete, and $\Pi$ satisfies \eqref{ass:relCompact}, then $\Pi$ also satisfies \eqref{ass:tight}.
\end{theorem}
\noindent Here we only need the implication \eqref{ass:tight} implies \eqref{ass:relCompact}, however note that Theorem \ref{thm:Cdsep} tells us $\calC^d$ is separable and complete, so we do indeed have that tightness and relative compactness are equivalent in this space.

Instead of replicating the full proof of Billingsley here \cite[pp.~59--63]{cpm}, we only give an outline of the proof of the first statement, the proof of the second is brief so we do provide that here.

\begin{proof}
	Using the tightness, one can construct a sequence of increasing compact sets which cover deterministically large amounts of the probability mass for all the probability measures $P_n$. Then a measure theory result states that we can use this sequence to construct a countable class of sets for which any element of an arbitrary open set $G$ must lie in one of these sets. Taking the $\sigma$-algebra of the compact sets and these countable sets, we get a countable class of compact sets which contain good approximating sets of the arbitrary set $G$, we will call this class $\calH$.
	
	Now, since the class was countable, a Cantor diagonal method allows us to be sure that there exists a subsequence $P_{n_i}$ for which $\lim_{i\rightarrow \infty} P_{n_i} H$ exists for all $H \in \calH$. Then we will try to find a probability measure $P$ such that 
	\[ P(G)= \sup_{H\subset G} \lim_{i\rightarrow \infty} P_{n_i} H.\]
	If this was true, then since the supremum is over $H\subset G$ we have $P(G) \leq \liminf_{i\rightarrow \infty} P_{n_i} G$, which is condition (iv) of the Portmanteau theorem, Theorem~\ref{thm:port} so we have $P_{n_i} \Rightarrow P$ as desired. The proof that such a measure exists can be found at \cite[pp.~61--63]{cpm}, we move on to the reverse implication.
	
	Consider a non-decreasing sequence of open sets $G_n$ with $\lim_{n\rightarrow \infty} G_n = S$. For each $\varepsilon$, there exists an $n$ for which $P(G_n)>1-\varepsilon$ for all $P \in \Pi$, otherwise the relative compactness assumption would mean the limit of this subsequence of bad measures is the whole space but with non-total probability. 
	
	Now consider a sequence of open balls $A_{k_1},A_{k_2},\ldots$ with radius $1/k$ which cover $S$, and take $n_k$ such that $P(\cup_{i\leq n_k}  A_{k_i}) > 1- 2^{-k}\varepsilon$ for all $P \in \Pi$ which we can do by the previous fact. Then by completeness of $S$, there exists a compact set $K \in S$ defined by 
	\[K= \cap_{k\geq 1} \cup_{i\leq n_k} A_{k_i}, \]
	with $P(K)>1-\varepsilon$ for all $P \in \Pi$, hence tightness holds.
\end{proof} 

\begin{corollary}
\label{thm:tight}
For elements $X,X_1,X_2,\ldots$ of $\calC^d$, if \eqref{ass:finiteDimDist} and \eqref{ass:tight} hold, then $X_n \Rightarrow X$.
\end{corollary}

\subsection{Tightness conditions for continuous trajectories}

Having proven that tightness is sufficient, we need to find a way to prove the tightness holds. In order to do this, we first need to state the Arzel\`{a}-Ascoli theorem in $d$-dimensions. The proof at \cite[Theorem~7.25]{rudin} is not dimension dependent so carries across. Recall a subset $A$ of a topological subspace is relatively compact if it has a compact closure.
\begin{theorem}\label{thm:AAinCd}
	A set $A$ in $\calC^d$ is relatively compact if and only if \[\sup_{f \in A}\|f(0)\|<\infty {\rm \quad and \quad} \lim_{\delta\rightarrow 0} \sup_{f\in A} w_f(\delta)=0.\]
\end{theorem}

This allows us to generalise the conditions for tightness at \cite[Theorem~7.3]{cpm} to $d$-dimensions.

\begin{lemma}
	\label{lem:tightC}
	Let $P_n$ be a sequence of probability measures on $\calC^d$. Then \eqref{ass:tight} holds if and only if the following two conditions hold.
	\begin{enumerate}
		\item[(i)] We have
		\begin{equation}
		\lim_{a\rightarrow \infty}\limsup_{n\to\infty} P_n ( \{ f :\| f(0) \| \geq a \} )=0.\label{eqn:tightCone}
		\end{equation}
		\item[(ii)] For each $\varepsilon>0$,
		\begin{equation}
		\lim_{\delta \downarrow 0 }\limsup_{n \to \infty}  P_n \left( \{ f : w_f (\delta) \geq \varepsilon \} \right) = 0.
		\label{eqn:tightCtwo}
		\end{equation}
	\end{enumerate}
\end{lemma}

\begin{proof}
	For the \lq only if\rq~case, given some $\gamma>0$, consider a compact $K$ such that $P_n(K)>1-\gamma$ for all $n$; such a $K$ exists by the tightness. Since $K$ is compact, Theorem~\ref{thm:AAinCd} tells us that $\sup_{f\in K}\|f(0)\|<\infty$ so $K \subseteq \{f:\|f(0)\|\leq a\}$ for a large enough choice of $a$. Further, $\lim_{\delta\rightarrow 0} \sup_{f\in K} w_f(\delta)=0$ so for a small enough choice of $\delta$, $K \subseteq \{f: w_f(\delta)\leq \varepsilon\}$. These two facts imply \eqref{eqn:tightCone} and \eqref{eqn:tightCtwo} respectively.
	
	For the reverse implication, we start by recalling Theorem~\ref{thm:Cdsep} which says that $\calC^d$ is separable and complete under $\rho_\infty$. Noting that a single measure clearly satisfies \eqref{ass:relCompact}, it follows from Prokhorov's theorem, Theorem~\ref{thm:Prok} that a single measure is tight. Then, using the \lq only if\rq~part of this lemma, for a fixed probability measure $P$, and a given $\gamma>0$ there is an $a$ such that $P(\{f:\|f(0)\|\geq a\})\leq \gamma$, and for a given $\varepsilon$ and $\gamma$ there is a $\delta$ such that $P(\{f: w_f(\delta)\geq\varepsilon\})\leq \gamma$.
	
	If we have \eqref{eqn:tightCone} and \eqref{eqn:tightCtwo}, then there exists a finite $n_0$ such that, for all $n> n_0$, 
	\begin{equation}\label{eqn:prob1} P_n(\{f : \|f(0)\|\geq a\}) \leq \gamma,
	\end{equation} holds for some large enough $a$ and 
	\begin{equation}\label{eqn:prob2}P_n(\{f:w_f(\delta)\geq \varepsilon\})\leq \gamma,
	\end{equation} holds for some small enough $\delta$. Then, for each of the finitely many measures $P_1,P_2,\ldots,P_{n_0}$ we have tightness so \eqref{eqn:prob1} and \eqref{eqn:prob2} still hold for these measures, possibly requiring a larger choice of $a$ or smaller choice of $\delta$. Using this, we can assume there exists some $a$ and some $\delta$ for which \eqref{eqn:prob1} and \eqref{eqn:prob2} hold for all $n$.
	
	Using this assumption, given $\gamma$, we can choose $a$ and $\delta_k$ such that the sets $B=\{f:\|f(0)\|\leq a\}$ and $B_k=\{f:w_f(\delta_k)<1/k\}$ have probabilities $P_n(B)\geq 1-\gamma$ and $P_n(B_k) \geq 1-\gamma 2^{-k}$ for all $n$. Consider the set $K=\cl (B \cap (\cap_{k\geq 1} B_k ))$ which has $P_n(K)\geq 1-\gamma - \gamma 2^{-k} \geq 1-2\gamma$ for all $n$. This closed set satisfies both conditions of Theorem~\ref{thm:AAinCd}, so it is compact, hence the $\{P_n\}$ are tight.
\end{proof}

The next ingredient we need is a theorem bounding the modulus of continuity which is the $d$-dimensional equivalent to \cite[Theorem~7.4]{cpm}.

\begin{theorem}
	Suppose that $0=t_0<t_1<\ldots<t_k=1$ and $\min_{1<i<k}(t_i-t_{i-1})\geq \delta$. Then, for arbitrary $f\in \calC^d$,
	\begin{equation} \label{eqn:modContBound} w_f(\delta)\leq 3 \max_{1\leq i \leq k} \sup_{t_{i-1}\leq s \leq t_i}\|f(s)-f(t_{i-1})\|,
	\end{equation}
	and, for any probability measure $P$ on $\calC^d$,
	\begin{equation}\label{eqn:probModCont}P\{f: w_f(\delta)\geq 3 \varepsilon\} \leq \sum_{i=1}^k P \left\{ f: \sup_{t_{i-1}\leq s \leq t_i} \|f(s)-f(t_{i-1})\|\geq \varepsilon \right\}.
	\end{equation}
\end{theorem}

\begin{proof}
	Let $m$ be the maximum in \eqref{eqn:modContBound}. If $s$ and $t$ lie in the same interval $I_i = [t_{i-1},t_i]$, then $\|f(s)-f(t)\|\leq \|f(s)-f(t_{i-1})\|+\|f(t)-f(t_{i-1})\| \leq 2m$. If $s$ and $t$ lie in adjacent intervals $I_i$ and $I_{i+1}$, then $\|f(s)-f(t)\| \leq \|f(s)-f(t_{i-1})\|+\|f(t_{i-1})-f(t_i)\|+\|f(t_i)-f(t)\| \leq 3m$. If $|s-t| \leq \delta$ then $s$ and $t$ must either lie in the same interval, or adjacent ones, which proves \eqref{eqn:modContBound}. The second statement follows by Boole's inequality.
\end{proof}

Next we present a lemma that gives a sufficient condition for tightness in $\calC^d_0$.

\begin{lemma}
	\label{lem:bill}
	Suppose that we have a random walk as defined at \eqref{ass:walk}, and define $Y_n$ as at~\eqref{eqn:trajectories}.
	Then a sufficient condition for $\{Y_n:n\in \N\}$ to be tight is 
	\begin{align}
	\label{eqn:lemma}
	\lim_{\lambda \to \infty} \limsup_{n \to \infty} \lambda^2 \Pr \left( \max_{0 \leq j \leq n} \| S_j \| \geq \lambda \sqrt{n} \right) = 0.
	\end{align}
\end{lemma}
\begin{proof} 
	We will show the two conditions in Lemma~\ref{lem:tightC} hold. The first, \eqref{eqn:tightCone}, clearly holds, since $Y_n(0)=0$. For the second condition, we use the bound in \eqref{eqn:probModCont}. In particular, we take $t_i = m_i/n$ for integers $m_i$ satisfying $0=m_0<m_1<\ldots<m_k=n$. Then the supremum in \eqref{eqn:probModCont} becomes a maximum of differences as follows, 
	\[ \Pr\left( w_{Y_n}(\delta)\geq 3\varepsilon\right) \leq \sum_{i=1}^k \Pr \left( \max_{m_{i-1}\leq j\leq m_i} \frac{\|S_j-S_{m_{i-1}}\|}{\sqrt{n}}\geq \varepsilon \right) = \sum_{i=1}^k \Pr \left( \max_{0\leq j\leq m_i-m_{i-1}} \|S_j\| \geq \varepsilon \sqrt{n}\right),\]
	where the equality is due to the identical distribution of the increments. For this to hold, of course we need the choice of $m_i$ to satisfy the condition $\min_{1<i<k} (m_i - m_{i-1})n^{-1} \geq \delta$. We can further simplify this choice by taking $m_i=im$ for each $i<k$ and some $m>1$. In order to satisfy the criterion we take $m = \lceil n\delta \rceil$. By this choice, we naturally fix $k=\lceil n/m \rceil$, with $m_k=n$. Note that this means, for large enough $n$, $|k- \delta^{-1}| \leq 1$, so for large enough $n$ and $\delta<1$, we have $k < 2\delta^{-1}$. Also, for large enough $n$, $|n/m- \delta^{-1}|<1$ so for large enough $n$ and $\delta<1/2$, we have $n>m/2\delta$. Using these inequalities, we have, for large enough $n$ and small enough $\delta$,
	\[ \Pr\left( w_{Y_n}(\delta)\geq 3\varepsilon\right) \leq \sum_{i=1}^k \Pr \left( \max_{0\leq j\leq m_i-m_{i-1}} \|S_j\| \geq \varepsilon \sqrt{n}\right) \leq \frac{2}{\delta} \Pr\left( \max_{0\leq j\leq m}\|S_j\| \geq \frac{\varepsilon\sqrt{m}}{\sqrt{2\delta}}\right). \]
	If we now take $\lambda= \varepsilon/\sqrt{2\delta}$, we get,
	\[\limsup_{n\rightarrow \infty}\Pr\left( w_{Y_n}(\delta)\geq 3\varepsilon\right) \leq \frac{4\lambda^2}{\varepsilon^2} \limsup_{m\rightarrow \infty}\Pr\left( \max_{0\leq j\leq m}\|S_j\| \geq \lambda\sqrt{m}\right). \]
	Now, under the suggested condition \eqref{eqn:lemma}, for a fixed $\varepsilon$ and any $\gamma>0$, there exists a $\lambda$ such that
	\[ \frac{4\lambda^2}{\varepsilon^2} \limsup_{m\rightarrow \infty}\Pr\left( \max_{0\leq j\leq m}\|S_j\| \geq \lambda\sqrt{m}\right) < \gamma. \]
	Fixing $\varepsilon$ and a large enough $\lambda$ means fixing $\delta$ to be small enough. The second condition in Lemma~\ref{lem:tightC} follows, and the proof is complete. 
\end{proof}

We state the following estimate separately because it will be useful in the proof of both parts of Theorem~\ref{thm:Donsker-dd}, not just in the continuous case.

\begin{lemma}
	\label{lem:rwmax}
	Suppose that we have a random walk as defined at \eqref{ass:walk} with $\mu =0$, and satisfying \eqref{ass:Sigma} with $\Sigma = I_d$.
	Then
	there exists a constant $C \in \RP$ such that for all $k \in \N$ and all $\lambda \geq 0$,
	\[ \limsup_{n \to \infty}  \Pr \left( \max_{0 \leq j \leq \lfloor n/ k \rfloor} \| S_j \| \geq \lambda \sqrt{n} \right) \leq C k^{-2} \lambda^{-4} .\]
\end{lemma}
\begin{proof}
	Let $Z\sim \calN (0,I_d)$. Then by Markov's inequality
	there is a constant $C \in \RP$ depending only on $d$ such that, for all $a \geq 0$,
	\begin{equation}
	\label{eq:markov}
	\Pr \left( \| Z \| \geq \frac{a}{3} \right) = 
	\Pr \left(\|Z\|^4 \geq \left( \frac{a}{3}\right)^4\right) \leq C a^{-4} .\end{equation}
	We apply the $d$-dimensional version of Etemadi's inequality (see Lemma~\ref{lem:etemadi}) to obtain, for $\lambda \geq 0$,
	\begin{align*}
	\Pr \left( \max_{0 \leq j \leq \lfloor n/k \rfloor} \| S_j \| \geq \lambda \sqrt{n} \right) \leq 3  \max_{0 \leq j \leq \lfloor n/k \rfloor} \Pr \left( \|  S_j \| \geq \frac{\lambda \sqrt{n}}{3} \right).
	\end{align*}
	Now for any $n_0 \in \N$,
	\begin{align*}
	\max_{0 \leq j \leq \lfloor n/k \rfloor} \Pr \left( \|  S_j \| \geq \frac{\lambda \sqrt{n}}{3} \right) 
	& \leq \max_{0 \leq j \leq n_0} \Pr \left(  \| S_j \| \geq \frac{\lambda \sqrt{n}}{3} \right)+ 
	\max_{ n_0 \leq j \leq \lfloor n/k \rfloor } \Pr \left(  \| S_j \| \geq \frac{\lambda \sqrt{n}}{3} \right) \\
	& \leq \max_{0 \leq j \leq n_0} \Pr \left(  \| S_j \| \geq \frac{\lambda \sqrt{n}}{3} \right)+ 
	\max_{ n_0 \leq j \leq \lfloor n/k \rfloor} \Pr \left( j^{-1/2} \| S_j \| \geq \frac{\lambda \sqrt{n/j}}{3} \right)\\
	& \leq \max_{0 \leq j \leq n_0} \Pr \left(  \| S_j \| \geq \frac{\lambda \sqrt{n}}{3} \right)+ 
	\max_{n_0 \leq j \leq \lfloor n/k \rfloor} \Pr \left( j^{-1/2} \| S_j \| \geq \frac{\lambda \sqrt{k}}{3} \right) .
	\\
	& \leq \max_{0 \leq j \leq n_0} \Pr \left(  \| S_j \| \geq \frac{\lambda \sqrt{n}}{3} \right)+ 
	\max_{j \geq n_0} \Pr \left( j^{-1/2} \| S_j \| \geq \frac{\lambda \sqrt{k}}{3} \right) .
	\end{align*}
	Now if we consider Theorem~\ref{thm:CLT} in conjunction with Theorem~\ref{thm:polya}, and the $a = \lambda \sqrt{k}$ case of~\eqref{eq:markov}, then we can choose $n_0$ sufficiently large so that for all $k \in \N$ and all $\lambda \geq 0$,
	\[ \max_{j \geq n_0} \Pr \left( j^{-1/2} \| S_j \| \geq \frac{\lambda \sqrt{k}}{3} \right)  \leq 2 C k^{-2} \lambda^{-4} .\]
	Therefore,
	\begin{align*}
	\limsup_{n \rightarrow \infty} \left\{  
	\max_{0 \leq j \leq n_0} \Pr \left(  \| S_j \| \geq \frac{\lambda \sqrt{n}}{3} \right)
	+   \max_{j \geq n_0} \Pr \left( j^{-1/2} \| S_j \| \geq \frac{\lambda \sqrt{k}}{3} \right) \right\} \leq  2 C k^{-2} \lambda^{-4},
	\end{align*}
	which gives the claimed result.
\end{proof}

\subsection{Donsker's theorem for \texorpdfstring{$d$}{}-dimensional continuous trajectories - proof}
Now we are ready to complete the statement that the measures associated with trajectories in $\calC_0^d$ are tight, so we must turn our attention to showing that the finite-dimensional distributions do in fact converge to those of Brownian motion.
The following lemma will again be useful for both the continuous and discontinuous cases, hence we state it as a separate result.

\begin{lemma}
\label{lem:fdd}
Suppose that we have a random walk as defined at \eqref{ass:walk} with $\mu =0$, and satisfying \eqref{ass:Sigma} with $\Sigma = I_d$. 
Then for any $0 \leq t_1 < t_2 < \cdots < t_k \leq 1$, we have that as $n \to \infty$,
\[ n^{-1/2} \left( S_{\lfloor n t_1 \rfloor} , S_{\lfloor n t_2 \rfloor} - S_{\lfloor n t_1 \rfloor} , \ldots, S_{\lfloor n t_k \rfloor} - S_{\lfloor n t_{k-1} \rfloor} \right)
\tod \left( b_d (t_1) , b_d(t_2) - b_d(t_1) , \ldots, b_d(t_k) - b_d(t_{k-1} ) \right) .\]
\end{lemma}
\begin{proof}
The idea is contained already in the case $k=2$, so for simplicity we present that case here.
By the Markov property, $S_{\lfloor nt_2 \rfloor} - S_{\lfloor nt_1 \rfloor}$ and $S_{\lfloor nt_1 \rfloor}$ are independent.
By the multidimensional
central limit theorem, Theorem~\ref{thm:CLT}, we have
\[ \frac{1}{\sqrt{n}} S_{\lfloor nt_1 \rfloor} = \left( \frac{ \sqrt{\lfloor nt_1 \rfloor}}{\sqrt{n}} \right) \frac{1}{\sqrt{\lfloor nt_1 \rfloor}} S_{\lfloor nt_1 \rfloor} 
\tod t_1^{1/2} Z_1 ,\]
where $Z_1 \sim \calN (0, I_d)$, 
using the fact that, if $\alpha_n \to \alpha$ in $\R$ and $\zeta_n \tod \zeta$ in $\R^d$, then $\alpha_n \zeta_n \tod \alpha \zeta$ in $\R^d$.
Similarly, 
\[ \frac{1}{\sqrt{n}} \left( S_{\lfloor nt_2 \rfloor} -  S_{\lfloor nt_1 \rfloor} \right) \tod (t_2 - t_1)^{1/2} Z_2 ,\]
where $Z_2 \sim \calN (0,I_d)$. Here $Z_2$ is independent of $Z_1$ because if $X_n \tod X$ and $Y_n \tod Y$, and $X_n$ and $Y_n$ are pairwise independent, then $(X_n,Y_n)\tod (X,Y)$ where $(X,Y)$ are independent.
\end{proof}

Now we can complete the proof of part (a) of Donsker's theorem.

\begin{proof}[Proof of Theorem~\ref{thm:Donsker-dd}(a)]
We follow~\cite[\S 8]{cpm}, and aim to apply Corollary~\ref{thm:tight}.
Recall from Remark~\ref{rem:donsker} that it suffices to consider the case where $\Sigma = I_d$.

 First we must establish convergence of the finite-dimensional distributions of $Y_n$. 
We need to show that for any $0 \leq t_1 < t_2 < \cdots < t_k \leq 1$ we have
\[ \left( Y_{n} (t_1) , Y_{n} (t_2) , \ldots, Y_{n} ( t_k ) \right) \tod ( b (t_1) , b(t_2 ) , \ldots, b(t_k) ) .\]
By continuity of the function $(x_1,x_2,\ldots,x_k) \mapsto \left(x_1, x_1+x_2,\ldots,\sum_{i=1}^k x_i\right)$, it is sufficient to prove that
\[ ( Y_{n} (t_1), Y_n (t_2) - Y_n (t_1 ) , \ldots, Y_n (t_k) - Y_n (t_{k-1} ) ) \tod ( b (t_1) , b(t_2) - b(t_1) ,\ldots, b(t_k) - b(t_{k-1} ) ) .\] 
Lemma~\ref{lem:fdd} provides the main step here,
but there is a little more work due to the definition of $Y_n$ in terms of interpolation. Again, the main idea is contained in the case $k=2$
so we describe only that case here.
Let $0 \leq t_1 < t_2 \leq 1$.
Using~\eqref{eqn:trajectories} we may write
\begin{align*}
\left( Y_n(t_2), Y_n(t_2) - Y_n(t_1) \right) &= \frac{1}{\sqrt{n}} \left( S_{\lfloor nt_1 \rfloor},    S_{\lfloor nt_2 \rfloor} - S_{\lfloor nt_1 \rfloor}  \right)  + \left( \psi_{n,t_1}, \psi_{n,t_2} - \psi_{n,t_1} \right),
\end{align*}
where $\psi_{n,t} := \frac{nt- \lfloor nt \rfloor}{\sqrt{n}} \xi_{\lfloor nt \rfloor + 1}$. 
Using Markov's inequality, we have that for $r>0$,
\begin{align*}
\Pr(\|\xi \| \ge r ) \le \frac{\Exp [  \| \xi \|^2 ]}{r^2} = \frac{ \mathrm{tr~} \Sigma }{r^2} = \frac{d}{r^2} ,
\end{align*}
since $\mu=0$ and $\Sigma = I_d$. Since $\| \psi_{n,t} \| \leq n^{-1/2} \|  \xi_{\lfloor nt \rfloor + 1} \|$, we get
\begin{align*}
\Pr \left( \| \psi_{n,t} \| > r \right) & \le \Pr \left(\|\xi_{\lfloor nt \rfloor +1}\| \ge r \sqrt{n} \right) \leq \frac{d}{r^2 n} .
\end{align*}
It follows that $\psi_{n,t_1} \toP 0$, and similarly for $\psi_{n,t_2} - \psi_{n,t_1}$. Hence  $(\psi_{n,t_1}, \psi_{n,t_2} - \psi_{n,t_1}) \toP 0$.
Thus by Lemma~\ref{lem:fdd} and Theorem~\ref{thm:slutsky}, we get
\[ \left( Y_n(t_2), Y_n(t_2) - Y_n(t_1) \right) \tod \left( t_1^{1/2} Z_1 , (t_2 - t_1)^{1/2} Z_2 \right) ,\]
which is exactly the distribution of $(b(t_1), b(t_2) - b(t_1))$, as required.

Next we use Lemma~\ref{lem:bill} to establish tightness. The $k=1$ case of Lemma~\ref{lem:rwmax} shows that
\[ \limsup_{n \to \infty} \lambda^2  \Pr \left( \max_{0 \leq j \leq n} \| S_j \| \geq \lambda \sqrt{n} \right) \leq C \lambda^{-2} ,\]
which converges to $0$ as $\lambda \to \infty$. Thus Lemma~\ref{lem:bill} gives tightness, and Theorem~\ref{thm:tight} completes the proof of part (a) of Theorem~\ref{thm:Donsker-dd}. 
\end{proof}

\subsection{Weak convergence conditions in the Skorokhod topology}
Now we turn to part (b) of Theorem~\ref{thm:Donsker-dd}. 

The first difference for trajectories with discontinuities is that the spaces $\calD$ and $\calD^d$ do not automatically have the class of finite-dimensional sets as a separating class. This means the proof of Theorem \ref{thm:rc} does not translate to this setting. However, we extract the following result from Theorem 12.5 of \cite{cpm} which will help us.

\begin{theorem}\label{Thm:FDDSepClass}
	Let $T\subseteq [0,1]$ with $1\in T$ such that $T$ is dense in $[0,1]$, then the class of finite-dimensional sets taking values in $T$ is a separating class of $\calD^d$.
\end{theorem}

To prove this result we recall, without proof, some standard results from measure theory, see e.g. \cite[Theorem~A.1.4]{durrett}.
\begin{definition}
	Any non-empty collection of sets $\calP$ is a $\pi$-system if for any $A,B \in \calP$, then $A \cap B \in \calP$.
\end{definition}

\begin{theorem} \cite[Theorem~3.3]{billpm} \label{thm:pisystem}
	Suppose that $P_1$ and $P_2$ are probability measures on $\sigma(\calP)$, where $\calP$ is a $\pi$-system and $\sigma(\calP)$ is the $\sigma$-algebra generated by $\calP$. If $P_1$ and $P_2$ agree on $\calP$ then they agree on $\sigma(\calP)$.
\end{theorem}
We omit the proof of this result because it would require a considerable diversion into Dynkin's $\pi - \lambda$ theorem which is already well covered ground in the literature, see \cite[Theorem~3.2]{billpm}.

\begin{proof}[Proof of Theorem~\ref{Thm:FDDSepClass}]
	For the duration of this proof, let $\calB$ denote the Borel subsets of $(\calD^d,\rho_S)$,
	and recall that $\calB_d$ denotes the Borel subsets of $\R^d$.
	Let $\calC$ denote the finite cylinder sets over $T$, that is, the collection of all subsets of $\calD^d$ of
	the form
	\begin{equation}
	\label{eq:cylinder}
	\left\{ f \in \calD^d : \pi_{t_0, t_1, \ldots, t_k} f \in \prod_{i=1}^k A_i \right\} ,\end{equation}
	where $k \in \ZP$, $t_1, \ldots, t_k \in T$,  and $A_1, A_2, \ldots, A_k \in \calB_d$.
	If $C_1, C_2 \in \calC$ are of  the form~\eqref{eq:cylinder} with $k = k_1, k_2$ respectively,
	then $C_1 \cap C_2$ is also a set of  the form~\eqref{eq:cylinder} with $k = k_1 + k_2$. Thus
	$\calC$ is a $\pi$-system. It generates the $\sigma$-algebra $\sigma (\calC)$.
	
	By the assumption that $T$ is dense, there is a sequence $t_1 > t_2 > \cdots$ of elements of $T$ such that $t_n \downarrow 0$ as $n \to \infty$,
	and then any $f \in \calD^d$ has $\pi_0 f = \lim_{n \to \infty} \pi_{t_n} f$ by right continuity.
	Hence $\pi_0 = \lim_{n \to \infty} \pi_{t_n}$ pointwise, and so $\pi_0$ is a limit of  functions measurable with respect to $\sigma (\calC)$,
	and hence is itself measurable with respect to $\sigma (\calC)$.
	Thus we may assume that $0 \in T$. Then, for a given $m \in \N$, choose a positive integer $k$ and points $s_0, s_1, \ldots, s_k$ of $T$
	such that $0 = s_0 < \cdots < s_k =1$ and $\max_{1 \leq i \leq k} (s_i - s_{i-1} ) < m^{-1}$. For $\alpha = (\alpha_0, \ldots, \alpha_k)$ in $(\R^d)^{k+1}$,
	let $V_m \alpha$ be the element of $\calD^d$ such that $V_m \alpha (t) = \alpha_{i-1}$ for $t \in [s_{i-1},s_i)$ for each $1 \leq i \leq k$,
	and  $V_m \alpha (1) = \alpha_k$.
	Since $V_m : (\R^d)^{k+1} \to \calD^d$ is continuous, it is measurable, i.e., $V_m^{-1} ( B) \in \calB_{d(k+1)}$ for each $B \in \calB$.
	Since $\pi_{s_0,\ldots,s_k}$ is measurable from $(\calD^d, \sigma(\calC))$ to $((\R^d)^{k+1},\calB_{d(k+1)})$,
	the composition $V_m \pi_{s_0,\ldots,s_k}$ is measurable from $(\calD^d, \sigma(\calC))$ to $(\calD^d, \calB)$.
	It is a straightforward exercise to show that $\rho_S (f, V_m \pi_{s_0,\ldots,s_k} f ) \leq \max ( m^{-1}, w_f' (m^{-1} ) )$
	for any $f \in \calD^d$, which implies that $f = \lim_{m \to \infty} V_m \pi_{s_0,\ldots,s_k} f$.
	Hence the identity function on $\calD^d$ is a limit of a sequence of functions measurable from $(\calD^d, \sigma(\calC))$ to $(\calD^d, \calB)$
	and hence is itself measurable from $(\calD^d, \sigma(\calC))$ to $(\calD^d, \calB)$.
	It follows that $\sigma(\calC) = \calB$, i.e., the $\pi$-system $\calC$ generates the full Borel $\sigma$-algebra. Theorem~\ref{thm:pisystem} now completes the proof.
\end{proof}

Now, we can take $T\subseteq [0,1]$ to be the set of continuity points of $X\in \calD^d$, which must contain $1$ by the right continuity of $\calD^d$ and must be dense because the set of discontinuity points has measure $0$. Thus, we have the following replacement of Corollary~\ref{thm:tight}, with the proof now being identical to that of Theorem~\ref{thm:rc}, with the use of Prokhorov's theorem to allow us to claim the result for tightness not relative compactness.

\begin{theorem}\label{thm:FDDandTinD}
	For elements $X, X_1, X_2, \ldots$ of $\calD^d$, if \eqref{ass:finiteDimDist} and \eqref{ass:tight} hold, then $X_n \Rightarrow X$.
\end{theorem}

\subsection{Tightness conditions in Skorokhod topology}

First we need to state a generalised form of the Arzel\`{a}-Ascoli theorem, not only for the Skorokhod topology case, but also in $d$-dimensions. The proof of the Skorokhod case in $1$-dimension was done at \cite[Theorem~12.3]{cpm}, but the proof has no dimensional dependency so we refrain from copying it here. Recall the definition of $w'_f$ from~\eqref{eqn:Dmodulusofcont}.

\begin{theorem}\label{thm:AAinDd}
	A set $A$ in $\calD^d$ is relatively compact if and only if 
	\[\sup_{f \in A}\|f\|_{\infty}<\infty {\rm \quad and \quad} \lim_{\delta\rightarrow 0} \sup_{f\in A} w'_f(\delta)=0.\]
\end{theorem}

Now we can also generalize the tightness conditions of \cite[Theorem~13.2]{cpm} to $d$-dimensions, the proof reads the same as that for Lemma~\ref{lem:tightC} with the modulus of continuity $w_f$ replaced with $w'_f$ so we omit it.

\begin{lemma}
\label{lem:tightD}
Let $P_n$ be a sequence of probability measures on $\calD^d$. Then \eqref{ass:tight} holds if and only if the following two conditions hold.
\begin{enumerate}
\item[(i)] We have
\begin{equation}
\lim_{a\rightarrow \infty}\limsup_{n\to\infty} P_n ( \{ f :\| f \|_\infty \geq a \} )=0.\label{eqn:tightDone}
\end{equation}
\item[(ii)] For each $\varepsilon>0$,
\begin{equation}
\lim_{\delta \downarrow 0 }\limsup_{n \to \infty}  P_n \left( \{ f : w^{\prime}_f (\delta) \geq \varepsilon \} \right) = 0.
\label{eqn:tightDtwo}
\end{equation}
\end{enumerate}
\end{lemma}

\subsection{Donsker's theorem in \texorpdfstring{$d$}{}-dimensional Skorokhod space - proof}

\begin{proof}[Proof of Theorem~\ref{thm:Donsker-dd}(b)]
The convergence of the finite-dimensional distributions is a consequence of Lemma~\ref{lem:fdd} and the continuous mapping theorem, Theorem~\ref{thm:mapping}, which is applicable because the mapping $(x_1,x_2,\ldots,x_k)\mapsto (x_1,x_1+x_2,\ldots,\sum_{i=1}^k x_i)$ defined for $x_1,\ldots,x_k \in \R^d$ is continuous.
For tightness, it will be sufficient to check the
 conditions in Lemma~\ref{lem:tightD} applied to the measures $P_n$ defined by $P_n ( B) = \Pr ( Y'_n \in B )$. 
The  condition~\eqref{eqn:tightDone} then  becomes
\begin{equation}
\lim_{a\to\infty}\limsup_{n\to\infty} \Pr \left( \max_{0 \leq j \leq n} \left\|  S_j \right\|\geq a \sqrt{n} \right) =0,
\nonumber
\end{equation}
which is easily verified by the $k=1$ case of Lemma~\ref{lem:rwmax}.

The condition~\eqref{eqn:tightDtwo} becomes
$$ \lim_{\delta\downarrow 0}\limsup_{n\to\infty} \Pr \left( \inf_{\{t_i\}} 
\max_{1\leq i\leq v}  \sup_{t,s\in [t_{i-1},t_i)}\|Y'_n(s)-Y'_n(t)\|\geq \varepsilon \right) = 0,$$ 
where the infimum is over all $\delta$-sparse sets
$\{t_0, t_1, \ldots, t_v\}$. 
It suffices to suppose $\delta = 1/2k$, with $k \in \N$, and then choose $t_i = i/k$ and $v=k$ to obtain an upper bound for the probability. This gives
\begin{align*}
\Pr \left( \inf_{\{t_i\}} 
\max_{1\leq i\leq v}  \sup_{t,s\in [t_{i-1},t_i)}\|Y'_n(s)-Y'_n(t)\|\geq \varepsilon \right) &
\leq \Pr \left( \max_{1\leq i\leq v}  \sup_{t,s\in [\frac{i-1}{k},\frac{i}{k})}\|Y'_n(s)-Y'_n(t)\|\geq \varepsilon \right) \\
& = \Pr \left( \bigcup_{i=1}^k \left\{  \sup_{t,s\in [\frac{i-1}{k},\frac{i}{k})}\|Y'_n(s)-Y'_n(t)\|\geq \varepsilon \right\} \right) \\
& \leq \sum_{i=1}^k \Pr \left(   \sup_{t,s\in [\frac{i-1}{k},\frac{i}{k})}\|Y'_n(s)-Y'_n(t)\|\geq \varepsilon   \right) . \end{align*}
Here we have $\| Y_n' (s) - Y_n' (t) \| = \sum_{j=\lfloor ns \rfloor +1}^{\lfloor nt \rfloor} \xi_j$ if $s<t$ (and we can restrict the supremum to such $t,s$) so
that the distribution of $\sup_{t,s\in [\frac{i-1}{k},\frac{i}{k})}\|Y'_n(s)-Y'_n(t)\|$ is the same for each $i$. Hence
\begin{align*}
& \lim_{\delta\downarrow 0}\limsup_{n\to\infty} \Pr \left( \inf_{\{t_i\}} 
\max_{1\leq i\leq v}  \sup_{t,s\in [t_{i-1},t_i)}\|Y'_n(s)-Y'_n(t)\|\geq \varepsilon \right) \\
&{} \qquad \qquad \qquad {} \leq \lim_{k \to \infty} \limsup_{n \to \infty} k 
 \Pr \left(   \sup_{t,s\in [0,\frac{1}{k})}\|Y'_n(s)-Y'_n(t)\|\geq \varepsilon   \right) . \end{align*}
Here we have that 
\begin{align*}
\Pr \left(   \sup_{t,s\in [0,\frac{1}{k})}\|Y'_n(s)-Y'_n(t)\|\geq \varepsilon   \right) & \leq \Pr \left(   \sup_{t,s\in [0,\frac{1}{k})}( \|Y'_n(s)\| + \|Y'_n(t)\| ) \geq \varepsilon   \right) \\
&=  \Pr \left(   \sup_{t \in [0,\frac{1}{k})} \|Y'_n(t)\|  \geq \varepsilon /2   \right) \\
& = \Pr \left( \max_{0 \leq j \leq \lfloor n/k \rfloor} \| S_j \| \geq (\eps/2) \sqrt{n} \right) .
\end{align*}
Then by Lemma~\ref{lem:rwmax} we have that
\[\lim_{k \to \infty} \limsup_{n \to \infty} k 
 \Pr \left(   \sup_{t,s\in [0,\frac{1}{k})}\|Y'_n(s)-Y'_n(t)\|\geq \varepsilon   \right) \leq \lim_{k \to \infty} C k^{-1} (\eps/2)^{-4} = 0, \]
which verifies condition~\eqref{eqn:tightDtwo}.
This completes the proof of tightness which, with the convergence of the finite-dimensional distributions and Theorem~\ref{thm:FDDandTinD}, completes the proof.
\end{proof} 

\section{Set convergence}
\label{sec:set}

\subsection{Hausdorff distance}
\label{sec:Haus}

In this section we establish convergence of the set $\{S_0, S_1, \ldots, S_n \}$, suitably scaled, in the context of the
law of large numbers and the central limit theorem. We present some consequences of this convergence in
the subsequent sections. Note that we need an appropriate metric space of sets on which this convergence
should take place. The purpose of this subsection is to set this up.

Let $\fS^d_0$
denote the collection of bounded subsets of $\R^d$ containing $0$. Let $\fK^d_0$
denote the set of compact subsets
of $\R^d$ containing $0$. Recall that $A^\eps=\{x\in \R^d : \rho_E(x,A)\leq \eps\}$ and $\rho_E(x,A)$ is the distance between a point $x$ and a set $A$. For $A, B \in \fS^d_0$,
 the
\emph{Hausdorff distance}
 between $A$ and $B$ is given by either of the following two equivalent definitions \cite[p.~84]{gruber}:
\begin{align}
\rho_H(A,B) &\coloneqq \max \left\{ \sup_{x\in A}\rho_E(x,B),\sup_{y\in B}\rho_E (y,A)\right\},
\label{eqn:Hausdef}\\
\rho_H(A,B) &\coloneqq \inf\{\eps \geq 0: A \subseteq B^\eps \text{ and } B \subseteq A^\eps\}.
\label{eqn:Hausinf}
\end{align}
Note that $\rho_H$ is a metric on $\fK^d_0$.
On $\fS^d_0$, $\rho_H$ is a only a pseudometric, since while
the triangle inequality still holds, $\rho_H(A,B) = 0$ does not imply $A = B$ (e.g.~take an open set $A$ and take
$B$ to be its closure; see Lemma~\ref{lemma:clHauss} below). Thus convergence must take place in $(\fK^d_0, \rho_H)$.

We need the following observations about the Hausdorff distance.

\begin{lemma}\label{lemma:HaussSkorSup}
Consider functions $f, g \in \calM_0^d$. Then $f[0,1], g[0,1] \in \fS^d_0$ and
\[\rho_H ( f[0,1],g[0,1] ) \leq \rho_S(f,g) \leq \rho_{\infty}(f,g).\]
\end{lemma}
\begin{proof}
Let $\Lambda'$ denote the set of all $\lambda : [0,1] \rightarrow [0,1]$ such that $\lambda[0,1]=[0,1]$. Then by~\eqref{eqn:Hausdef},
\begin{align*}
\rho_H (f[0,1],g[0,1] ) & =\sup_{t \in [0,1]} \rho_E (f(t),g[0,1] )  \vee  \sup_{t \in [0,1]}  \rho_E (g(t),f[0,1] ) \\
&= \sup_{t\in[0,1]} \inf_{s\in[0,1]} \|f(t)-g(s)\|  \vee   \sup_{t\in[0,1]}\inf_{s\in[0,1]} \|g(t)-f(s)\|\\
&= \sup_{t\in[0,1]} \inf_{s\in[0,1]} \|f(t)-g\circ\lambda(s)\|  \vee   \sup_{t\in[0,1]}\inf_{s\in[0,1]} \|g\circ\lambda(t)-f(s)\|, \end{align*}
for any $\lambda \in\Lambda'$.
Using the fact that for any $t \in [0, 1]$, $\inf_{s\in[0,1]} h(s) \leq h(t)$, we get
\[
\rho_H (f[0,1],g[0,1] ) \leq \sup_{t \in [0,1] }\| f (t) - g \circ \lambda (t)\| ,\]
for any $\lambda \in \Lambda'$, and hence
\[\rho_H (f[0,1],g[0,1] ) \leq  \inf_{\lambda \in \Lambda'} \|f-g\circ \lambda\|_{\infty} .\]
It follows that $\rho_H (f[0,1],g[0,1] ) \leq \rho_S (f,g)$, and
Lemma~\ref{lem:skor} completes the proof.
\end{proof}

Note that if $ f \in \calC^d_0$ then $f[0,1]$  is the continuous image of a compact set, containing $f(0) = 0$, and
hence $f[0, 1] \in \fK^d_0$. Thus Lemma~\ref{lemma:HaussSkorSup} shows that 
$f \mapsto f[0, 1]$ is a continuous map from $(\calC^d_0 , \rho_\infty)$
 to $(\fK^d_0 , \rho_H)$.
For $f\in \calD^d_0$, we need to work instead with $\cl f[0, 1]$. We need the following simple fact.

\begin{lemma}\label{lemma:clHauss}
For any   $A,B \in \fS^d_0$,
\[\rho_H( \cl A,B)=\rho_H(A,B).\]
\end{lemma}
\begin{proof}
Clearly $A\subseteq\cl A$, so
\begin{equation}
\label{eq:clHauss1}
 \sup_{x\in\cl A} \rho_E(x,B) \geq \sup_{x\in A} \rho_E(x,B).
\end{equation}
For any $z \in \cl A$, there exist $z_n \in A$ such that $z_n \to z$; by continuity, $\rho_E(z_n,B) \to \rho_E(z,B)$. Also, since
$z_n \in A$, it is clear that $\rho_E(z_n,B) \leq \sup_{x\in A} \rho_E(x,B)$, which gives $\rho_E(z,B) \leq \sup_{x\in A} \rho_E(x,B)$, and hence
\begin{equation}
\label{eq:clHauss2}
\sup_{z \in \cl A} \rho_E(z,B) \leq \sup_{x \in A} \rho_E(x,B).
\end{equation}
Combining~\eqref{eq:clHauss1} and~\eqref{eq:clHauss2} shows that 
$\sup_{x \in\cl A} \rho_E (x,B) = \sup_{x \in A} \rho_E(x,B)$.

Since $A \subseteq \cl A$ we have $\rho_E(y, \cl A) \leq \rho_E(y,A)$ for all $y \in B$. For any $z \in \cl A$, there exist $z_n \in A$ such
that $z_n\to z$. Then
\[ \rho_E(y, z) = \lim_{n \to \infty} \rho_E(y , z_n) \geq \rho_E (y,A), \]
so that for any $y \in B$,
\[ \rho_E (y,  \cl A) = \inf_{z \in \cl A} \rho_E(y , z) \geq \rho_E(y,A). \]
Hence $\rho_E(y, \cl A) = \rho_E (y,A)$ for any $y \in B$, so the result follows from \eqref{eqn:Hausdef}.
\end{proof}

Combining the preceding two lemmas gives the following result, which shows that $f \mapsto \cl f[0, 1]$ is a
continuous map from $(\calD^d_0, \rho_S)$ to $(\fK^d_0,\rho_H)$.

\begin{corollary}\label{cor:HclSkor}
Consider functions $f,  g \in \calD^d_0$. Then $\cl f[0, 1], \cl g[0, 1] 
\in \fK^d_0$
and
\[\rho_H ( \cl f[0,1], \cl g[0,1] )\leq \rho_S(f,g) \leq \rho_\infty (f,g) .\]
\end{corollary}
\begin{proof}
First note that $\{ f(x) : x\in [0,1]\}$ is contained in the closed
Euclidean ball centred at the origin with radius $\| f \|_\infty$, which is finite for $f \in \calD_0^d$ \cite[p.~121]{cpm}. Thus if $f, g \in \calD^d_0$, then $f[0,1], g[0,1]$ are bounded, and hence their closures are compact. 
We use Lemma~\ref{lemma:clHauss} twice to see $\rho_H (\cl f[0,1],\cl g[0,1])=\rho_H (f[0,1],g[0,1] )$,
and the result then follows from Lemma~\ref{lemma:HaussSkorSup}.
\end{proof}

\subsection{Convergence of random walk points}

Now we can present our limit theorems for the set $\{S_0, S_1, \ldots, S_n \}$. First we state a law of large numbers.
Recall that $I_\mu(t) =\mu t$.

\begin{theorem}\label{thm:SetFLLN}
Suppose that we have a random walk as defined at~\eqref{ass:walk}. Then,
as elements of $(\fK^d_0,\rho_H)$,
\[n^{-1}\{S_0,S_1,\ldots,S_n\}\toas I_{\mu}[0,1] . \]
\end{theorem}
\begin{proof}
The functional law of large numbers, Theorem~\ref{thm:flln}(b), shows that 
$X'_n\toas I_\mu$ on $(\calD^d_0, \rho_\infty)$. Corollary~\ref{cor:HclSkor}
 shows that $f \mapsto \cl f[0, 1]$ is continuous from $(\calD^d_0, \rho_\infty)$ to $(\fK^d_0, \rho_H)$, so the mapping theorem, Theorem~\ref{thm:mapping}, shows that $\cl X'_n [0, 1] \toas \cl I_\mu [0, 1]$; note that $\cl X'_n[0, 1] = X'_n[0, 1] = n^{-1} \{ S_0, S_1, \ldots, S_n \}$
and $\cl I_\mu[0, 1] = I_\mu[0, 1]$.
\end{proof}

\begin{theorem}\label{thm:SetWeakConv}
Suppose that we have a random walk as defined at \eqref{ass:walk} with $\mu =0$ and satisfying \eqref{ass:Sigma}. Then, as elements of $(\fK^d_0, \rho_H)$,
\[ n^{-1/2}\{S_0,S_1,\ldots,S_n\} \Rightarrow \Sigma^{1/2}b_d[0,1]  . \]
\end{theorem}
\begin{proof}
Donsker's theorem, Theorem~\ref{thm:Donsker-dd}(b), 
shows that $Y'_n \Rightarrow \Sigma^{1/2} b_d$ on $(\calD^d_0, \rho_S)$.
Corollary~\ref{cor:HclSkor}
 shows that $f \mapsto \cl f[0, 1]$ is continuous from $(\calD^d_0, \rho_S)$ to $(\fK^d_0, \rho_H)$, so the mapping theorem, Theorem~\ref{thm:mapweak},
shows that $\cl Y'_n [0, 1] \Rightarrow \cl \Sigma^{1/2}b_d[0,1]$;  
note that  $\cl Y'_n [0, 1] = Y'_n [0,1] = n^{-1/2}\{S_0,S_1,\ldots,S_n\}$
and $\cl \Sigma^{1/2}b_d[0,1] = \Sigma^{1/2}b_d[0,1]$.
\end{proof}

\subsection{Diameter of random walks}

As a first application of the results of this section (we see another application  in Section~\ref{sec:hulls}),
we consider the \emph{diameter} of the random walk.
Define 
\[ D_n := \diam \{S_0,\ldots,S_n\} = \max_{0 \leq i, j \leq n} \| S_i - S_j \| .\]
The following is a generalisation to $d$-dimensions of the $2$-dimensional almost-sure result contained in \cite[Theorem~1.3]{mcrwade}.
\begin{theorem}\label{thm:RWDiam} Suppose that we have a random walk as defined at \eqref{ass:walk}.
\begin{itemize}
\item[(a)] $n^{-1}D_n \toas \|\mu\|$ as $n \to \infty$.
\item[(b)] If $\mu=0$ and \eqref{ass:Sigma} holds, then $n^{-1/2}D_n \tod \diam ( \Sigma^{1/2}  b_d[0,1] )$ as $n \to \infty$.
\end{itemize}
\end{theorem}

The theorem rests on the following result, which shows that $A \mapsto \diam A$ is continuous from $(\fK^d_0, \rho_H)$ to $(\RP, \rho_E)$ which can also be found at \cite[Lemma~3.5]{mcrwade}.

\begin{lemma}
\label{lemma:DiamCont}
For any  $A,B \in \fS^d_0$,
\[| \diam A-\diam B |\leq 2 \rho_H(A,B) .\]
\end{lemma}
\begin{proof}
Let $\rho_H(A,B)=r$. From~\eqref{eqn:Hausinf} we have that for any $x_1,x_2 \in A$ and any $s>r$, there exist $y_1,y_2\in B$ such that $\rho_E(x_i,y_i)\leq s$. Then,
\[\rho_E(x_1,x_2)\leq \rho_E(x_1,y_1)+\rho_E(y_1,y_2)+\rho_E(y_2,x_2) \leq 2s + \diam B  .\]
Hence
\[ \diam A = \sup_{x_1, x_2 \in A} \rho_E ( x_1, x_2 ) \leq 2s + \diam B ,\]
and since $s > r$ was arbitrary we get $\diam A - \diam B \leq 2r$.
Similarly, $\diam B - \diam A \leq 2 r$, giving the result.
\end{proof}

\begin{proof}[Proof of Theorem \ref{thm:RWDiam}]
For part (a), we have from the law of large numbers for sets, Theorem~\ref{thm:SetFLLN},
that $n^{-1} \{ S_0, S_1, \ldots, S_n\} \toas I_\mu [0,1]$ on $(\fK^d_0, \rho_H)$,
while Lemma~\ref{lemma:DiamCont} shows that 
$A \mapsto \diam A$ is continuous from $(\fK^d_0, \rho_H)$ to $(\RP, \rho_E)$.
Thus the mapping theorem, Theorem~\ref{thm:mapping}, yields
$n^{-1} D_n \toas \diam ( I_\mu [0,1] ) = \| \mu \|$.

For part (b), we have from the central limit theorem for sets, Theorem~\ref{thm:SetWeakConv}, 
that $n^{-1/2} ~\{ S_0 , S_1, \ldots, S_n\}$ $\Rightarrow \Sigma^{1/2} b_d[0,1]$ on $(\fK^d_0, \rho_H)$. 
 Lemma~\ref{lemma:DiamCont}  together with the mapping theorem, Theorem~\ref{thm:mapweak}, yield the result.
\end{proof}

\section{Convex hulls}
\label{sec:hulls}

\subsection{Introduction}

Given $A \subseteq \R^d$,
we denote the \emph{convex hull} of $A$ by 
$\hull A$, 
the smallest convex set containing $A$.
Given the random walk $S_n$, we are interested in this section
its associated convex hull,         $\hull \{ S_0, \ldots, S_n\}$.
We apply the compact-set convergence results for $\{ S_0,\ldots, S_n\}$
from Section~\ref{sec:set} to deduce compact-set convergence of $\hull \{ S_0, \ldots, S_n\}$,
and then we give some applications to properties of the convex hull such 
as its volume or surface area.
First, we give a brief survey of existing results on convex hulls of random walks.

The first considerations of the subject were largely combinatorial in their nature. Sparre Andersen \cite{sparreandersen1954} considered the space-time picture of a one-dimensional random walk and studied the number of vertices of the convex minorant, the bottom half of the convex hull. He established that this number, $H_n$, grows like $\log n$ where $n$ is the length of the walk. This was contrasted by Qiao and Steele \cite{qiaosteele1994} more recently when they proved that, for any natural number $m$, $H_n = m$ infinitely often for some increment distribution $Z$. 

Another functional that has been largely studied is the perimeter length of the convex hull. Spitzer and Widom \cite{spitzer1961circumference} used combinatorial results with Cauchy's theorem for the length of convex polygons, which states 
\[ L=\int_0^\pi D(\theta) \ud\theta, \]
where $L$ is the length of the perimeter of a convex polygon, and $D(\theta)$ is the width of the projection of the polygon in the direction $\theta$. From this, the elegant formula 
$\Exp L_n = 2 \sum_{i=1}^n \frac{1}{i} \Exp\|S_i\|$ was established. This was extended by Snyder and Steele in 1993 \cite{snyder1993convex} in proving a law of large numbers and some bounds on the variance of the perimeter length. Steele further developed results of this type by relating functionals of the convex hull to functions of permutations of random vectors, reaffirming the results of Sparre Andersen, and Snyder and Steele. In particular he commented that 
\begin{quote}
	This \ldots tells us that the expected length of the concave majorant grows exactly like the length of the line from $(0,0)$ to the point $(n, \Exp S_n) = (n,n\mu)$.
\end{quote}
Of course this is the heuristic of our result, Theorem~\ref{thm:flln}, formalising the convergence of the walk to exactly this line.

In even more recent work, Wade and Xu studied the perimeter length and the area of the convex hull in two papers concerning the zero drift and non-zero drift case separately \cite{wade2015convex,wade2015drift}. In the zero drift case, they established limit theorems for the expectation and variance of both the perimeter length and area, using similar analysis to that presented here, considering the convergence of the hull to that of a Brownian path. In the non-zero drift case, a different technique was used and a variance convergence and a central limit theorem was established for walks, provided the distribution of $Z$ was not supported on a single straight line.

Other notable works relating to the perimeter length and area include a large deviation study by Akopyan and Vysotsky \cite{akopyanVysotsky} and a paper describing further expansions of the expectation of both functionals, restricted to the case of symmetric or Gaussian increments, which was written by Grebenkov, Lanoisel\'{e}e and Majumdar \cite{GrebenkovLanoiseleeMajumdar}.

In what follows, we improve Snyder and Steele's result, by using the functional law of large numbers to remove the necessity for the finite second moment condition and then extend Wade and Xu's results to higher dimensions using the $d$-dimensional Donsker's theorem. We also consider an additional functional, the mean width of the convex hull.

\subsection{Trajectories and hulls}

We use the notation for sets of subsets of $\R^d$ and for the Hausdorff distance $\rho_H$
from Section~\ref{sec:Haus}.
We need the following result.

\begin{lemma}\label{lemma:hullHaus}
For any  $A,B \in \fS^d_0$,
\[\rho_H( \hull A, \hull B ) \leq \rho_H(A,B) .\]
\end{lemma}
\begin{proof}
For any $x \in \hull A$ there exist finitely many points $x_1, x_2, \ldots, x_n \in A$ 
and $\lambda_1, \lambda_2, \ldots, \lambda_n$ with $\lambda_i \geq 0$, $\sum_{i=1}^n \lambda_i =1$, for which $x = \sum_{i=1}^n \lambda_i x_i$ (see e.g.~\cite[p.~42]{gruber}).
Let $r := \rho_H (A, B)$. For any $s>r$, we have from~\eqref{eqn:Hausinf} that for each $x_i \in A$ there exists $y_i \in B$ such that 
$\rho_E (x_i, y_i ) \leq s$.
Now consider $y=\sum_{i=1}^n \lambda_i y_i \in \hull B$. Then 
\[ \rho_E(x,y) \leq \sum_{i=1}^n \lambda_i \rho_E ( x_i,y_i ) \leq s . \]
This calculation implies that $\hull A \subseteq ( \hull B )^s$,
and by a similar argument  we get $\hull B \subseteq  (\hull A)^s$.
With~\eqref{eqn:Hausinf} we get $\rho_H ( \hull A, \hull B) \leq s$.
Since $s >r$ was arbitrary, the result follows.
\end{proof}

Let $\fC^d_0$
denote the set convex compact subsets of $\R^d$ containing $0$. For $A \in \fC^d_0$, we define the \emph{support function} of A by
\begin{equation}
\label{eq:support}
h_A(x) \coloneqq \sup_{y\in A}(x\cdot y), \text{ for any } x\in \R^d  .\end{equation}
Then for $A, B \in \fC^d_0$
we have another equivalent description of $\rho_H(A,B)$ (see e.g.~\cite[p.~84]{gruber}):
\begin{equation}
\rho_H(A,B) = \sup_{e\in \Sp^{d-1}}|h_A(e)-h_B(e)|.
\label{eqn:Haussupport}
\end{equation}
Given $f \in \calD^d_0$, we have $\cl  f[0,1]$ is compact and contains
$0 = f(0)$. A theorem of Carath\'eodory \cite[p.~44]{gruber} says that if $A$ is compact then so is $\hull A$; hence
$\hull \cl f[0,1]$ is compact. Moreover, we have that $\hull \cl A = \cl \hull A$ \cite[p.~45]{gruber}. Hence if $f \in \calD^d_0$
then $\cl \hull f[0,1] \in \fC^d_0$. Of course, if $f \in \calC^d_0$ then $f[0,1]$ and hence $\hull f[0,1]$ is already compact.
The following result shows that $f \mapsto \cl \hull f[0,1]$ is a continuous map from $(\calD^d_0,\rho_S)$ to $(\fC^d_0,\rho_H)$.
This fact is also found as Lemma~5.1 in the recent paper of Molchanov and Wespi~\cite{molchanov2016convex}.

\begin{lemma}\label{lemma:hullHausSkor}
Consider two functions $f, g \in \calM^d_0$. Then,
\[\rho_H \left(\cl \hull f[0,1], \cl \hull g[0,1]\right) \leq \rho_S (f,g ).\] 
\end{lemma}
\begin{proof}
First,   Lemma~\ref{lemma:clHauss} (twice) and  Lemma~\ref{lemma:hullHaus} yield
\begin{align*}
\rho_H \left(\cl\hull f[0,1],\cl \hull g[0,1]\right) = \rho_H \left( \hull f[0,1], \hull g[0,1]\right)
\leq \rho_H ( f[0,1], g[0,1] ) .\end{align*}
 Lemma~\ref{lemma:HaussSkorSup} completes the proof.
\end{proof}

\subsection{Limit theorems for convex hulls}

The following is our law of large numbers for the convex hull.

\begin{theorem}\label{thm:hullFLLN}
Suppose that we have a random walk as defined at~\eqref{ass:walk}. 
Then, as elements of $(\fC_0^d, \rho_H)$,
\[n^{-1}\hull \{S_0,\ldots,S_n\} \toas I_{\mu}[0,1].\]
\end{theorem}
\begin{proof}
Theorem~\ref{thm:SetFLLN} shows that $n^{-1} \{ S_0, \ldots, S_n\} \toas I_{\mu} [0,1]$
on $(\fK^d_0, \rho_H)$. Lemma~\ref{lemma:hullHaus}
shows that $A \mapsto \hull A$ is a continuous map
from $(\fK^d_0, \rho_H)$
to $(\fC^d_0, \rho_H)$, so the mapping theorem, Theorem~\ref{thm:mapping}, implies
that $\hull n^{-1} \{ S_0, \ldots, S_n \} \toas \hull I_{\mu} [0,1]$.
Here $\hull I_{\mu} [0,1] = I_{\mu} [0,1]$, and, since the convex hull is preserved under scaling,
$\hull n^{-1} \{ S_0, \ldots, S_n \} = n^{-1} \hull \{ S_0, \ldots, S_n \}$.
\end{proof}

Next we state the accompanying central limit theorem.
Let $h_d := \hull b_d[0,1]$, the convex hull of $d$-dimensional Brownian motion run for unit time.

\begin{theorem}
\label{thm:hullCLT}
Suppose that we have a random walk as defined at \eqref{ass:walk} with $\mu =0$ and satisfying \eqref{ass:Sigma}.
Then, as elements of $(\fC_0^d, \rho_H)$,
\[n^{-1/2} \hull \{S_0,\ldots,S_n\} \Rightarrow \Sigma^{1/2}h_d.\]
\end{theorem}
\begin{proof}
Theorem~\ref{thm:SetWeakConv} shows that $n^{-1/2} \{ S_0, \ldots, S_n\} \Rightarrow \Sigma^{1/2} b_d[0,1]$
on $(\fK^d_0, \rho_H)$. Lemma~\ref{lemma:hullHaus}
shows that $A \mapsto \hull A$ is a continuous map
from $(\fK^d_0, \rho_H)$
to $(\fC^d_0, \rho_H)$, so the mapping theorem, Theorem~\ref{thm:mapweak}, implies
that $\hull n^{-1/2} \{ S_0, \ldots, S_n \} \Rightarrow \hull \Sigma^{1/2} b_d[0,1]$.
Since the convex hull is preserved under affine transformations, $\hull \Sigma^{1/2} b_d[0,1] = \Sigma^{1/2} \hull b_d[0,1]$.
\end{proof}

\begin{remark}
Alternatively, we could obtain Theorems~\ref{thm:hullFLLN} and~\ref{thm:hullCLT}
directly from the functional law of large numbers, Theorem~\ref{thm:flln},
and Donsker's theorem, Theorem~\ref{thm:Donsker-dd}, using Lemma~\ref{lemma:hullHausSkor}.
\end{remark}

Suppose now $d\geq 2$. To obtain second-order results in the case where $\mu \neq 0$, an additional scaling limit
is required. Let $\{e_1,\ldots,e_d\}$ be the standard orthonormal basis of $\R^d$, and supposing that $\mu \neq 0$, let $\{u_1, \ldots, u_d \}$ be another orthonormal basis of $\R^d$ with $u_1 = \hat \mu$. Then we transform $\xi$ into $\xi'$ by taking
\[ \xi' = (\xi'_1,\xi'_2,\ldots,\xi'_d) := (\xi\cdot u_1, \xi \cdot u_2,\ldots , \xi \cdot u_d),\]
and consider $\xi'_{\perp} := (\xi'_2,\ldots, \xi'_d)$. Note that, since $\Exp \xi \cdot u_k = \mu \cdot u_k = 0$ for $k \neq 1$, we have $\Exp \xi'_\perp = 0$. Then set
\begin{equation} \label{eqn:sigmamuperp}
\Sigma_{\mu_{\perp}} := \Exp [ \xi'_\perp (\xi'_\perp)^\tra ] .
\end{equation}
This defines a $(d-1)$-dimensional covariance matrix, describing the covariances of the process projected onto the hyperplane orthogonal to the mean vector. Note that $\Sigma_{\mu_{\perp}}$ is non-negative definite and hence it has a unique non-negative definite symmetric square root matrix $\Sigma_{\mu_{\perp}}^{1/2}$. It will be useful to have notation for $\Sigma_{\mu_{\perp}}^{1/2}$ extended back to a $d$-dimensional matrix which we will denote as $\tilde{\Sigma}_{\mu_\perp}^{1/2}$, specifically we define
\begin{equation}\label{eqn:sigmaextended}
\tilde{\Sigma}_{\mu_\perp}^{1/2} := \left(
\begin{array}{cccc}
	1 & 0 & \hdots & 0\\
	0 &   &        &  \\
	\vdots & &\scalebox{1.3}{ $\Sigma_{\mu_{\perp}}^{1/2}$} & \\
	0 & & & 
\end{array}\right).
\end{equation}
We will need a new weak convergence result and as we took a mapping of the increments above, we need to define a different mapping for the walk process itself, for which we use a $d$-dimensional analogue of that used in \cite{wade2015convex}. Namely, for $n\in \N$, define $\psi_{n,\mu}:\R^d \rightarrow \R^d$ by the image of $x \in \R^d$ in Cartesian components:
\begin{equation*}
\psi_{n,\mu}(x)=\left( \frac{x \cdot u_1}{n\|\mu\|},\frac{x\cdot u_2}{\sqrt{n}},\ldots,\frac{x\cdot u_d}{\sqrt{n}}\right),
\end{equation*}
where $\{u_1,\ldots,u_d\}$ is the orthonormal basis defined above. We extend this, and subsequent similar notation, to sets in the usual way, $\psi_{n,\mu}(A)=\{\psi_{n,\mu}(x):x\in A\}$. This mapping has an effect which is the natural extension of its $2$-dimensional equivalent, rotating $\R^d$ mapping $\hat{\mu}$ to the unit vector in the horizontal direction, and scaling space with a horizontal shrinking factor of $\|\mu\|n$, but now also a factor of $\sqrt{n}$ in all $d-1$ directions orthogonal to the horizontal.

We will also need some notation for the first component of the mapping, and the $d-1$ vector containing the elements orthogonal to the mean, so we define the following: \[\psi_{n,\mu}^1(x):= \frac{x\cdot u_1}{n\|\mu\|} \qquad {\rm and} \qquad \psi_{n,\mu}^{\perp}(x):= \left( \frac{x\cdot u_2}{\sqrt{n}},\ldots,\frac{x\cdot u_d}{\sqrt{n}}\right).\]
Naturally, we also need to define a new limiting process which combines the drift with Brownian motion in a time-space way. We denote this $\tilde b_d(t)$, which is defined as
\begin{equation} \label{eqn:btilde}
\tilde b_d (t) = (t, b_{d-1}(t)), {\rm \ for\ }t \in [0,1],
\end{equation} 
where we use the notation $b_{d-1}$ to be clear that we mean $(d-1)$-dimensional Brownian motion. We use the notation $\tilde h_d^\Sigma := \hull \tilde{\Sigma}_{\mu_\perp}^{1/2} \tilde b_d[0,1]$, the hull of $\tilde{\Sigma}_{\mu_\perp}^{1/2} \tilde b_d$ run for unit time.

\begin{lemma}\label{lemma:driftmapping}
Suppose that we have a random walk as defined at \eqref{ass:walk} with $\mu \neq 0$ and satisfying \eqref{ass:Sigma}.
Then, as $n\rightarrow \infty$, as elements of $(\fC^d_0, \rho_H)$,
\[ \psi_{n,\mu}( \hull\{S_0,S_1,\ldots,S_n\} ) \Rightarrow \tilde{h}_d^\Sigma . \]
\end{lemma}

\begin{proof}
First, note that, since $\psi_{n,\mu}$ is an affine transformation, we have \[\psi_{n,\mu}(\hull\{S_0,\ldots,S_n\}) = \hull\left(\psi_{n,\mu} (\{S_0,\ldots,S_n\})\right).\]
Noting that $A \mapsto \hull A$ is continuous from $(\fK^d_0,\rho_H)$ to $(\fC^d_0,\rho_H)$ by Lemma~\ref{lemma:hullHaus}, the continuous mapping theorem, Theorem~\ref{thm:mapweak}, means it is sufficient to show 
\begin{equation}\label{eqn:btildewc}
\psi_{n,\mu} ( \left\{ S_0,\ldots,S_n\right\}) \Rightarrow \tilde{\Sigma}_{\mu_\perp}^{1/2} \tilde{b}_d [0,1] {\rm\ on\ } (\fK^d_0,\rho_H).
\end{equation}

In order to show this, we first define a new unscaled trajectory as $W'_n(t):=S_{\lfloor nt \rfloor}$.
Then we will show that,
\begin{equation} \label{eqn:drifttrajectories}
\psi_{n,\mu}(W'_n)\Rightarrow \tilde{\Sigma}_{\mu_\perp}^{1/2} \tilde{b}_d, {\rm \ on\ } (\calD_0^d,\rho_S).
\end{equation}

First, recall Theorem~\ref{thm:slutsky}: if $X_n, Y_n,$ and $X$ are elements of a metric space $(S,\rho)$, such that $X_n \Rightarrow X$ and $\rho(X_n,Y_n) \toP 0$, then $Y_n \Rightarrow X$. Taking $X_n = (I, \psi_{n,\mu}^\perp(W'_n))$ where we recall $I$ is the identity map on $[0,1]$, $Y_n = \psi_{n,\mu}(W'_n)$ and $X = \tilde{\Sigma}_{\mu_\perp}^{1/2}\tilde{b}_d$, all elements of $(\calD_0^d,\rho_S)$ it suffices to show that 
\begin{equation}\label{eqn:ethier1}
	\rho_S(\psi_{n,\mu}(W'_n),(I, \psi_{n,\mu}^\perp(W'_n))) \toP 0,
\end{equation}
and 
\begin{equation}\label{eqn:ethier2}
	(I, \psi_{n,\mu}^\perp(W'_n)) \Rightarrow \tilde{\Sigma}_{\mu_\perp}^{1/2} \tilde{b}_d(t), {\rm \ on \ } (\calD_0^d,\rho_S).
\end{equation}

To prove \eqref{eqn:ethier1}, notice that $\psi_{n,\mu}^1 (W'_n)$ is the piecewise constant trajectory of a one-dimensional walk with $\|\mu\|>0$ now normalised by $\|\mu\|^{-1}n^{-1}$, so Theorem~\ref{thm:flln} applies and we have 
\begin{equation}\label{eqn:psi1} \lim_{n\rightarrow\infty} \psi_{n,\mu}^1(W'_n) = I\ \as
\end{equation} 
Using Lemma~\ref{lem:skor} it becomes a simple exercise to see that, for $f \in \calC_0^{d-1}$ and $g,h \in \calC_0$ we have $\rho_S((f,g),(f,h)) \leq \rho_\infty((f,g),(f,h)) = \rho_\infty(g,h)$, which shows that \eqref{eqn:psi1} implies \eqref{eqn:ethier1}.
 
For \eqref{eqn:ethier2}, note $\psi_{n,\mu}^\perp W'_n$ is the piecewise constant trajectory of a $(d-1)$-dimensional walk with $\mu=0$, normalised by $n^{-1/2}$ so Theorem~\ref{thm:Donsker-dd} gives
\begin{equation*}
\psi_{n,\mu}^\perp (W'_n) \Rightarrow \Sigma_{\mu_{\perp}}^{1/2} b_{d-1} {\rm \ on \ } (\calD_0^{d-1},\rho_S).
\end{equation*}
This implies that, for all bounded, continuous $f:\calD_0^{d-1} \mapsto \R$, 
\begin{equation}\label{eqn:finnonzero} 
\Exp[f(\psi_{n,\mu}^\perp (W'_n))] \to \Exp[f(\Sigma_{\mu_{\perp}}^{1/2} b_{d-1})], {\rm \ as \ } n \rightarrow \infty.
\end{equation}
Now consider $\Exp[g(I,\psi_{n,\mu}^\perp (W'_n)]$ for any bounded, continuous $g:\calD_0^d\mapsto \R$. Then, since $I$ is a non-random function, there exists a function $f$, chosen such that $f(\cdot)=g(I,\cdot)$ which is itself bounded and continuous on $\calD_0^{d-1}$. By \eqref{eqn:finnonzero}, it follows that
\[ \Exp[g(I,\psi_{n,\mu}^\perp (W'_n)]= \Exp[f(\psi_{n,\mu}^\perp (W'_n))] \to \Exp[f(\Sigma_{\mu_{\perp}}^{1/2} b_{d-1})] = \Exp[g(I,\Sigma_{\mu_{\perp}}^{1/2} b_{d-1})], \]
and noting $g(I,\Sigma_{\mu_{\perp}}^{1/2} b_{d-1})=g(\tilde{\Sigma}_{\mu_\perp}^{1/2}\tilde{b}_d)$, we have proven \eqref{eqn:ethier2} and hence \eqref{eqn:drifttrajectories}.

The final step is to notice that Corollary~\ref{cor:HclSkor} shows that $f \mapsto \cl f[0, 1]$ is continuous from $(\calD^d_0, \rho_S)$ to $(\fK^d_0, \rho_H)$, so the mapping theorem, Theorem~\ref{thm:mapweak}, with \eqref{eqn:drifttrajectories} shows that $\cl \psi_{n,\mu}(W_n[0, 1]) \Rightarrow \cl \tilde{\Sigma}_{\mu_\perp}^{1/2} \tilde{b}_d[0,1]$. Observing that  $\cl \psi_{n,\mu}(W_n[0, 1]) = \psi_{n,\mu}(\{S_0,\ldots,S_n\})$
and $\cl \tilde{\Sigma}_{\mu_\perp}^{1/2} \tilde{b}_d[0,1] = \tilde{\Sigma}_{\mu_\perp}^{1/2} \tilde{b}_d[0,1]$, we have proven \eqref{eqn:btildewc} and so the proof is complete.
\end{proof}

\subsection{Applications to functionals of convex hulls}

We consider three functionals defined on non-empty convex compact sets. First, let $\calW : \fC^d_0 \to \RP$ denote the \emph{mean width} defined by
\[ \calW ( A ) := \int_{\Sp^{d-1}} h_A (e) \ud e ,\]
where $h_A$ is the support function of $A$ as defined at~\eqref{eq:support}.
Define the \emph{volume} functional by
\[ \calV (A ) := \mu_d (A) ,\]
the $d$-dimensional Lebesgue measure of $A$.
Also we follow Gruber \cite[p.~104]{gruber} and define the \emph{surface area} functional by
\[ \calS ( A ) := \lim_{\eps \downarrow 0} \left(   \frac{ \calV (A^\eps) - \calV(A)}{ \eps } \right); \]
which was a definition originally suggested by Minkowski; the limit exists by the Steiner formula of integral geometry \cite[Theorem~6.6]{gruber} which states, for $S \in \fC^d$,
\begin{equation}\label{eqn:steiner}
\mu_d(S^\lambda) = \sum_{i=0}^d \binom{d}{i} Q_i(S) \lambda^i,
\end{equation}
where $\binom{x}{y}$ is the binomial coefficient with the convention $\binom{x}{0}=1$, and $Q_i(S)$ are the quermassintegrals of $S$.

For the random walk, we use the notation
\[ \calW_n := \calW ( \hull \{S_0, \ldots, S_n\} ) ; ~~~ \calV_n := \calV ( \hull \{S_0, \ldots, S_n\} ) ; ~~~ \calS_n := \calS ( \hull \{S_0, \ldots, S_n\} ) .\]

We first investigate basic continuity properties of these functionals. We define the Euler gamma function by 
\[
\Gamma (t) := \int_{0}^{\infty } x^{t-1} \re^{-x} \ud x, \text{ for } t >0.
\]
\begin{lemma}\label{lemma:convexhullcont}
Suppose that $A, B \in \fC^d_0$. Then
\begin{align}
\rho_E(\calW(A),\calW(B))&\leq 2\pi \rho_H(A,B)^{d-1}\ ;\label{eqn:WcurlCont}\\[0.2in]
\rho_E(\calS(A),\calS(B))&\leq (d-1)\left(\frac{2\pi^{(d-1)/2}(\diam(B)+\rho_H(A,B))^{d-2}}{\Gamma(\frac{d-1}{2})}\right)^{d-1}\cdot \rho_H(A,B)^{d-1}; \label{eqn:LcurlCont}\\[0.2in]
\rho_E(\calV(A),\calV(B))&\leq \pi^{d-1} \rho_H(A,B)^d \nonumber\\
&\quad + \max\limits_{S\in\{A,B\}}\left( \calS(S)+\sum\limits_{i=2}^{d-1}2\pi\max\left\{\frac{\mathrm{diam}(S)}{2},1\right\}^d \rho_H(A,B)^{i-1}\right)\rho_H(A,B)\ .\label{eqn:AcurlCont}
\end{align}
\end{lemma} 
Before we complete the proofs of these inequalities we note Cauchy's surface area formula and a further geometric lemma. Recall that if $\nu_{d}$ is the volume of the unit ball in $d$-dimensions, then Cauchy's surface area formula \cite[p.~106]{gruber} states that for $A \in \fC^d$,
\[ \calS(A)=\frac{1}{\nu_{d-1}}\int_{\Sp^{d-1}}\mu_{d-1}(A|u^{\perp})\ud u, \]
where $A|u^{\perp}$ denotes the projection of $A$ onto the $d-1$-dimensional subspace of $\R^d$ perpendicular to $u$.
\begin{remark}\label{rem:Cauchy}
	In $d=2$ Cauchy's formula says $\calS(A)=\calW(A)$.
\end{remark}

The geometric lemma is a bound on the Lebesgue measure of the difference in volume of two convex sets.
\begin{lemma}\label{lemma:lebmeashulldiff}
Consider two sets $S_1,S_2,\in\fC^{d}_0$ with $\rho_H(S_1,S_2)=r$, then 
\[\mu_{d}(S_1 \backslash S_2)\leq \frac{2\pi^{d/2}(\diam(S_2)+r)^{d-1}}{\Gamma (\frac{d}{2})}\cdot r.\]
\end{lemma}
\begin{proof}
First we recall \eqref{eqn:steiner} and note that $Q_0(S)=\mu_d(S)$; for a comprehensive discussion on quermassintegrals see \cite[Ch.~6]{gruber}. We also note one further result of Steiner, see \cite[Theorem~6.14]{gruber} which states, for $S \in \fC^d$,
\begin{equation*}
	\mu_{d-1}(\partial (S^\lambda)) = d \sum_{i=0}^{d-1} \binom{d-1}{i} Q_{i+1}(S) \lambda^i = \sum_{i=1}^{d} i \binom{d}{i} Q_{i}(S) \lambda^{i-1}.
\end{equation*}
It is a simple exercise by comparison of terms in the summations and use of the fact $Q_0(S)=\mu_d(S)$ to see 
\begin{equation}\label{eqn:volSAdifference}
\mu_d(S^\lambda) - \mu_d(S) = \lambda \sum_{i=1}^d \binom{d}{i} Q_{i}(S)\lambda^{i-1} \leq \lambda \mu_{d-1}(\partial (S^\lambda)).
\end{equation}
Now, if $\rho_H(S_1,S_2)=r$, for any $s>r$, $S_1 \subseteq S_2^s$, so $S_1 \setminus S_2 \subseteq S_2^s \setminus S_2$. It follows from \eqref{eqn:volSAdifference},  \begin{equation}\label{eqn:setlem1}
\mu_{d}(S_1\setminus S_2)\leq \mu_d(S_2^s \setminus S_2) = \mu_d(S_2^s) - \mu_d(S_2) \leq s \mu_{d-1}(\partial(S_2^s)).
\end{equation} 
Now, recall $\B^d$ is the $d$-dimensional unit ball. Then notice that it follows from Cauchy's formula that for convex sets $A$ and $B$ such that $A\subseteq B$,  $S(A)\leq S(B)$, so, because $S_2^s \subseteq (\diam(S_2)+s)\B^d$, we have \begin{equation}\label{eqn:setlem2}
s\mu_{d-1}(\partial (S_2^s))\leq s\mu_{d-1}(\partial(\diam(S_2)+s)(\B^d)).
\end{equation} 
Since $s>r$ was arbitrary, the statement of the lemma follows from \eqref{eqn:setlem1}, \eqref{eqn:setlem2} and the surface area formula for $\B^d$, see for example \cite[p.~136]{sommerville}.
\end{proof}
\noindent Now we turn to the proof of Lemma \ref{lemma:convexhullcont}.
\begin{proof}[Proof of Lemma~\ref{lemma:convexhullcont}]
We first prove \eqref{eqn:WcurlCont}. By Cauchy's formula and the triangle inequality,
$$|\calW(A)-\calW(B)|=\left|\int_{\Sp^{d-1}}(h_A(e)-h_B(e))\ud e\right|\leq 2\pi \left(\sup\limits_{e\in \Sp^{d-1}}|h_A(e)-h_B(e)|\right)^{d-1},$$
which holds because the $2\pi$ is the constant when $d=2$, and all other constants are less than this. This equation with \eqref{eqn:Haussupport} gives \eqref{eqn:WcurlCont}.

Next we consider \eqref{eqn:LcurlCont}. Suppose, without loss of generality, $S(A) \geq S(B)$. Then, by Cauchy's surface area formula and the volume of $\B^d$ formula, see \cite[p.~136]{sommerville}, 
\begin{align*}
\rho_E(\calS(A),\calS(B))&=\frac{\Gamma(\frac{d+1}{2})}{\pi^{(d-1)/2}} \int_{\Sp^{d-1}}\left( \mu_{d-1}(A|u^{\perp})-\mu_{d-1}(B|u^{\perp})\right)\ud u\\
&\leq \frac{\Gamma(\frac{d+1}{2})}{\pi^{(d-1)/2}} \int_{\Sp^{d-1}}\left( \mu_{d-1}(A|u^{\perp})-\mu_{d-1}(A \cap B|u^{\perp})\right) \ud u\\
&\leq \frac{\Gamma(\frac{d+1}{2})}{\pi^{(d-1)/2}} \int_{\Sp^{d-1}}\mu_{d-1}(A\setminus B|u^{\perp})\ud u.\\
\intertext{Now, noting that $A\setminus B | u^{\perp} = (A|u^{\perp}) \setminus (B| u^{\perp})$, that for a set $B$, $\diam(B|u^{\perp})\leq \diam(B)$, and that $\rho_H(A|u^{\perp},B|u^{\perp})\leq \rho_H(A,B)=r$ we can apply Lemma~\ref{lemma:lebmeashulldiff} to get,}
\rho_E(\calS(A),\calS(B))&\leq \frac{\Gamma(\frac{d+1}{2})}{\pi^{(d-1)/2}} \int_{\Sp^{d-1}}\frac{2\pi^{(d-1)/2}(\diam(B)+r)^{d-2}\cdot r}{\Gamma(\frac{d-1}{2})}\ud u\\
& = \frac{\Gamma(\frac{d+1}{2})}{\pi^{(d-1)/2}}  \cdot \frac{2\pi^{(d-1)/2}}{\Gamma(\frac{d-1}{2})}\cdot \left(\frac{2\pi^{(d-1)/2}(\diam(B)+r)^{d-2}\cdot r}{\Gamma(\frac{d-1}{2})}\right)^{d-1}\\
& = (d-1)\left(\frac{2\pi^{(d-1)/2}(\diam(B)+r)^{d-2}}{\Gamma(\frac{d-1}{2})}\right)^{d-1}\cdot r^{d-1}
\end{align*}
and the result follows.

And finally, we consider \eqref{eqn:AcurlCont}. Set $r=\rho_H(A,B)$. Then, by \eqref{eqn:Hausinf}, $A\subseteq B^s$ for any $s>r$. Hence,
\begin{align*}
\calV(A) &\leq \calV(B^s)\\
&\leq \calV(B) + \calW(B)s + \pi^{d-1}s^d + \sum_{i=2}^{d-1}Q_i(B)s^i\ ,
\end{align*}
by the Steiner formula \eqref{eqn:steiner} where $Q_i$ are the quermassintegrals. However, as discussed at \cite[p.~109]{gruber} the quermassintegrals can be expressed as the mean of the $(d-i)$-dimensional volumes of the projections of the set $B$ into $(d-i)$ dimensional subspaces. Thus we can establish the crude bound $Q_i \leq 2\pi \left(\max\left\{\frac{\diam B}{2},1\right\}\right)^d$ for all $i\in\{2,\ldots,d-1\}$ and so each $Q_i$ is finite because $B$ is compact (assume $d$ fixed). By symmetry we can get a similar inequality starting from $\calV(B)$ and since $s>r$ was arbitrary, \eqref{eqn:AcurlCont} follows.
\end{proof}

So now we have the weak convergence result, continuity of the relevant functionals and the mapping theorem, we can return to the weak convergence of the functionals. The $2$-dimensional statements for the surface area and volume were previously studied in \cite{wade2015convex}.

\begin{theorem}\label{thm:ddimLcurlAcurlzerodrift}
	Suppose we have the walk defined at \eqref{ass:walk} with $\mu=0$, \eqref{ass:Sigma} and, $\calW_n$, $\calS_n$ and $\calV_n$ are the mean width, surface area and volume respectively of the hull of the $d$-dimensional random walk. Then, as $n\rightarrow \infty$,
	\begin{align*}
	n^{-1/2} \calW_n &\tod \calW \left( \Sigma^{1/2} h_d \right)\\
	n^{-(d-1)/2}\calS_n &\tod \calS\left(\Sigma^{1/2}h_d\right)\\
	n^{-d/2} \calV_n &\tod \calV\left(\Sigma^{1/2}h_1^d\right)=v_d \sqrt{\det(\Sigma)}
	\end{align*}
	where $v_d$ is the volume of $h_d$.
\end{theorem}

\begin{proof}
Notice that Theorem \ref{thm:hullCLT} gives \[n^{-1/2}\mathrm{hull}\{S_0,S_1,\ldots,S_n\}\Rightarrow \Sigma^{1/2}h_d, {\rm \ on\ } (\fC^d_0,\rho_H),\] where $h_d$ is the hull of the $d$-dimensional Brownian motion starting at $b_d(0)=0$. Using this fact and Lemma \ref{lemma:convexhullcont}, it only remains to observe that the rescaling of the walk by $n^{-1/2}$ in all directions rescales $\calW$ by $n^{-1/2}$, $\calS$ by $n^{-(d-1)/2}$ and $\calV$ by $n^{-d/2}$ which are continuous functions and therefore the mapping from the original walk to that of Brownian motion is also continuous. The result with the given limits follows, with the additional equality for the volume functional following from the Jacobian of the transformation $x \mapsto \Sigma^{1/2} x$ being $\sqrt{\det \Sigma}$.
\end{proof}

In the special case $d=2$, $\calL_n := \calS_n$ is the perimeter length of $\hull \{S_0, \ldots, S_n\}$;
Cauchy's formula also confirms that $\calL_n$ is equal to $\calW_n$ in this case, see Remark~\ref{rem:Cauchy}.

\begin{theorem}
\label{thm:perimeter-lln}
Suppose that we have a random walk as defined at~\eqref{ass:walk}. 
Then 
\[n^{-1} \calL_n  \toas 2 \| \mu \| .\]
\end{theorem}

\begin{remark}
This result was proven in \cite{mcrwade} \lq directly\rq~from the strong law of large numbers and Cauchy's surface area formula. Snyder and Steele~\cite{snyder1993convex} had previously obtained the result under the stronger condition $\Exp ( \| \xi \|^2 ) < \infty$
as a consequence of an upper bound on $\Var \calL_n$ deduced from Steele's version of the Efron--Stein inequality. In fact, Snyder and Steele state the result only for the case $\mu  \neq 0$, but their proof works equally well when $\mu  = 0$.
\end{remark}

\begin{proof} 
Using $\calL_n = \calW_n$ in the case $d=2$, the almost-sure convergence of Theorem~\ref{thm:hullFLLN}, the continuity of $\calW_n$ from Lemma~\ref{lemma:convexhullcont}, and the continuous mapping theorem from Theorem~\ref{thm:mapping} to establish $n^{-1} \calL_n \toas \calW(I_\mu [0,1])$. Without loss of generality, we will assume $\mu = \|\mu\| e_{\pi/2}$ in order to  calculate the right hand side explicitly:
\begin{align*} \calW(I_\mu [0,1]) = \int_{\Sp} h_{I_\mu[0,1]}(e)\ud e &= \int_0^\pi (0, \|\mu\|) \cdot (\cos \theta, \sin \theta) \ud \theta + \int_\pi^{2\pi} (0,0)\cdot (\cos \theta,\sin \theta) \ud \theta\\
&= -\|\mu\|\cos \pi + \|\mu\| \cos 0 = 2\|\mu\|.\qedhere
\end{align*}
\end{proof}
We finish this section with the weak convergence statement for the $d$-dimensional volume of the walk with drift. This was also studied in \cite{wade2015convex} for the specific case $d=2$.
\begin{theorem}\label{thm:ddimAcurlwithdrift}
Suppose we have the walk defined at (\ref{ass:walk}) with $\|\mu\| > 0$, (\ref{ass:Sigma}) and $\calV_n$ is the volume of the hull of the $d$-dimensional random walk. Then, as $n\rightarrow \infty$,
\[n^{-(d+1)/2} \calV_n\tod \|\mu\| \sqrt{\det \Sigma_{\mu_\perp}}  \tilde{v}_d,\]
where $\tilde{v}_d$ is the volume of $\tilde{h}_d := \hull \tilde{b}_d[0,1]$ where $\tilde{b}_d[0,1]=\left\{\tilde{b}_d(t):t\in [0,1]\right\}$ with $\tilde{b}_d(t)$ described at \eqref{eqn:btilde} and $\Sigma_{\mu_\perp}$ as described at \eqref{eqn:sigmamuperp}.
\end{theorem}
\begin{proof}
Recall the definition of $\tilde{\Sigma}_{\mu_\perp}^{1/2}$ from~\eqref{eqn:sigmaextended}.
Then note that $\hull \tilde{\Sigma}_{\mu_\perp}^{1/2} \tilde{b}_d[0,1] = \tilde{\Sigma}_{\mu_\perp}^{1/2} \hull \tilde{b}_d[0,1]$ because left multiplication by $\tilde{\Sigma}_{\mu_\perp}^{1/2}$ is an affine transformation, and that $\calV(\tilde{\Sigma}_{\mu_\perp}^{1/2} A) = \sqrt{\det\tilde{\Sigma}_{\mu_\perp}}\calV(A) = \sqrt{\det\Sigma_{\mu_\perp}}\calV(A)$ because $\sqrt{\det\tilde{\Sigma}_{\mu_\perp}}$ is the Jacobian of the transformation. It follows, 
\begin{equation}\label{eqn:voltransform}
\calV(\psi_n^{\mu}(A))=n^{-(d+1)/2}\left(\|\mu\|\sqrt{\det\Sigma_{\mu_\perp}}\right)^{-1} \calV(A)
\end{equation} for $A\in \fC^d_0$. Then we use Lemma \ref{lemma:driftmapping}, the continuous mapping theorem, and the continuity of the functional, Lemma \ref{lemma:convexhullcont} in the usual way with~\eqref{eqn:voltransform} to complete the proof.
\end{proof}

\section{Centre of mass}\label{sec:COM}

\subsection{Introduction}
\label{introduction}

Given a random walk, as defined at \eqref{ass:walk}, we denote the centre of mass of the process $(G_n, n \in \ZP)$ using the definition $G_n := \frac{1}{n}\sum_{i=1}^{n}{S_i}$ for all $n \ge 1$ and the convention $G_0:=0$. Random walks can be used to model physical polymer molecules \cite{CMVW,RC} in which case the centre of mass is of obvious physical relevance. The random walk can also be used to model animal behaviour, and the motion of both macroscopic and microscopic organisms \cite{PH, PB3}. In this context the centre of mass is a natural summary statistic of an animal's roaming behaviour.

We apply the theory of the functional law of large numbers and the functional central limit theorem to state a convergence results for the centre of mass process. 

To give an idea of the behaviour exhibited by this functional, we have performed some simulations which are exhibited in Figure~\ref{sim:comlln} and Figure~\ref{sim:comclt}. The former shows a comparison of the centre of mass and the random walk itself with the scaling of $1/n$, for two differing values of $n$ and mean drift $\mu$. We can see the respective centre of mass processes exhibit the behaviour described in Theorem~\ref{comt1} and Theorem \ref{comt3}, specifically that they converge to a straight line with slope $\mu /2$ when $n \to \infty$. The second simulation also compares the centre of mass processes to random walks, but this time we use a different scaling of $1 / \sqrt{n}$, and we still show the two different values of $n$, but have fixed $\mu=0$ for both cases. These trajectories show Theorem \ref{comt2} and Theorem \ref{comt4} in action. 
\begin{figure}[!ht]
	\centering
	\begin{multicols}{2}
		\includegraphics[width=245pt]{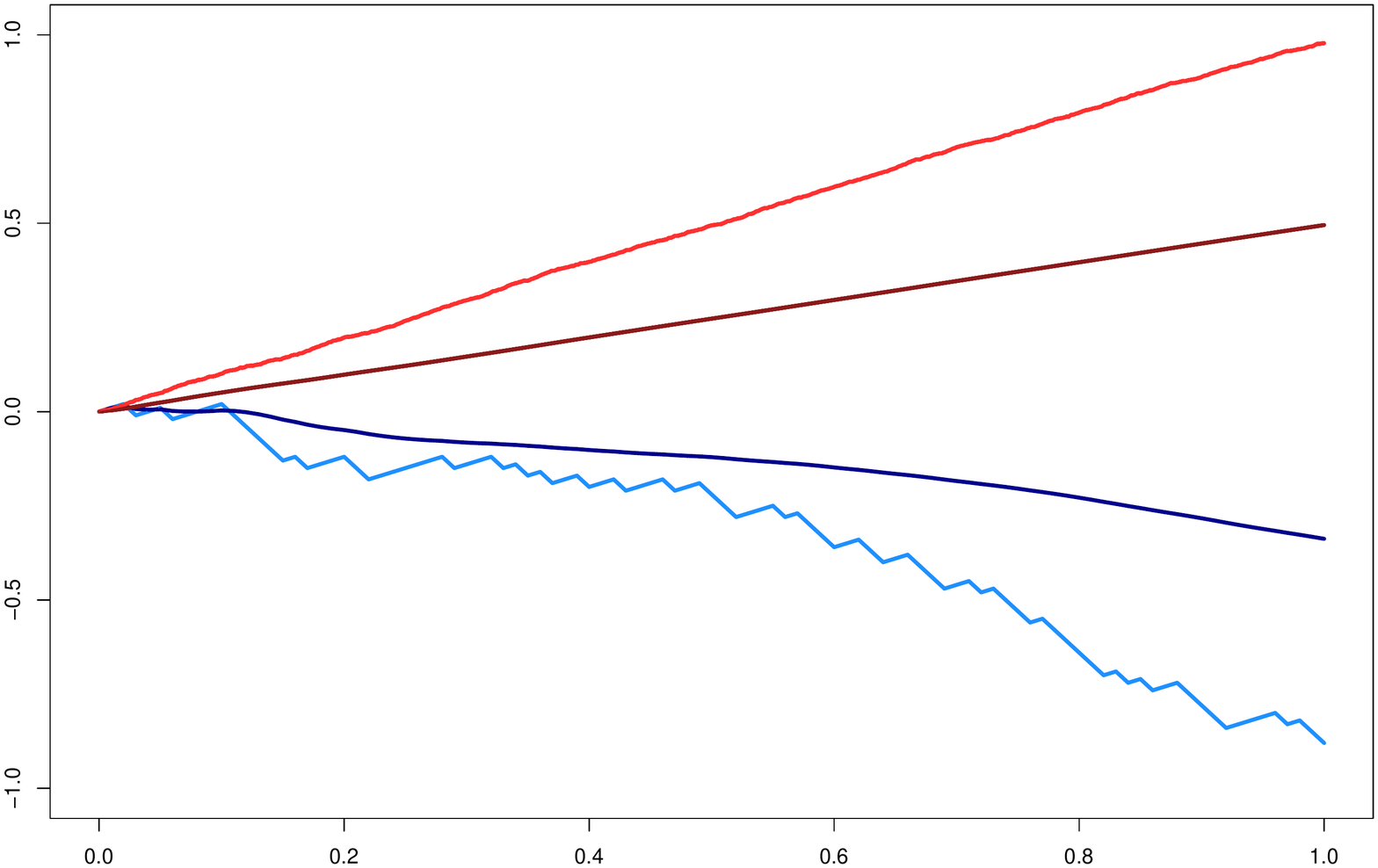}
		\caption{Sample paths of $X_n(t)$, the lighter colour, and $G_{\lfloor nt \rfloor}/n$, the darker colour, for the cases $n=100$ in blue with $\mu=-1$ and $n=10000$ in red with $\mu=1$.}
		\label{sim:comlln}
		\includegraphics[width=245pt]{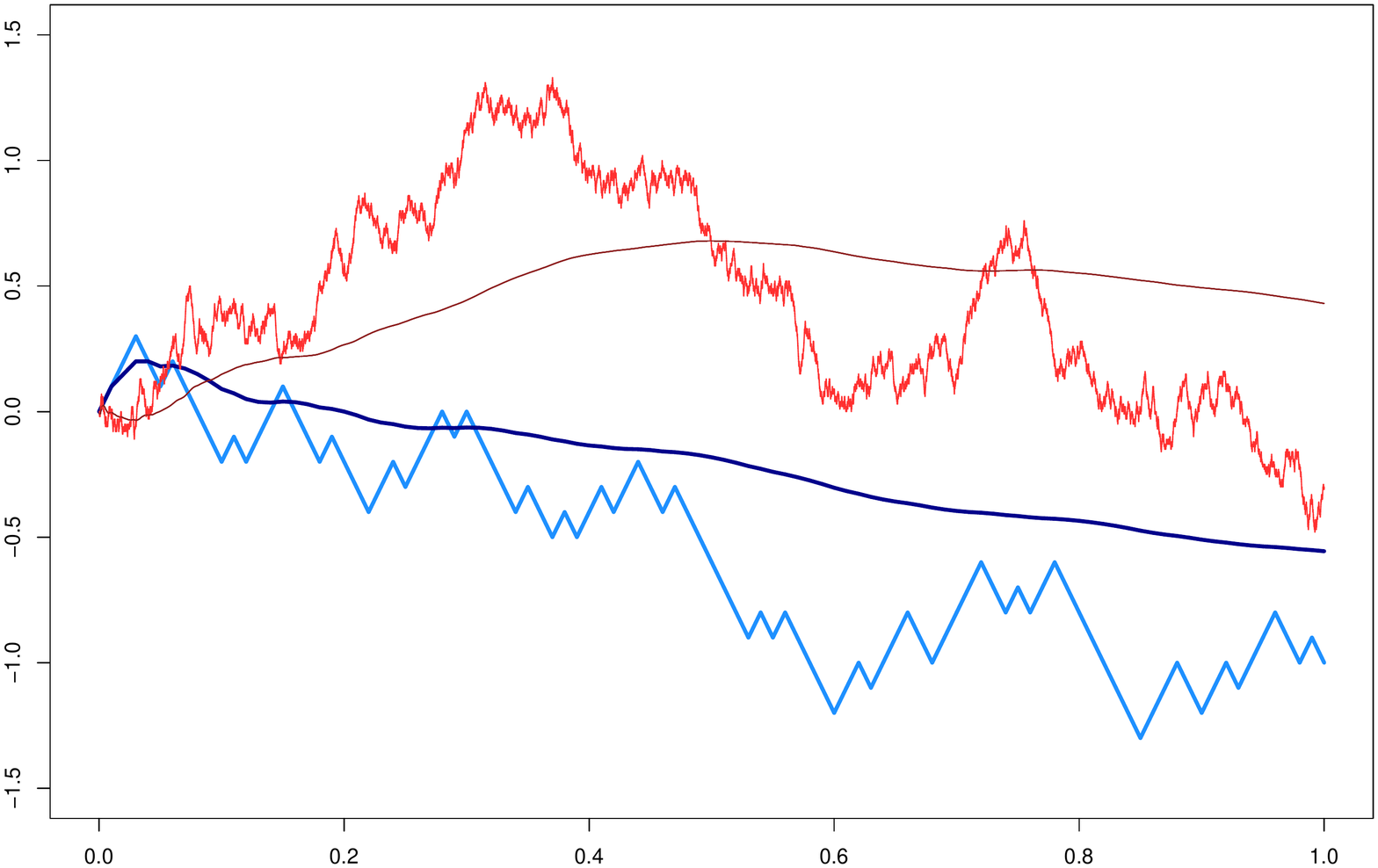}
		\caption{Sample paths of $Y_n(t)$, the lighter colour, and $G_{\lfloor nt \rfloor}/\sqrt{n}$, the darker colour, for the cases $n=100$ in blue and $n=10000$ in red, both with $\mu=0$.}
		\label{sim:comclt}
	\end{multicols}
\end{figure}

Throughout this section it is more convenient to work with the metric $\rho_S^\circ$ instead of $\rho_S$, thus, the continuity lemmas, Lemma~\ref{coml1} and Lemma~\ref{coml2}, are stated in terms of $\rho_S^\circ$. However, in light of Proposition~\ref{prop:equivmetrics} and Remark~\ref{rem:wcequivalence}, we state the results in the theorems as convergence under the standard metric $\rho_S$ despite the proofs using $\rho_S^\circ$.

\subsection{Law of large numbers and central limit theorem} 

The aim of this section is to establish the following limit theorems. The first result is a law of large numbers. A direct proof of the case $t=1$, using the strong law for the random walk, Theorem~\ref{thm:SLLN}, is given in \cite{lowade}.
In Section~\ref{sec:comflt} we extend these results to functional limit theorems.

\begin{theorem}
	\label{comt1}
	Consider the random walk defined at~\eqref{ass:walk}. Let $t \in [0,1]$. Then,
	\[
	\frac{1}{n} G_{\lfloor n t \rfloor} 	\toas \frac{\mu t }{2}, \text{ as } n \to \infty.
	\]
\end{theorem}

The next result is a central limit theorem.

\begin{theorem}
	\label{comt2}
	Suppose that we have a random walk as defined at \eqref{ass:walk} with $\mu =0$ and satisfying \eqref{ass:Sigma}. Let $t \in [0,1]$. Then, 
	\[
	\frac{1}{\sqrt{n}}G_{\lfloor nt \rfloor} \tod \calN \left(0,\frac{t\Sigma}{3}  \right) , \text{ as } n \to \infty.
	\]
\end{theorem}

Again, an alternative proof using the representation 
\begin{equation}
\label{eq:weighted-sum}
G_n = \sum_{i=1}^n \left( \frac{n-i+1}{n} \right) \xi_i
\end{equation}
and the central limit theorem for triangular arrays, is given in \cite{lowade}.

\begin{remark}
	The central limit theorem for $G_n$ is similar to the central limit theorem for $S_n$, Theorem \ref{thm:CLT}, but with a factor of $1/3$ in the variance; but the recurrence and transience behaviour is very different, see \cite{lowade}.
\end{remark}

The method of proof for these two theorems is to view the centre of mass as an appropriate functional from $(\calD^d, \rho_S^\circ)$ to $(\R^d, \rho_E)$
and then apply our functional limit theorem results. 
First, for all $f \in \calM^d_0$, we define for $t \in [0,1]$ a functional 
$g_t: \calM^d_0 \to \R^d$ given by $g_0(f) :=f( 0 ) = 0$ and
\[ 
g_t(f) :=\frac{1}{t}\int_0^t f(s) \ud s , \text{ for } t \in (0,1] . \]
Note that for any $c \in \R$, the function $g_t$ is \emph{homogeneous} in the sense that
$g_t (c f) = c g_t (f)$ for all $t \in [0,1]$. 

The slight complication is that $g_t (X_n')$, for example, is not exactly equal to $n^{-1} G_{\lfloor nt \rfloor}$.
To deal with this, we will need the following estimate both here and  when we consider functional limit theorems, below.
We write $S_{\lfloor n \cdot \rfloor}$ to represent the function $t \mapsto S_{\lfloor n t \rfloor}$ over $t \in [0,1]$.

\begin{lemma}
	\label{coml5}
	For $n \in \N$ and $t \in [0,1]$, define $\Delta_n (t) := g_t ( S_{\lfloor n \cdot \rfloor} ) - G_{\lfloor n t \rfloor}$.
	Then  for any $\alpha >0$, we have $n^{-\alpha} \| \Delta_n \|_\infty \toas 0$ as $n \to \infty$.
\end{lemma}
\begin{proof}
	First note that $\Delta_n (0) = S_0 - \frac{1}{n} G_0 = 0$ since $S_0 = G_0 = 0$.
	Now for $t>0$, we have
	\begin{align*}
	g_t( S_{\lfloor n \cdot \rfloor} ) &= \frac{1}{t}\int_0^t S_{\lfloor ns \rfloor} \ud s \nonumber \\
	&= \frac{1}{t} \left( \sum_{k=0}^{\lfloor nt \rfloor -1} \int_{k/n}^{(k+1)/n} S_{\lfloor ns \rfloor} \ud s + \int_{\lfloor nt \rfloor/n}^t S_{\lfloor ns \rfloor} \ud s \right) \nonumber \\
	&= \frac{1}{nt} \left[ \sum_{k=0}^{\lfloor nt \rfloor -1}S_k + (nt- \lfloor nt \rfloor )S_{\lfloor nt \rfloor}  \right]  \nonumber \\
	&=  \frac{1}{nt} G_{\lfloor nt \rfloor} - \frac{1}{nt} S_{\lfloor nt \rfloor} + \frac{nt-\lfloor nt \rfloor }{nt}  S_{\lfloor nt \rfloor} \nonumber \\
	&=  G_{\lfloor nt \rfloor} -  \frac{nt- \lfloor nt \rfloor }{nt} G_{\lfloor nt \rfloor} + \frac{nt-\lfloor nt \rfloor -1 }{nt}  S_{\lfloor nt \rfloor}.
	\end{align*}
	Hence for $t >0$ we have
	\begin{equation*}
	\Delta_n(t)  = \left( \frac{nt-\lfloor nt \rfloor -1 }{nt}  S_{\lfloor nt \rfloor} \right) - \left( \frac{nt- \lfloor nt \rfloor }{nt} G_{\lfloor nt \rfloor} \right).  
	\end{equation*}
	Now, since  $-1 \leq nt - \lfloor nt \rfloor -1  \leq 0$ and $S_0 = 0$ we have 
	\begin{align*}
	\sup_{0 < t \le 1} \left\| \frac{nt-\lfloor nt \rfloor -1 }{nt}   S_{\lfloor nt \rfloor} \right\| 
	& \leq \max_{1 \leq k \leq n-1} \sup_{t \in \left[ \frac{k}{n} , \frac{k+1}{n} \right]} \frac{1}{nt} \| S_{\lfloor n t \rfloor} \| \\
	& \leq \sup_{1 \leq k \leq n} \frac{1}{k} \| S_k \| .\end{align*}
	Similarly,
	\[ \sup_{0 < t \le 1} \left\| \frac{nt-\lfloor nt \rfloor }{nt}   G_{\lfloor nt \rfloor} \right\|  \leq \sup_{1 \leq k \leq n} \frac{1}{k} \| G_k \| . \]
	
	The strong law of large numbers for $S_n$, Theorem~\ref{thm:SLLN}, implies that $\| S_n \| \leq n \left(1 + \| \mu \| \right)$ for all $n \geq N$ with $\Pr (N < \infty) =1$.
	Moreover,
	\[ \| G_n \| \leq \frac{1}{n} \sum_{i=1}^n \| S_i \| \leq \frac{1}{n} \sum_{i=1}^N \| S_i \| + 
	\frac{1}{n} \sum_{i=N}^n i \left(1 + \| \mu \| \right) \leq \frac{1}{n} \sum_{i=1}^N \| S_i \| + 
	n \left(1 + \| \mu \| \right)  . \] 
	Hence
	\[  \limsup_{n \to \infty} n^{-\alpha} \sup_{0 < t \leq 1}  \left\| \frac{nt-\lfloor nt \rfloor -1 }{nt}  S_{\lfloor nt \rfloor} \right\|
	\leq \limsup_{n \to \infty}  n^{-\alpha} \left( \sup_{1 \leq k \leq N} \| S_k \| + \left(1 + \| \mu \| \right)  \right) = 0 , \as \]
	and
	\[ \limsup_{n \to \infty} n^{-\alpha} \sup_{0 < t \leq 1}  \left\|  \frac{nt- \lfloor nt \rfloor }{nt} G_{\lfloor nt \rfloor}  \right\|
	\leq \limsup_{n \to \infty} n^{-\alpha} \left( \frac{1}{n^2} \sum_{i=1}^N \| S_i \|  + \left(1 + \| \mu \| \right)  \right)   = 0 , \as \]
	This completes the proof.
\end{proof}

To prove Theorems~\ref{comt1} and~\ref{comt2}, we will need to show that the functional $g_t$ is continuous.
This is the content of the next result.

\begin{lemma}
	\label{coml1}
	For any $t \in [0,1]$, the functional $f \mapsto g_t(f)$ is continuous as a map from $(\calD^d_0, \rho_S^\circ)$ to $(\R^d, \rho_E)$.
\end{lemma}

\begin{proof}
	Consider $f_1, f_2 \in \calD^d_0$.
	For $t=0$ we have $\rho_E ( g_0 (f_1) , g_0 (f_2 ) ) = \rho_E (0,0) = 0 \leq \rho_S^\circ ( f_1, f_2)$. Thus fix $t \in (0,1]$ and let $\lambda \in \Lambda$. Then, by the triangle inequality,
	\begin{align}
	\rho_E (g_t(f_1),g_t(f_2)) & = \left\| \frac{1}{t} \int_0^t\left[f_1(s)-f_2(s)\right] \ud s\right\| \nonumber\\
	&\leq \left\|\frac{1}{t}\int_0^t\left[f_1(s)-f_2\circ \lambda (s)\right] \ud s\right\| + \left\|\frac{1}{t}\int_0^t\left[f_2 \circ \lambda (s)-f_2(s)\right] \ud s\right\| \nonumber\\
	& \leq \left\| f_1 - f_2 \circ \lambda \right\|_\infty + \left\|\frac{1}{t}\int_0^t\left[f_2 \left( \lambda (s)\right) -f_2(s)\right] \ud s\right\| . \label{eqn:comEbound}
	\end{align}
	
	Then for the second term we have,
	\begin{align*}
	\left\| \frac{1}{t} \int_0^t [ f_2 (\lambda (s)) - f_2 (s) ] \ud s \right\| & 
	\leq \left\| \frac{1}{t} \int_0^t  f_2 (\lambda (s)) \ud s - \frac{1}{t} \int_0^{\lambda(t)} f_2 (s)  \ud s \right\|
	+  \int_{t \wedge \lambda(t)}^{t \vee \lambda(t)} \| f_2 (s) \| \ud s  \\
	& \leq 
	\left\| \frac{1}{t} \int_0^t  f_2 (\lambda (s)) \ud s - \frac{1}{\lambda(t)} \int_0^{\lambda(t)} f_2 (s)  \ud s \right\| \\
	& {} \qquad {}
	+ \left| \frac{1}{t} - \frac{1}{\lambda(t)} \right| \int_0^{\lambda(t)} \| f_2 (s) \| \ud s + | t - \lambda(t) | \| f_2 \|_\infty .
	\end{align*}
	Now in the second integral on the right-hand side put $s = \lambda(u)$ so that $\ud s = \lambda'(u) \ud u$
	for almost every $u \in (0,1)$, and then
	\begin{align}
	\left\| \frac{1}{t} \int_0^t [ f_2 (\lambda (s)) - f_2 (s) ] \ud s \right\| 
	& \leq
	\left\| \frac{1}{t} \int_0^t  f_2 (\lambda (s)) \ud s - \frac{1}{t} \int_0^{t} f_2 ( \lambda(u)) \lambda'(u)  \ud u \right\| \nonumber\\
	& {} \qquad {}
	+ \frac{| \lambda(t) - t |}{t} \| f_2 \|_\infty + \| f_2 \|_\infty c (\lambda) \nonumber\\
	& \leq  \frac{1}{t}  \int_0^t \| f_2 (\lambda (s) ) \| | 1 - \lambda'(s) | \ud s + \frac{| \lambda(t) - t |}{t} \| f_2 \|_\infty + \| f_2 \|_\infty c (\lambda) \nonumber\\
	& \leq 3 \| f_2 \|_\infty c(\lambda) ,\label{eqn:combound2}
	\end{align}
	using Lemma~\ref{lem:estlam} several times.
	Then combining \eqref{eqn:comEbound} and \eqref{eqn:combound2}, we get
	\begin{align*}
	\rho_E (g_t(f_1),g_t(f_2)) &\le \left\| f_1  - f_2 \circ \lambda   \right\|_\infty + 3  
	\|f_2\|_\infty c  (   \lambda ) .
	\end{align*}
	Let $\eps >0$. Recall that $\rho_S^\circ(f,g) \coloneqq \inf_{\lambda \in \Lambda} \left\lbrace \left\| \lambda \right\|^\circ \vee \left\| f - g\circ \lambda\right\|_{\infty} \right\rbrace$.
	Now $c (\lambda ) \to 0$ as $\| \lambda \|^\circ \to 0$, so if $\rho_S^\circ (f_1, f_2) \to 0$
	we can find $\lambda \in \Lambda$ such that $\| f_1 - f_2 \circ \lambda \|_\infty < \eps$
	and $\| \lambda \|^\circ$ is small enough so that $\| f_2 \|_\infty c (\lambda ) < \eps$ too (note $\| f_2 \|_\infty < \infty$).
	Then $\rho_E (g_t(f_1),g_t(f_2)) \leq 4 \eps$, and since $\eps>0$ was arbitrary, the result follows.
\end{proof}
\noindent Now we are ready to prove our Theorem \ref{comt1}.
\begin{proof}[Proof of Theorem \ref{comt1}]
First, we know that $X'_n$ converges to $I_\mu$ a.s.~where $I_\mu(t)=\mu t$, as defined in Theorem~\ref{thm:flln}. With Lemma \ref{coml1}, we can apply the continuous mapping theorem, Theorem~\ref{thm:mapping}, for the function $g_t$; thus we get, for $t \in [0,1]$, 
	\begin{equation}
	g_t(X'_n) \to g_t(I_\mu), \quad \text{a.s.} \label{cal10}
	\end{equation}
	With a quick calculation we see that, for $t>0$,
	\begin{equation}
	g_t(I_\mu) = \frac{1}{t}\int_0^t \mu s \ud s = \frac{\mu t}{2} = \frac{1}{2} I_\mu(t),  \label{cal1}
	\end{equation}
	and $g_0(I_\mu)=I_\mu (0)=0$. Now we see that
	\begin{equation*}
	\frac{1}{n} G_{\lfloor nt \rfloor}  = \frac{1}{n} g_t \left( S_{\lfloor n \cdot  \rfloor} \right) -\frac{1}{n} \Delta_n(t) = g_t(X_n') - \frac{1}{n} \Delta_n(t).
	\end{equation*}
	Now using the convergence in \eqref{cal10} and equation \eqref{cal1}, together with the $\alpha=1$ case of Lemma~\ref{coml5}, we complete the proof.
\end{proof}
\noindent The next theorem to prove is the central limit theorem. Recall that we will have an extra assumption that $\mu = 0$ for a meaningful analysis. This time we use the same function $g_t(f)$, but we take 
\begin{equation*}
Y'_n := \frac{1}{\sqrt{n}}S_{\lfloor nt \rfloor}
\end{equation*}
instead of $X'_n$ to get the right meaningful scaling. We shall rewrite $g_t(b_d)$ in the following stochastic integral form to simplify later calculations.
\begin{lemma}
	\label{sitrick}
For any $t \ge 0$,
	\begin{equation*}
	g_t(b_d) = \int_0^t \left( 1 - \frac{s}{t} \right) \ud b_d(s).
	\end{equation*}
\end{lemma}
\begin{remark}
	This is the continuous analogue of equation~\eqref{eq:weighted-sum}.
\end{remark}
\begin{proof}
	We apply the integration by part formula for stochastic calculus, see e.g. \cite[p.~129]{medvegyev}, and get
	\begin{align*}
	\int_0^t b_d(s) \ud s + \int_0^t s \ud b_d(s) = t b_d(t)  = \int_0^t t \ud b_d(s).
	\end{align*}
	Dividing by $t>0$ and rearranging, we get
	\begin{align*}
	\frac{1}{t}\int_0^t b_d(s) \ud s = \int_0^t \left(1- \frac{s}{t}\right) \ud b_d(s).
	\end{align*}
	Hence we obtain the statement we needed from the definition of $g_t$.
\end{proof}
Now we are ready to prove Theorem \ref{comt2}.
\begin{proof}[Proof of Theorem \ref{comt2}]
Recall that Donsker's theorem implies that $Y'_n \Rightarrow \Sigma^{1/2}b_d$ on $(\mathcal{D}^d_0, \rho_S^\circ)$, where $\Sigma$ is the covariance matrix of $\xi$ and $b_d$ is the standard Brownian motion in $d$ dimensions. Using Lemma~\ref{coml1} and the continuous mapping theorem, Theorem~\ref{thm:mapweak}, we get, for $t \in [0,1]$,
	\begin{equation}
	g_t(Y'_n) \tod g_t\left( \Sigma^{1/2}b_d \right). \label{gc1}
	\end{equation}
	Using Lemma~\ref{sitrick}, we see that 
	\begin{align*}
	g_t\left(\Sigma^{1/2}b_d\right) = \Sigma^{1/2} \int_0^t \left(1- \frac{s}{t}\right) \ud b_d(s), 
	\end{align*}
	and as the integrand is deterministic, the integral is a Wiener integral and so it is normally distributed, see \cite[p.~11]{kuo}, and hence $g_t\left( \Sigma^{1/2}b_d \right)$ is also normal.
	So all we left to do now is to find the variance of $g_t\left( \Sigma^{1/2}b_d \right)$. We should first consider the expression
	\begin{equation*}
	B(t) = \int_0^t b_d(s) \ud s
	\end{equation*}
	We calculate that
	\begin{align}
	\Var (B(t))= \Exp \left[B(t)B(t)^\tra \right] &= \Exp \left[\int_0^t b_d(r)\ud r \times \int_0^t b_d^\tra (s) \ud s \right]  \nonumber  \\
	&= \Exp \left[\int_0^t \int_0^t b_d(r)b_d^\tra (s) \ud r\ud s \right] \nonumber \\
	&= \int_0^t \int_0^t \Exp \left[b_d(r)b_d^\tra (s)\right] \ud r\ud s. \label{b1}
	\end{align}
	The change of order of the expectation and the integration in the last step is guaranteed by Fubini's theorem as the function is integrable. To evaluate the expectation, we first consider the case that $r>s$, then we can write $b_d(r)=b_d(s)+(b_d(r)-b_d(s))$, then as $b_d(s)$ is independent of $b_d(r)-b_d(s)$, we get 
	\begin{equation*}
	\Exp \left[b_d(r)b_d^\tra(s)\right] = \Exp \left[(b_d(r)-b_d(s))b_d^\tra(s)\right] + \Exp \left[b_d(s) b_d^\tra(s) \right] = \Var(b_d(s))= s I_d.
	\end{equation*}
	where $I_d$ is the $d$-dimensional identity matrix. Similarly, for the case $r<s$, we get $\Exp \left[b_d(r)b_d^\tra(s)\right]=r I_d$. Combining the two cases, we get $\Exp \left[b_d(r)b_d^\tra(s)\right]=\min(r,s)I_d$. Putting this back to equation \eqref{b1}, we have for $t \ge 0$,
	\begin{equation*}
	\Var (B(t)) = \int_0^t \int_0^t \min(r,s) I_d \ud r\ud s = \frac{t^3}{3} I_d.
	\end{equation*}
	The integral is essentially just finding the volume of a pyramid with a square base of side length $t$, with height $t$, attained at the point $(t,t)$. So we get $B(t) \sim \mathcal{N}(0,t^3 I_d/3)$. Hence for $t>0$,
	\begin{equation*}
	g_t\left(\Sigma^{\frac{1}{2}}b_d\right) = \frac{\Sigma^{\frac{1}{2}}}{t} B(t) \sim \mathcal{N}\left(0,\frac{t\Sigma}{3}\right) 
	\end{equation*}
	Together with the convergence \eqref{gc1}, for all $t>0$, we get
	\begin{equation*}
	g_t(Y'_n) = \frac{1}{\sqrt{n}} \left( G_{\lfloor nt \rfloor} + \Delta_n(t) \right) \tod \mathcal{N}\left(0,\frac{t\Sigma}{3}\right) \mathrm{\ as\ } n \to \infty.
	\end{equation*}
	With an implication of Lemma \ref{coml5} that $\Delta_n (t)/ \sqrt{n} \to 0$ a.s.~as $n \to \infty$, by Slutsky's theorem, Theorem~\ref{thm:slutsky}, we obtain, for all $t>0$,
	\begin{equation}
	\frac{1}{\sqrt{n}} G_{\lfloor nt \rfloor} \tod \mathcal{N}\left(0,\frac{t\Sigma}{3}\right) \mathrm{\ as\ } n \to \infty, \label{tfix2}
	\end{equation}
	which is also true for $t=0$. Notice that at $t=0$, the normal distribution $\mathcal{N}(0,0) \equiv 0$. So equation \eqref{tfix2} is true for all $t \in [0,1]$. Hence the proof is completed.
\end{proof}

\subsection{Functional limit theorems}
\label{sec:comflt}

The aim of this section is to extend the law of large numbers and central limit theorem
to the whole trajectory of the centre of mass process.
Here are the results. Recall that $I_\mu (t) = \mu t$.

\begin{theorem}
	\label{comt3}
	Consider the random walk defined at~\eqref{ass:walk}. Then, as $n \to \infty$, as elements of $(\calD^d_0, \rho_S)$, 
	\[
	\frac{1}{n}\left(G_{\lfloor nt \rfloor}\right)_{t \in [0,1]} \toas   \frac{1}{2} I_\mu .
	\]
\end{theorem}

\begin{theorem}
	\label{comt4}
	Suppose that we have a random walk as defined at \eqref{ass:walk} with $\mu =0$ and satisfying \eqref{ass:Sigma}. Then, as $n \to \infty$,
	as elements of $(\calD^d_0, \rho_S)$, 
	\[
	\frac{1}{\sqrt{n}}\left(G_{\lfloor nt \rfloor}\right)_{t \in [0,1]} \Rightarrow \mathcal{GP}(0,K),
	\]
	where $\mathcal{GP}(0,K)$ is a Gaussian process with mean $0$ and the symmetric covariance function $K$ defined by
	\[
	K(t_1,t_2) = 
	\begin{cases}
	t_1 \Sigma (3t_2 -t_1)/( 6 t_2), & \text{for } 0 < t_1 \le t_2; \\
	t_2 \Sigma /3, & \text{for } t_1=0, t_2 \ne 0; \\
	0 , & \text{for } t_1=t_2=0.
	\end{cases}
	\]
\end{theorem}

In order to prove Theorem \ref{comt3} and Theorem \ref{comt4}, we need to introduce a bigger functional which contains all the information in the trajectory. 
Define the functional $g$ acting on $f \in \calM^d_0$ by $g(f) (t) = g_t (f)$ for all $t \in [0,1]$.
Note that if $f \in \calM^d_0$ then $g(f)(0) = g_0 (f) = f(0) =0$.
For any fixed $f \in \calD^d_0$,
observe that $g_t(f)$ is continuous in $t$. The latter is true because $t \mapsto \frac{1}{t}$ and $t \mapsto \int_0^t f(s) \ud s$ are both continuous on $(0,1]$, 
so $t \mapsto g_t$ is continuous on $(0,1]$. To check the continuity at $t=0$, we use l'H\^opital's rule to see that 
\begin{equation*}
\lim_{t \to 0^+}\left(\frac{\int_0^t f(s) \ud s}{t}\right) = \lim_{t \to 0^+}\left(\frac{\frac{\ud}{\ud t}\int_0^t f(s) \ud s}{1}\right) = \lim_{t \to 0^+}f(t)=f(0),
\end{equation*}
using the fact that $f$ is right-continuous at $0$.
So we conclude that if $f \in \calD^d_0$, then $g(f) \in \calC^d_0$.
The next result shows that $f \mapsto g(f)$ is continuous.

\begin{lemma}
\label{coml2}
The functional $f \mapsto g(f)$ is continuous as a map from $(\calD^d_0, \rho_S^\circ)$ to $(\calC^d_0, \rho_S^\circ)$.
\end{lemma}
\begin{proof}
Take $f_1, f_2 \in \calD^d_0$. For $f_1, f_2 \in \calD^d_0$,
	\begin{align*}
	\rho_S^\circ (g (f_1) , g(f_2) ) & = \inf_{\lambda \in \Lambda} \{ \| \lambda \|^\circ \vee \| g (f_1) , g(f_2) \circ \lambda \|_\infty \} ,\end{align*}
	where
	\[ \| g (f_1) , g(f_2) \circ \lambda \|_\infty = \sup_{0\leq t \leq 1} \left\| g_t (f_1) - g_{\lambda(t)} (f_2 ) \right\|  .\]
	Now since $\lambda (0) =0$ we have  $g_0 (f_1) - g_{\lambda(0)} (f_2) = 0$.
	For $t \in (0,1]$, we have
	\begin{align*}
	g_t (f_1) - g_{\lambda(t)} (f_2 )  & =  \frac{1}{t} \int_0^t f_1(s) \ud s - \frac{1}{\lambda(t)} \int_0^{\lambda(t)} f_2(s) \ud s 
	\\ 
	&= \frac{1}{t} \int_0^t f_1(s) \ud s - \frac{1}{t} \int_0^{\lambda(t)} f_2(s) \ud s - \left( \frac{1}{\lambda(t)}- \frac{1}{t} \right) \int_0^{\lambda(t)} f_2(s) \ud s \\
	&= \frac{1}{t} \int_0^t f_1(s) \ud s - \frac{1}{t} \int_0^{t} f_2(\lambda(u)) \lambda'(u) \ud u - \left( \frac{t- \lambda(t)}{t \lambda(t)} \right) \int_0^{\lambda(t)} f_2(s)\ud s \\
	&= \frac{1}{t} \int_0^t ( f_1(s) - f_2 \circ \lambda (s) ) \ud s  - \left( \frac{t- \lambda(t)}{t \lambda(t)} \right) \int_0^{\lambda(t)} f_2(s) \ud s - \frac{1}{t}\int_0^t f_2(\lambda(u))(\lambda'(u) -1 ) \ud u.
	\end{align*}
	It follows that, for $t \in (0,1]$,
	\begin{align*}
	\left\| g_t (f_1) - g_{\lambda(t)} (f_2 ) \right\| 
	&\le \left\| f_1  - f_2 \circ \lambda   \right\|_\infty + \frac{|t- \lambda(t)|}{t}  \|f_2\|_\infty   +  \|f_2\|_\infty    \| \lambda' - 1 \|_\infty  .\end{align*}
	Applying the estimates~\eqref{s1} and~\eqref{s3} in Lemma~\ref{lem:estlam}, we get
	\begin{align*}
 	\inf_{\lambda \in \Lambda} \left\| g (f_1) - g (f_2 )\circ \lambda  \right\|_\infty &\le  \inf_{\lambda \in \Lambda} \left\| f_1  - f_2 \circ \lambda   \right\|_\infty + 2   \left\| f_2 \right\|_\infty \inf_{\lambda \in \Lambda} c(\lambda).
	\end{align*}
	However, $\frac{c(\lambda)}{\|\lambda\|^\circ} \to 1$ as $\|\lambda\|^\circ \to 0$, so for any $\varepsilon >0$, we can find $\delta>0$ such that $\rho_S^\circ(f_1, f_2) \le \delta$, implies that $\inf_{\lambda \in \Lambda} \left\| f_1  - f_2 \circ \lambda   \right\|_\infty \le \varepsilon$ and $\inf_{\lambda \in \Lambda} c(\lambda) \le \varepsilon$. Finally noting that $\|\lambda\|^\circ \leq \rho_S^\circ(f_1,f_2)$, we have
	\begin{align*}
	\rho_S^\circ \left(g(f_1),g(f_2)\right) \le \rho_S^\circ(f_1,f_2) \vee (\varepsilon + 2   \left\| f_2 \right\|_\infty \varepsilon).
	\end{align*}
	Since $\eps>0$ was arbitrary, the result follows.
\end{proof}

\noindent Now we are ready to proof Theorem \ref{comt3} and Theorem \ref{comt4}.
\begin{proof}[Proof of Theorem \ref{comt3}]
	With the fact that $\frac{1}{n} g_t \left(S_{\lfloor n \cdot \rfloor}\right) = g_t \left(\frac{1}{n} S_{\lfloor n \cdot \rfloor} \right)$, by the functional law of large numbers Theorem~\ref{thm:flln}(b), Lemma~\ref{coml2} and the mapping theorem for almost-sure convergence Theorem~\ref{thm:mapping}, we get  
	\begin{equation*}
	\left(\frac{1}{n} g_t \left( S_{\lfloor n \cdot \rfloor}\right)\right)_{t \in [0,1]} \toas \left( g_t(I_\mu)\right)_{t \in [0,1]} \quad \text{on } (\calD^d_0,\rho_S). 
	\end{equation*}
	By Lemma~\ref{coml5} we also have
	\begin{align*}
	\rho_S^\circ \left(\left( \frac{1}{n} G_{\lfloor n t \rfloor} \right)_{t \in [0,1]} , \left(\frac{1}{n} g_t \left( S_{\lfloor n \cdot \rfloor}\right)\right)_{t \in [0,1]}\right) &\le \left\| \left( \frac{1}{n} G_{\lfloor n t \rfloor} \right)_{t \in [0,1]} - \left(\frac{1}{n} g_t \left( S_{\lfloor n \cdot \rfloor}\right)\right)_{t \in [0,1]} \right\|_\infty \\
	&= \frac{1}{n} \left\|\Delta_n \right\|_\infty \to 0 \quad \text{a.s.}
	\end{align*}
	Hence by equation~\eqref{cal1}, we have
	\begin{equation*}
	\left( \frac{1}{n} G_{\lfloor n t \rfloor} \right)_{t \in [0,1]} \toas \frac{1}{2} I_\mu \quad \text{on }(\calD^d_0,\rho_S). \qedhere
	\end{equation*}
\end{proof}

\begin{proof}[Proof of Theorem \ref{comt4}]
	Similarly to the proof of Theorem \ref{comt2}, applying Donsker's theorem and the continuous mapping theorem to $g(f)$, with the continuity of $g(f)$ given by Lemma \ref{coml2}, we get 
	\begin{equation}
	\label{convwt}
	g(Y'_n) \Rightarrow g(\Sigma^{\frac{1}{2}}b_d) \quad \text{on } (\calD^d_0,\rho_S).
	\end{equation}
	The next step is to prove $g(\Sigma^{\frac{1}{2}}b_d)$ is a non-stationary Gaussian process, i.e. every finite collection of $g_t(\Sigma^{\frac{1}{2}}b_d)$ has a multivariate normal distribution, see \cite[p.~4]{piterbarg}. We will use a definition of multivariate normal distribution, e.g. see \cite[p.~121]{gutintermediate}, that $X \in \R^m$ is multivariate normal if and only if $u^\tra X$ is normal for all $u \in \mathbb{S}^{m-1}$. We will use this fact with $m=dk$, and without loss of generality, assume $t_1 \le t_2 \le \cdots \le t_k$. Using Lemma~\ref{sitrick}, consider
	\begin{align*}
	\sum_{l=1}^k \alpha_l g_{t_l} 
	&= \sum_{l=1}^k \int_0^{t_l} \alpha_l \left( 1-\frac{s}{{t_l}} \right) \ud b_s \\
	&= \int_0^{\max\{t_1, t_2, \cdots, t_k\}} f(s) \ud b_s
	\end{align*}
	where
	\begin{align*}
	f(s) &= \begin{cases}
	\sum_{l_1=1}^k  \alpha_{l_1} \left( 1-\frac{s}{{t_{l_1}}} \right), & \text{if } s \leq t_1; \\
	\sum_{l_2=2}^k  \alpha_{l_2} \left( 1-\frac{s}{{t_{l_2}}} \right), & \text{if } t_1 \leq s \leq t_2; \\
	\quad \vdots \quad    \\
	\alpha_n \left( 1-\frac{s}{{t_k}} \right), & \text{if } t_{k-1} \leq s \leq t_k; \\
	0, & \text{otherwise. }
	\end{cases}
	\end{align*}

	Now as $f(s)$ is piecewise continuous, the whole integral is a Wiener integral which is normal, by \cite[p.~11]{kuo}. Hence $g(\Sigma^{\frac{1}{2}}b_d)$ is a Gaussian process. So now all left to do now is to find the covariance functions, which completely categorize the Gaussian process. 

	Denote this non-stationary Gaussian process by $\mathcal{GP}(0,K(t_1,t_2))$, where $K(t_1,t_2)$ is the covariance function of $g(\Sigma^{\frac{1}{2}}b_d)$ evaluated at the points $t_1$ and $t_2$. We have mean $0$ is because we have a zero mean at any fix $t$. Now to calculate  $K$, for non-zero $t_1$ and $t_2$, without loss of generality, suppose $0< t_1 \le t_2$, then
	\begin{align*}
	K(t_1,t_2) &= 
	\Cov \left(g_{t_1} \left(\Sigma^\frac{1}{2}b_d \right), g_{t_2} \left(\Sigma ^\frac{1}{2}b_d \right)\right) \\
	&= \Exp \left[ \left(g_{t_1}\left(\Sigma^{\frac{1}{2}}b_d \right)\right) \left(g_{t_2}\left(\Sigma ^{\frac{1}{2}} b_d\right)\right)^\tra \right] \\
	&= \Exp\left[\left(\frac{1}{t_1} \int_0^{t_1}\Sigma^{\frac{1}{2}} b_d(r) \ud r\right) \left(\frac{1}{t_2} \int_0^{t_2}\Sigma^{\frac{1}{2}} b_d(s) \ud s\right)^\tra \right] 
	\end{align*}
	Now we apply the Fubini's theorem to swap the integral and expectation to get
	\begin{align}
	\frac{1}{t_1 t_2} \int_0^{t_1} \int_0^{t_2} \Sigma^{\frac{1}{2}} \Exp \left[b_d(r)b_d^\tra(s) \right] \Sigma^{\frac{1}{2}} \ud r \ud s 
	&= \frac{\Sigma}{t_1 t_2} \int_0^{t_1} \int_0^{t_2} \min(r,s) \ud r \ud s \nonumber \\
	&= \frac{\Sigma}{t_1 t_2} \left[ \frac{t_1^3}{3} + \frac{t_1^2}{2}(t_2-t_1) \right] \nonumber \\
	&= \frac{t_1 (3t_2 -t_1) }{6 t_2}\Sigma .
	\label{callst}
	\end{align}
	If one of $t_1$ or $t_2$ is $0$, then we just have $K(0,t) = t\Sigma/3$ putting $t_1=0$ and $t_2=t$ in equation \eqref{callst}. If both of them are zero, then $K(0,0)=0$ by the definition that $g_0\left(\Sigma^{\frac{1}{2}}b_d\right)=0$. So we have
	\begin{equation*}
	g\left(\Sigma^{\frac{1}{2}}b_d\right) \Rightarrow \mathcal{GP}(0,K) \quad \text{on } (\calD^d_0,\rho_S).
	\end{equation*}
	Together with the convergence~\eqref{convwt}, we have
	\begin{equation*}
	g(Y'_n) = \frac{1}{\sqrt{n}}\left(G_{\lfloor nt \rfloor} + \Delta_n (t)\right)_{t \in [0,1]} \Rightarrow \mathcal{GP}(0,K).
	\end{equation*}
	Lastly, choosing $\alpha = 1/2$ in Lemma~\ref{coml5} so that $\Delta_n / \sqrt{n} \to 0$ a.s.~on $(\calD^d_0,\rho_S)$ as $n \to \infty$, we apply Slutsky's theorem, Theorem~\ref{thm:slutsky}, we get the desired result. So we have proved the last theorem of this chapter.
\end{proof}

\appendix 

\section{Extension of Etemadi's inequality}
\label{sec:Etemadi}

This is the $d$-dimensional version of the inequality of Etemadi \cite[Theorem~22.5]{billpm}.
\begin{lemma}
\label{lem:etemadi}
Let $S_n = \sum_{i=1}^n \xi_i$ be a random walk on $\R^d$. Then for any $x \geq 0$,
\[ \Pr \left( \max_{0 \leq j \leq n} \| S_j \| \geq 3 x \right) \leq 3 \max_{0 \leq j \leq n} \Pr ( \| S_j \| \geq x ) .\]
\end{lemma}
\begin{proof}
For given $x$ and fixed $n$, let 
\begin{align*}
B_k & := \left\{ \max_{0 \leq j \leq k-1}  \|   S_j \| \leq 3x \right\} \cap \left\{ \|   S_k \| \geq 3x \right\} \\
B & := \bigcup_{k=1}^n B_k  = \left\{ \max_{0 \leq k \leq n} \| S_k \| \geq 3x \right\}
\end{align*}
Then the $B_k$ are disjoint for $x>0$, and for $k \leq n$, by the triangle inequality,
\begin{align*}
B_k \cap \left\lbrace \|   S_n \| \leq x \right\rbrace \subseteq B_k \cap \left\lbrace \| S_n - S_k \| > 2x \right\rbrace,
\end{align*}
and the terms on the right hand side are independent of each other. We therefore have that,
\begin{align*}
\Pr(B) & = \Pr \left(B \cap \left\lbrace \|  S_n \| > x\right\rbrace \right) 
+ \Pr \left(B \cap \left\lbrace \|  S_n \| \leq x\right\rbrace \right) \\
& \leq \Pr\left( \|  S_n \| > x\right) 
+ \Pr\left(B \cap \left\lbrace \|   S_n \| \leq x \right\rbrace \right) \\
& = \Pr \left(\|  S_n \| > x \right) 
+ \sum_{k=1}^n \Pr \left(B_k \cap \left\lbrace \|   S_n \| \leq x \right\rbrace \right)\\
& \leq \Pr \left(\|  S_n \| > x \right) 
+ \sum_{k=1}^n \Pr \left(B_k \cap \left\lbrace \|  S_n - S_k \| > 2x \right\rbrace \right) \\
& \leq \Pr \left( \|  S_n \| > x \right)
+ \sum_{k=1}^n \Pr(B_k) \Pr \left( \| S_n - S_k \| > 2x \right) \\
& \leq \Pr \left( \|  S_n \| > x \right)
+ \max_{k\leq n} \Pr \left( \|  S_n - S_k \| > 2x \right) \\
& \leq \Pr \left( \|   S_n \| > x \right) 
+ \max_{k \leq n}\left[ \Pr \left( \|  S_n \| > x \right) + \Pr \left( \|  S_k \| > x \right) \right] \\
& \leq 3 \max_{k \leq n} \Pr \left( \|   S_k \| > x \right),
\end{align*}
as required.
\end{proof}

\bibliographystyle{plain}
\bibliography{flt}
\end{document}